\documentclass[a4paper,12pt,reqno]{amsart}
\usepackage[colorlinks=true]{hyperref}
\usepackage{amssymb}
\usepackage{pstricks}
\usepackage{graphicx}
\usepackage{mathdots}
\allowdisplaybreaks[1]

\newtheorem{theorem}{Theorem}
\newtheorem*{theorem*}{Theorem}
\newtheorem{corollary}[theorem]{Corollary}
\newcommand{\ASM}{\mathrm{ASM}}
\newcommand{\SVDWBC}{\mathrm{6VDW}}
\newcommand{\SV}{\mathrm{6V}}
\newcommand{\rT}{\rho_\mathrm{T}}
\newcommand{\rL}{\rho_\mathrm{L}}
\newcommand{\rB}{\rho_\mathrm{B}}
\newcommand{\rR}{\rho_\mathrm{R}}
\newcommand{\adj}{\mathrm{adj}}
\newcommand{\opp}{\mathrm{opp}}
\newcommand{\qua}{\mathrm{quad}}
\newcommand{\tri}{\mathrm{tri}}
\newcommand{\gen}{\mathrm{gen}}
\newcommand{\A}{\mathcal{A}}
\newcommand{\Z}{\widetilde{Z}}
\newcommand{\N}{\mathcal{N}}
\newcommand{\row}{\mathrm{row}}
\newcommand{\col}{\mathrm{col}}
\newcommand{\G}{\mathcal{G}_n}
\author[R.~E.~Behrend]{Roger E.~Behrend}
\address{R.~E.~Behrend, School of Mathematics, Cardiff University, Cardiff, CF24 4AG, UK}
\email{behrendr@cardiff.ac.uk}
\gdef\s{s}
\title[Multiply-refined enumeration of ASM{\s}]{Multiply-refined enumeration of\\
alternating sign matrices}
\begin{document}
\begin{abstract}
Four natural boundary statistics and two natural bulk statistics are considered for alternating sign matrices (ASMs).
Specifically, these statistics are the positions of the~$1$'s in the first and last rows and columns of an ASM,
and the numbers of generalized inversions and~$-1$'s in an ASM.
Previously-known and related results for the exact enumeration of ASMs with prescribed values of some of these statistics are discussed in detail.
A quadratic relation which recursively determines the generating function associated with all six statistics is then obtained.
This relation also leads to various new identities satisfied by generating functions associated with fewer than six of the statistics.
The derivation of the relation involves combining the Desnanot--Jacobi determinant identity with the
Izergin--Korepin formula for the partition function of the six-vertex model with domain-wall boundary conditions.
\end{abstract}
\maketitle
{\small\tableofcontents}
\section{Introduction}\label{intro}
A major focus of attention throughout the history of alternating sign matrices (ASMs)
has simply been the derivation of
results related to their exact enumeration.  Such results typically state that the number of ASMs which satisfy
specific conditions, such as having prescribed values of certain statistics or being invariant under certain symmetry
operations, is given by an explicit formula or generating function,
is equal to the number of combinatorial objects of some other variety satisfying specific conditions,
or is equal to a number which arises from a particular physical model.

A few examples of such results, with references to conjectures and initial proofs, are as follows:
a formula for the total number of ASMs of any fixed size but with no further conditions applied
(Mills, Robbins and Rumsey~\cite{MilRobRum82,MilRobRum83},
Zeilberger~\cite{Zei96a}, and Kuperberg~\cite{Kup96});
a formula for the number of ASMs with a prescribed boundary row or column (Mills, Robbins and
Rumsey~\cite{MilRobRum82,MilRobRum83}, and Zeilberger~\cite{Zei96b});
formulae for numbers of ASMs invariant under certain natural symmetry operations (Robbins~\cite{Rob91,Rob00}, Kuperberg~\cite{Kup02},
Okada~\cite{Oka06}, and Razumov and Stroganov~\cite{RazStr06b,RazStr06a});
equalities between numbers of certain ASMs and numbers of certain totally symmetric self-complementary plane partitions
(Mills, Robbins and Rumsey~\cite{MilRobRum86}, and Fonseca and Zinn-Justin~\cite{FonZin08});
equalities between numbers of certain ASMs and numbers of certain descending plane partitions (Mills, Robbins and
Rumsey~\cite{MilRobRum83}, and Behrend, Di Francesco and Zinn-Justin~\cite{BehDifZin12});
equalities between numbers of certain ASMs and numbers associated with a certain case of the~$O(1)$ loop model
(Razumov and Stroganov~\cite{RazStr04a}, and Cantini and Sportiello~\cite{CanSpo11}).
For reviews of some of these results, see, for example,
Bressoud~\cite{Bre99,Bre08}, Bressoud and Propp~\cite{BrePro99}, Di Francesco~\cite[Sec.~4]{Dif12a},~\cite{Dif12b}, Hone~\cite{Hon06},
Zeilberger~\cite{Zei05}, or Zinn-Justin~\cite{Zin09}.

The main result of this paper is of the type which provides a relation which determines the generating function
associated with numbers of ASMs of any fixed size with prescribed values of particular statistics.
In any ASM, the first and last rows and columns each contain a single entry of~1, with all of their other entries being~$0$'s,
so that the positions of these~1's provide four natural statistics which describe the boundary configuration of an ASM.
The generating function under primary consideration in this paper is associated with these four boundary statistics,
together with two statistics which
depend on the bulk structure of an ASM, namely the number of generalized inversions (which will be defined in~\eqref{numuA})
and the number of entries of~$-1$.  A detailed account of the elementary properties of this generating function and these statistics will be
given in Section~\ref{defprop}.

All previously-known results for the exact enumeration of ASMs have involved fewer than six of these statistics.
A review of such results will be given in Section~\ref{prev}.  In particular,
that section will aim to provide a comprehensive account of all previous appearances of these statistics in exact enumeration,
together with an outline of some new results which can be obtained relatively straightforwardly from the previously-known results.
Recently, a result involving all four of the boundary statistics, but neither of the bulk statistics,
was obtained, independently of the work reported in this paper, by Ayyer and Romik~\cite[Thm.~2]{AyyRom13}.  That result will be outlined
in Section~\ref{prev4}.

In ASM enumerations, it is common to refer to a certain order of refinement which is based on
the number of boundary statistics involved.  Hence, the primary generating functions used in this paper,
and by Ayyer and Romik~\cite{AyyRom13}, can be described as quadruply-refined.

The main result of this paper, Theorem~\ref{thm}, will be stated in Section~\ref{newres}, and proved in Section~\ref{proof}.
The result consists of a quadratic relation satisfied by the ASM generating function associated with all six statistics, and
which enables this generating function to be obtained recursively for successive ASM sizes.  Various corollaries of Theorem~\ref{thm},
consisting of new expressions or relations for ASM generating functions associated with fewer than six of the statistics,
will also be obtained in Section~\ref{newres}.

The approach used in the proof of Theorem~\ref{thm} is essentially as follows. First, certain standard techniques,
which have played a crucial role in proofs of many of the other known enumerative results for ASMs,
are used to obtain a determinantal expression related to ASM generating functions.  More specifically, a
bijection between ASMs and configurations of the
statistical mechanical six-vertex, or square ice, model with domain-wall boundary conditions (DWBC) is used to derive a relation between
ASM generating functions and the partition function of the model,
and the Izergin--Korepin formula
is used to provide a determinantal expression for this partition function.
Next, the Desnanot--Jacobi determinant identity is applied to the matrix in the Izergin--Korepin formula,
which leads to a quadratic relation involving the required quadruply-refined
ASM generating function, and generating functions associated with no
ASM boundary statistics or with two ASM boundary statistics, for all four pairs of adjacent boundaries.

The key step in this derivation is the use of the Desnanot--Jacobi identity together with the Izergin--Korepin
formula, and the reason that this leads to a relation of the previously-described form can be understood relatively easily.
The six-vertex model with DWBC, in the form used here, involves so-called spectral parameters $u_1,\ldots,u_n$ and $v_1,\ldots,v_n$,
and the Izergin--Korepin formula expresses the partition function of this model as an explicit prefactor multiplied by
the determinant of an $n\times n$ matrix, whose entry in row $i$ and column~$j$ depends only on spectral parameters~$u_i$ and~$v_j$.
Furthermore, in terms of generating functions for $n\times n$ ASMs,
the parameters~$u_i$ and~$v_j$ are, in a certain sense,
associated with row $i$ and column $j$ of the ASMs.  For an $n\times n$ matrix $M$, the Desnanot--Jacobi identity,
in the form used here, consists of a relation
involving $\det M$ and five minors of $M$: the four connected $(n-1)\times(n-1)$ minors,
corresponding to deletion of the first or last row and first or last column of~$M$, and the central $(n-2)\times(n-2)$ minor,
corresponding to deletion of the first and last rows and columns of~$M$.
Therefore, it seems that applying the Desnanot--Jacobi identity to the
matrix in the Izergin--Korepin formula, and using an assignment of spectral parameters in which~$u_1$,~$u_n$,~$v_1$ and~$v_n$ remain arbitrary, while
$u_2,\ldots,u_{n-1}$ are equal and $v_2,\ldots,v_{n-1}$ are equal,
should lead to an expression involving a quadruply-refined generating
function for $n\times n$ ASMs, four doubly-refined generating functions for $(n-1)\times(n-1)$ ASMs, and an unrefined generating function
for $(n-2)\times(n-2)$ ASMs.  Indeed, this essentially does occur, with some further, subsidiary aspects of the resulting expression being
related to the form of the prefactor in the Izergin--Korepin formula, and to the fact that each corner entry of an ASM is
associated with two of the parameters $u_1$,~$u_n$,~$v_1$ or $v_n$.

The Desnanot--Jacobi identity has, in fact, been associated with ASMs
since these matrices first arose.
In particular, the identity is used in Dodgson's condensation algorithm~\cite{Dod66} for determinant evaluation,
and it was through studies of this algorithm by Mills, Robbins and Rumsey that ASMs initially appeared.
In particular, it was shown by Robbins and Rumsey~\cite[Eq.~(27)]{RobRum86} that
if Dodgson's algorithm is applied to an $n\times n$ matrix~$M$ using a modified form of the Desnanot--Jacobi identity containing a parameter $\lambda$,
then the resulting so-called $\lambda$-determinant of~$M$ can be expressed naturally as a sum over $n\times n$ ASMs.
The latter expression involves the two ASM bulk statistics which are used in this paper, and,
as will be shown in Section~\ref{prev3}, the $\lambda$-determinant shares certain features with the
quadruply-refined ASM generating function studied in this paper.  For example,
the modified Desnanot--Jacobi identity satisfied by the $\lambda$-determinant
is analogous to the quadratic relation satisfied by the quadruply-refined ASM generating function.
For further information regarding Dodgson condensation, $\lambda$-determinants, and related matters, see, for example,
Abeles~\cite{Abe08}, Bressoud~\cite[Sec.~3.5]{Bre99}, Bressoud and Propp~\cite{BrePro99},
Di Francesco~\cite{Dif13}, Langer~\cite{Lan13}, Propp~\cite{Pro05}, or Robbins and Rumsey~\cite{RobRum86}.

It should also be noted that the Desnanot--Jacobi identity, the $\lambda$-modified Desnanot--Jacobi identity,
and further related equations studied by Robbins and Rumsey~\cite{RobRum86}, and others,
are closely connected with the Laurent phenomenon, cluster algebras, and associated areas.
For reviews of such connections, see, for example, Di Francesco~\cite[Sec.~5]{Dif12a},
Hone~\cite[Sec.~5]{Hon06}, or Propp~\cite[Sec.~10]{Pro01}.

Following a suggestion by a referee, a file \texttt{MultiplyRefinedEnumerationOfASMs.nb},
in which cases of the main results of this paper are demonstrated in Mathematica, has been provided at the author's web page.

\section{Definitions and basic properties}\label{defprop}
\subsection{Statistics}\label{stat}
In this section, the standard definitions of ASMs and certain statistics for ASMs
are given, and the elementary properties of these statistics are identified in detail.

An ASM, as first defined by Mills, Robbins and Rumsey~\cite{MilRobRum82,MilRobRum83},
is simply a square matrix in which each entry is~$0$, $1$ or~$-1$, and along each row and column
the nonzero entries alternate in sign and have a sum of~$1$.

It follows that, for any ASM $A$, each partial row sum $\sum_{j'=1}^jA_{ij'}$  and each partial column sum $\sum_{i'=1}^{i}A_{i'j}$ is 0 or 1.
It can also be seen that any permutation matrix is an ASM, and that, in any ASM,
the first and last rows and columns each contain a single~1, with all of their other entries being~$0$'s.

For each positive integer $n$, denote the set of all $n\times n$ ASMs as $\ASM(n)$.
For example, for $n=1,2,3$, these sets are
\begin{align}\notag
\ASM(1)&=\{(1)\},\\
\notag\ASM(2)&=\left\{\!\left(\begin{array}{@{}c@{\:\,}c@{}}1&0\\[-0.7mm]0&1\end{array}\right)\!,
\left(\begin{array}{@{}c@{\:\,}c@{}}0&1\\[-0.7mm]1&0\end{array}\right)\!\right\},\\
\label{ASM123}\ASM(3)&=\left\{\!
\left(\begin{array}{@{}c@{\:\,}c@{\:\,}c@{}}1&0&0\\[-0.7mm]0&1&0\\[-0.7mm]0&0&1\end{array}\right)\!,
\left(\begin{array}{@{}c@{\:\,}c@{\:\,}c@{}}0&1&0\\[-0.7mm]1&0&0\\[-0.7mm]0&0&1\end{array}\right)\!,
\left(\begin{array}{@{}c@{\:\,}c@{\:\,}c@{}}1&0&0\\[-0.7mm]0&0&1\\[-0.7mm]0&1&0\end{array}\right)\!,
\left(\begin{array}{@{}c@{\:\,}c@{\:\,}c@{}}0&1&0\\[-0.7mm]0&0&1\\[-0.7mm]1&0&0\end{array}\right)\!,
\left(\begin{array}{@{}c@{\:\,}c@{\:\,}c@{}}0&0&1\\[-0.7mm]1&0&0\\[-0.7mm]0&1&0\end{array}\right)\!,
\left(\begin{array}{@{}c@{\:\,}c@{\:\,}c@{}}0&0&1\\[-0.7mm]0&1&0\\[-0.7mm]1&0&0\end{array}\right)\!,
\left(\begin{array}{@{}c@{\:\,}c@{\:\,}c@{}}0&1&0\\[-0.7mm]1&-1&1\\[-0.7mm]0&1&0\end{array}\right)\!\right\}.\end{align}

For any $A\in\ASM(n)$, define statistics which depend on the bulk structure of $A$ as
\begin{equation}
\label{numuA}\nu(A)=\sum_{\substack{1\le i<i'\le n\\1\le j'\le j\le n}}\!A_{ij}\,A_{i'j'},\qquad\quad
\mu(A)=\text{number of $-1$'s in }A,\end{equation}
and define statistics which describe the configuration of $A$ at its top, right, bottom and left boundaries as, respectively,
\begin{align}
\notag\rT(A)&=\text{number of 0's to the left of the 1 in the top row of }A,\\[1mm]
\notag\rR(A)&=\text{number of 0's below the 1 in the right-most column of }A,\\[1mm]
\notag\rB(A)&=\text{number of 0's to the right of the 1 in the bottom row of }A,\\[1mm]
\label{rho}\rL(A)&=\text{number of 0's above the 1 in the left-most column of }A.\end{align}

Note that certain generalizations of the statistics~\eqref{numuA}--\eqref{rho} will be defined in Section~\ref{gensect}.

The statistic $\nu(A)$ in~\eqref{numuA} is a nonnegative integer for any $A\in\ASM(n)$, since it can
be written as $\nu(A)=\sum_{i,j=1}^n(\sum_{i'=1}^{i-1}A_{i'j})(\sum_{j'=1}^jA_{ij'})$,
where each factor in the summand (being a partial row or column sum of an ASM) is 0 or 1.
This statistic can also be written as
$\rule[-1.5ex]{0ex}{0ex}\nu(A)=\sum_{1\le i\le i'\le n;\;1\le j'<j\le n}A_{ij}\,A_{i'j'}$,
where this can be obtained from the definition in~\eqref{numuA} using the fact that each complete row and column sum of $A$ is a constant.
It will be seen in Section~\ref{ASMbijsect} that, in terms of the configuration of the six-vertex model with DWBC which
corresponds to~$A$,~$\nu(A)$ is simply the number of vertex configurations
of type~(1), or equally the number of vertex configurations of type~(2).

If~$A$ is a permutation matrix, then it can be seen from~\eqref{numuA} that~$\nu(A)$ is the number of inversions
in the permutation~$\pi$ given by $\delta_{\pi_i,j}=A_{ij}$.  Accordingly, for any
ASM $A$, $\nu(A)$ is referred to as the number of generalized inversions in~$A$.
This statistic was first defined and used
by Robbins and Rumsey~\cite[Eq.~(18)]{RobRum86}, who referred to it as the number
of positive inversions in an ASM~\cite[p.~182]{RobRum86}.
A closely-related statistic, $\rule[-1.5ex]{0ex}{0ex}\sum_{1\le i<i'\le n;\;1\le j'<j\le n}A_{ij}\,A_{i'j'}=\nu(A)+\mu(A)$,
was previously defined and
used by Mills, Robbins and Rumsey~\cite[p.~344]{MilRobRum83}, and provides an alternative generalized inversion number for each $A\in\ASM(n)$.

Proceeding to the other statistics, $\mu(A)$ in~\eqref{numuA} can also be written as $\mu(A)=\bigl(\sum_{i,j=1}^n(A_{ij})^2-n\bigr)/2$,
and the boundary statistics of~\eqref{rho} can be depicted diagrammatically as
\psset{unit=4mm}
\begin{equation}\label{rhodiag}\raisebox{-24.2mm}{\pspicture(-3,-1.7)(13,11.5)
\rput(5,5){$\left(\begin{array}{@{\,}c@{\;\;}c@{\;\;}c@{\;\;}c@{\;\;}c@{\;\;}c@{\;\;}c@{\;\;}c@{\;\;}c@{\,}}
0&0&0&0&1&0&0&0&0\\0&&&&&&&&0\\0&&&&&&&&0\\1&&&&&&&&0\\0&&&&A&&&&1\\0&&&&&&&&0\\0&&&&&&&&0\\0&&&&&&&&0\\
0&0&1&0&0&0&0&0&0\end{array}\right)$}
\rput(2.1,11.1){$\scriptstyle\rT(A)$}
\rput(6.6,-1){$\scriptstyle\rB(A)$}
\rput(-1.8,8.7){$\scriptstyle\rL(A)$}
\rput(11.9,2){$\scriptstyle\rR(A)$}
\rput(13.5,4){.}
\psline[linewidth=1pt]{<-}(0.3,11.1)(0.9,11.1)\psline[linewidth=1pt]{->}(3.3,11.1)(3.9,11.1)
\psline[linewidth=1pt]{<-}(3.6,-1)(5.4,-1)\psline[linewidth=1pt]{->}(7.8,-1)(9.6,-1)
\psline[linewidth=1pt]{<-}(-1.8,7.1)(-1.8,8.1)\psline[linewidth=1pt]{->}(-1.8,9.3)(-1.8,10.3)
\psline[linewidth=1pt]{<-}(11.9,-0.2)(11.9,1.4)\psline[linewidth=1pt]{->}(11.9,2.6)(11.9,4.2)
\endpspicture}\end{equation}

It can be checked that the statistics of~\eqref{numuA}--\eqref{rho}, regarded as functions on $\ASM(n)$, have ranges
\begin{align}\notag\nu(\ASM(n))&=\bigl\{0,1,\ldots,\tfrac{n(n-1)}{2}\bigr\},\\
\notag\mu(\ASM(n))&=\bigl\lfloor\tfrac{(n-1)^2}{4}\bigr\rfloor=\begin{cases}\bigl\{0,1,\ldots,\tfrac{(n-1)^2}{4}\bigr\},&n\text{ odd},\\
\bigl\{0,1,\ldots,\tfrac{n(n-2)}{4}\bigr\},&n\text{ even,}\end{cases}\\
\label{range}\rT(\ASM(n))&=\rR(\ASM(n))=\rB(\ASM(n))=\rL(\ASM(n))=\{0,1,\ldots,n-1\}.\end{align}

The ASMs $A$ which give the extreme values within the ranges in~\eqref{range} are as follows.
The only $A$ with $\nu(A)=0$ is the $n\times n$ identity matrix, and the only~$A$
with $\nu(A)=\frac{n(n-1)}{2}$ is the $n\times n$ antidiagonal matrix given by $A_{ij}=\delta_{i,n+1-j}$.
The $A$ with $\mu(A)=0$ are the $n!$ $n\times n$ permutation matrices, and the
$A$ with $\mu(A)=\bigl\lfloor\tfrac{(n-1)^2}{4}\bigr\rfloor$ are given by
$A_{ij}=(-1)^{i+j+k_n}$ if $k_n+2\le i+j\le 2n-k_n\text{ and }|i-j|\le k_n$,
and $A_{ij}=0$ otherwise, where $k_n=\frac{n-1}{2}$ for $n$ odd (giving a single such $A$), and $k_n=\frac{n}{2}-1$ or $\frac{n}{2}$
for $n$ even (giving two such $A$). The $A$ for which a boundary statistic of~\eqref{rho} is $0$ or $n-1$ are the
ASMs with a 1 in a corner, which will be discussed further shortly.

It can be seen that transposition or anticlockwise quarter-turn rotation of an ASM give another
ASM.  It can also be checked easily
that, for each $A\in\ASM(n)$, the statistics~\eqref{numuA}--\eqref{rho} behave under these operations according to
\begin{align}\notag\nu(A)&=\nu(A^T)=\tfrac{n(n-1)}{2}-\nu(A^Q)-\mu(A),\quad&\mu(A)&=\mu(A^T)=\mu(A^Q),\\
\notag\rT(A)&=\rL(A^T)=n-1-\rL(A^Q),&\rR(A)&=\rB(A^T)=n-1-\rT(A^Q),\\
\label{TQ}\rB(A)&=\rR(A^T)=n-1-\rR(A^Q),&\rL(A)&=\rT(A^T)=n-1-\rB(A^Q),\end{align}
where $Q$ denotes anticlockwise quarter-turn rotation, i.e., $A^Q_{ij}=A_{j,n+1-i}$.
By combining transposition and anticlockwise quarter-turn rotation, the behaviour of any of the eight operations of
the dihedral group acting on ASMs can be obtained.

The properties of ASMs with a 1 in a corner can be described easily.
For example, for the case of ASMs with a~1 in the top-left corner, the sets $\{A\in\ASM(n)\mid A_{11}=1\}=\{A\in\ASM(n)\mid\rT(A)=\rL(A)=0\}$
and $\ASM(n-1)$ are in bijection for any $n\ge2$,
where an ASM from the first set is mapped to the second set by simply deleting the first row and first column.
Furthermore, if $A$ from the first set is mapped to $A'$ in the second set, then
$\nu(A')=\nu(A)$, $\mu(A')=\mu(A)$, $\rR(A')=\rR(A)$ and $\rB(A')=\rB(A)$.

The properties of ASMs in which a 1 on a boundary is separated from a corner by a single zero,
i.e., ASMs in which a boundary statistic of~\eqref{rho} is $1$ or $n-2$, can also be described relatively easily.
For example, for the case of ASMs $A$ with $A_{21}=1$, the sets
$\{A\in\ASM(n)\mid A_{21}=A_{1,k+1}=1\}=\{A\in\ASM(n)\mid\rL(A)=1,\:\rT(A)=k\}$ and
$\{A\in\ASM(n-1)\mid A_{11}=\ldots=A_{1,k-1}=0\}=\{A\in\ASM(n-1)\mid\rT(A)\ge k-1\}$ are
in bijection for any $n\ge2$ and $k=1,\ldots,n-1$,
where an ASM $A$ from the first set is mapped to the second set by replacing $A_{2,k+1}$ by $A_{2,k+1}+1$,
while leaving all other entries unchanged, and then deleting the first row and first column.
Furthermore, if $A$ from the first set is mapped to $A'$ in the second set, then $\nu(A')=\nu(A)-A_{2,k+1}-1$ and $\mu(A')=\mu(A)+A_{2,k+1}$.

The previous bijection seems not to have appeared explicitly in the literature, although it is likely to have been known to several
authors.  In particular, it seems to have been used by Mills, Robbins and Rumsey~\cite[p.~344]{MilRobRum83}
in claiming that the case $k=1$ of their Conjecture~2
(which will appear as the identity $\A_{n,1}=\tfrac{n}{2}\,\A_{n-1}$ in~\eqref{A1id})
could be proved easily, and it also seems to have been used by Stroganov~\cite[p.~61]{Str06} in claiming that
a certain identity (which corresponds to the first equation of~\eqref{A1id}) could be proved bijectively.

\subsection{Generating functions}\label{genfunc}
In this section, various ASM generating functions are defined, and some simple relations which they satisfy are derived.
Each of these generating functions is described by a certain order of refinement, which
corresponds to the number of boundary statistics of~\eqref{rho} (or simply the number of boundaries) with which it is associated.
Each generating function is also associated with the two bulk statistics of~\eqref{numuA},
i.e., the numbers of generalized inversions and~$-1$'s in an ASM.

For each positive integer $n$,
define a quadruply-refined ASM generating function, which involves all six statistics of~\eqref{numuA}--\eqref{rho},
and associated indeterminates $x$, $y$, $z_1$, $z_2$, $z_3$ and $z_4$, as
\begin{equation}\label{Zquad}Z^\qua_n(x,y;z_1,z_2,z_3,z_4)=
\textstyle\sum_{A\in\ASM(n)}x^{\nu(A)}\,y^{\mu(A)}\,z_1^{\rT(A)}\,z_2^{\rR(A)}\,z_3^{\rB(A)}\,z_4^{\rL(A)}.\\
\end{equation}
It follows that $x$ and $y$ can be regarded as bulk parameters or weights,
and that $z_1$, $z_2$, $z_3$ and $z_4$ can be regarded as boundary parameters or weights.

It can be seen, using~\eqref{range} and~\eqref{Zquad},
that $Z^\qua_n(x,y;z_1,z_2,z_3,z_4)$ is a polynomial, with nonnegative integer coefficients,
of degree $\frac{n(n-1)}{2}$ in $x$, of degree $\bigl\lfloor\tfrac{(n-1)^2}{4}\bigr\rfloor$
in $y$, and of degree $n-1$ in each of $z_1$, $z_2$, $z_3$ and $z_4$.

Examples of the quadruply-refined ASM generating function~\eqref{Zquad}, for $n=1,2,3$, are
\begin{align}\notag Z^\qua_1(x,y;z_1,z_2,z_3,z_4)&=1,\\
\notag Z^\qua_2(x,y;z_1,z_2,z_3,z_4)&=1+x\,z_1\,z_2\,z_3\,z_4,\\
\notag Z^\qua_3(x,y;z_1,z_2,z_3,z_4)&=1+x\,z_1\,z_4+x\,z_2\,z_3+
x^2\,z_1\,z_2\,z_3^2\,z_4^2+x^2\,z_1^2\,z_2^2\,z_3\,z_4\;+\\
\label{Zquad123}&\hspace*{60mm}x^3\,z_1^2\,z_2^2\,z_3^2\,z_4^2+x\,y\,z_1\,z_2\,z_3\,z_4,\end{align}
where the terms are written in orders which correspond to those used in~\eqref{ASM123}.

Now define triply-refined, adjacent-boundary doubly-refined,
opposite-boundary doubly-refined, singly-refined and unrefined ASM generating functions as, respectively,
\begin{align}\label{Zdef}
\notag Z^\tri_n(x,y;z_1,z_2,z_3)&=Z^\qua_n(x,y;z_1,1,z_2,z_3)=
\textstyle\sum_{A\in\ASM(n)}x^{\nu(A)}\,y^{\mu(A)}\,z_1^{\rT(A)}\,z_2^{\rB(A)}\,z_3^{\rL(A)},\\
\notag Z^\adj_n(x,y;z_1,z_2)&=Z^\qua_n(x,y;z_1,1,1,z_2)=\textstyle\sum_{A\in\ASM(n)}x^{\nu(A)}\,y^{\mu(A)}\,z_1^{\rT(A)}\,z_2^{\rL(A)},\\
\notag Z^\opp_n(x,y;z_1,z_2)&=Z^\qua_n(x,y;z_1,1,z_2,1)=\textstyle\sum_{A\in\ASM(n)}x^{\nu(A)}\,y^{\mu(A)}\,z_1^{\rT(A)}\,z_2^{\rB(A)},\\
\notag Z_n(x,y;z)&=Z^\qua_n(x,y;z,1,1,1)=\textstyle\sum_{A\in\ASM(n)}x^{\nu(A)}\,y^{\mu(A)}\,z^{\rT(A)},\\
Z_n(x,y)&=Z^\qua_n(x,y;1,1,1,1)=\textstyle\sum_{A\in\ASM(n)}x^{\nu(A)}\,y^{\mu(A)},\end{align}
where $z$ is a further indeterminate.

Also define alternative quadruply-refined and alternative adjacent-boundary doubly-refined
ASM generating functions as, respectively,
\begin{align}\label{Zaltdef}
\notag\Z^\qua_n(x,y;z_1,z_2,z_3,z_4)&=(z_2z_4)^{n-1}\,Z^\qua_n(x,y;z_1,\tfrac{1}{z_2},z_3,\tfrac{1}{z_4})\\
\notag&\textstyle=\sum_{A\in\ASM(n)}x^{\nu(A)}\,y^{\mu(A)}\,z_1^{\rT(A)}\,z_2^{n-\rR(A)-1}\,z_3^{\rB(A)}\,z_4^{n-\rL(A)-1},\\
\Z^\adj_n(x,y;z_1,z_2)&=Z^\qua_n(x,y;z_1,z_2,1,1)=\textstyle\sum_{A\in\ASM(n)}x^{\nu(A)}\,y^{\mu(A)}\,z_1^{\rT(A)}\,z_2^{\rR(A)}.\end{align}
Note that $\Z^\qua_n(x,y;z_1,z_2,z_3,z_4)$ is a generating function in which
the positions of the 1's in the first and last columns of an ASM are measured relative to the
opposite ends of the columns to those used in~\eqref{rho}--\eqref{rhodiag},
i.e., in this generating function, the statistics associated with $z_2$ and $z_4$ are,
respectively, the numbers of 0's above the 1 in the right-most column, and below the 1 in the left-most column
of an ASM.

It can be seen immediately that some relations among generating functions of~\eqref{Zdef}--\eqref{Zaltdef}, involving specializations
of boundary parameters to~1, are
\begin{align}
\notag Z^\tri_n(x,y;z_1,1,z_2)&=Z^\adj_n(x,y;z_1,z_2),\\
\notag Z^\tri_n(x,y;1,z_1,z_2)&=\Z^\adj_n(x,y;z_1,z_2),\\
\notag Z^\tri_n(x,y;z_1,z_2,1)&=Z^\opp_n(x,y;z_1,z_2),\\
\notag Z^\adj_n(x,y;z,1)=\Z^\adj_n(x,y;z,1)&=Z^\opp_n(x,y;z,1)=Z_n(x,y;z),\\
\label{1}Z_n(x,y;1)&=Z_n(x,y).
\end{align}

Some elementary identities satisfied by the ASM generating functions of~\eqref{Zquad},~\eqref{Zdef} and~\eqref{Zaltdef},
which follow from the ASM properties outlined in Section~\ref{stat}, will now be obtained.

Note that, in this and all subsequent sections, many of the identities which contain the positive integer~$n$
will be valid only for all $n\ge2$, or for all $n\ge3$.
This will often be due to their containing
terms (such as $Z_{n-1}(x,y)$ or $Z_{n-2}(x,y)$)
which are not defined if $n$ is taken to be~1 or~2.

Note also that several of the identities of this section, obtained here using simple combinatorial arguments, can alternatively be obtained
as special cases of more general identities which will be derived in subsequent sections using other methods.

Some symmetry relations, which can be derived by acting on $\ASM(n)$ with transposition or
anticlockwise quarter-turn rotation, and using~\eqref{TQ}, are
\begin{align}
\notag Z^\qua_n(x,y;z_1,z_2,z_3,z_4)&=Z^\qua_n(x,y;z_4,z_3,z_2,z_1)\\
\notag&=x^{n(n-1)/2}\,(z_1z_2z_3z_4)^{n-1}\,
Z^\qua_n\bigl(\tfrac{1}{x},\tfrac{y}{x};\tfrac{1}{z_2},\tfrac{1}{z_3},\tfrac{1}{z_4},\tfrac{1}{z_1}\bigr),\\
\notag\Z^\qua_n(x,y;z_1,z_2,z_3,z_4)&=(z_1z_2z_3z_4)^{n-1}\,\Z^\qua_n(x,y;\tfrac{1}{z_4},\tfrac{1}{z_3},\tfrac{1}{z_2},\tfrac{1}{z_1})\\
\notag&=x^{n(n-1)/2}\,\Z^\qua_n\bigl(\tfrac{1}{x},\tfrac{y}{x};z_2,z_3,z_4,z_1\bigr),\\
\notag Z^\tri_n(x,y;z_1,z_2,z_3)&=x^{n(n-1)/2}\,(z_1z_2z_3)^{n-1}\,Z^\tri_n(\tfrac{1}{x},\tfrac{y}{x};
\tfrac{1}{z_2},\tfrac{1}{z_1},\tfrac{1}{z_3}),\\
\notag Z^\adj_n(x,y;z_1,z_2)&=Z^\adj_n(x,y;z_2,z_1)\\
\notag&=x^{n(n-1)/2}\,(z_1z_2)^{n-1}\,\Z^\adj_n\bigl(\tfrac{1}{x},\tfrac{y}{x};\tfrac{1}{z_1},\tfrac{1}{z_2}\bigr),\\
\notag Z^\opp_n(x,y;z_1,z_2)&=Z^\opp_n(x,y;z_2,z_1)\\
\notag&=x^{n(n-1)/2}\,(z_1z_2)^{n-1}\,Z^\opp_n\bigl(\tfrac{1}{x},\tfrac{y}{x};\tfrac{1}{z_1},\tfrac{1}{z_2}\bigr),\\
\notag Z_n(x,y;z)&=x^{n(n-1)/2}\,z^{n-1}\,Z_n\bigl(\tfrac{1}{x},\tfrac{y}{x};\tfrac{1}{z}\bigr),\\
\label{Zsymm}Z_n(x,y)&=x^{n(n-1)/2}\,z^{n-1}\,Z_n\bigl(\tfrac{1}{x},\tfrac{y}{x}).
\end{align}

Some identities involving specializations of boundary parameters to~0,
which follow from the properties of ASMs with a 1 in a corner, as discussed near the end of Section~\ref{stat}, are
\begin{gather}
\notag Z^\qua_n(x,y;z_1,0,z_2,z_3)=Z^\tri_n(x,y;z_1,0,z_3)=Z^\adj_{n-1}(x,y;z_1,z_3),\\
\notag \Z^\adj_n(x,y;z,0)=Z^\opp_n(x,y;z,0)=Z_{n-1}(x,y;z),\\
\label{0}Z^\adj_n(x,y;z,0)=Z_n(x,y;0)=Z_{n-1}(x,y).
\end{gather}

It will sometimes be useful to refer to boundary parameter coefficients in the adjacent-boundary doubly-refined
and singly-refined ASM generating functions.  In particular, define
\begin{align}
\notag Z^\adj_n(x,y)_{k_1,k_2}&=\text{coefficient of $z_1^{k_1}z_2^{k_2}$ in }Z^\adj_n(x,y;z_1,z_2),\\
\label{Zcoeff}Z_n(x,y)_k&=\text{coefficient of $z^k$ in }Z_n(x,y;z).
\end{align}

These coefficients satisfy identities which correspond to identities of~\eqref{1}--\eqref{0}.  For example,
\begin{gather}
\notag\textstyle\sum_{k_2=0}^{n-1}Z^\adj_n(x,y)_{k_1,k_2}=Z_n(x,y)_{k_1},\qquad\sum_{k=0}^{n-1}Z_n(x,y)_k=Z_n(x,y),\\
\notag Z^\adj_n(x,y)_{k_1,k_2}=Z^\adj_n(x,y)_{k_2,k_1},\qquad Z_n(x,y)_k=x^{n(n-1)/2}\,Z_n(\tfrac{1}{x},\tfrac{y}{x})_{n-1-k},\\
\notag Z^\adj_n(x,y)_{k,0}=Z_{n-1}(x,y)\,\delta_{k,0},\qquad Z^\adj_n(x,y)_{k+1,n-1}=x^{n-1}\,Z_{n-1}(x,y)_k,\\
\label{Zcoeffid}Z_n(x,y)_0=Z_{n-1}(x,y).\end{gather}

Furthermore, the properties of ASMs in which a 1 on a boundary is separated from a corner by a single zero,
as discussed at the end of Section~\ref{stat}, lead to identities such as
\begin{align}\notag Z^\adj_n(x,y)_{k,1}&=\textstyle xZ_{n-1}(x,y)_{k-1}+y\,\sum_{i=k}^{n-2}Z_{n-1}(x,y)_i-yZ_{n-1}(x,y)\delta_{k,0},\\
\label{Zcoeffid2}Z_n(x,y)_1&=\textstyle xZ_{n-1}(x,y)+y\,\sum_{k=1}^{n-2}k\,Z_{n-1}(x,y)_k.\end{align}

\section{Previously-known and related results}\label{prev}
In this section, an account is given of results for the exact enumeration of $n\times n$ ASMs
(for arbitrary, finite $n$), involving any of the six
statistics of~\eqref{numuA}--\eqref{rho}, or any of the associated
generating functions of~\eqref{Zquad},~\eqref{Zdef} or~\eqref{Zaltdef}.

Most of these results have appeared elsewhere in the literature
(although possibly in different formulations, or with different notation),
but some new results which are closely related to previously-known results are also presented.
The main new results in this section are~\eqref{ff1} and some cases of~\eqref{ff2},
for ASM generating functions in which the bulk parameters are related by $y=x+1$, and~\eqref{Xid},~\eqref{ZmultASM} and~\eqref{schurgen},
for a generating function, which will be introduced in~\eqref{ZrowASM}, associated with several rows (or several columns) of ASMs.
Derivations of the new results will be given in this section and in Section~\ref{proof}.

Various cases for the values of the parameters $x$ and $y$, associated with the
bulk statistics~\eqref{numuA}, will first be considered separately:
$y=0$ in Section~\ref{prev2},
$y=x+1$ in Section~\ref{prev3},
$x=y=1$ in Section~\ref{prev4},
and $x$ and $y$ arbitrary in Section~\ref{prev1}.
Further results will then be discussed in Sections~\ref{gensect}--\ref{invsymm}.

\subsection{Bulk parameter $y=0$}\label{prev2}
The case in which the bulk parameter $y$ is~$0$ corresponds simply to the enumeration of permutation matrices, with prescribed
values of the inversion number of the associated permutations, and prescribed positions of~$1$'s on
the boundaries of the matrices.

Let $x$-numbers and the $x$-factorial be defined, as usual, as $[n]_x=1+x+\ldots+x^{n-1}$,
and $[n]_x!=[n]_x[n-1]_x\ldots[1]_x$.

For $y=0$, the quadruply-refined ASM generating function is explicitly
\begin{multline}\label{perm1}Z^\qua_n(x,0;z_1,z_2,z_3,z_4)=\\
\shoveleft{\qquad x^2z_1z_2z_3z_4
\textstyle\sum_{0\le i<j\le n-3}\bigl(x^{n+i-j-3}z_4^iz_2^{n-j-3}+x^{n-i+j-4}z_4^{n-i-3}z_2^j\bigr)\;\times}\\
\shoveright{\textstyle\sum_{0\le i<j\le n-3}\bigl(x^{n+i-j-3}z_1^iz_3^{n-j-3}+x^{n-i+j-4}z_1^{n-i-3}z_3^j\bigr)\;[n\!-\!4]_x!\;+}\\
\bigl(xz_4z_1\,[n\!-\!2]_{xz_4}\,[n\!-\!2]_{xz_1}+z_1z_2(xz_3z_4)^{n-1}\,[n\!-\!2]_{xz_1}\,[n\!-\!2]_{xz_2}\;+
\qquad\qquad\qquad\qquad\qquad\\
xz_2z_3\,[n\!-\!2]_{xz_2}\,[n\!-\!2]_{xz_3}+
z_3z_4(xz_1z_2)^{n-1}\,[n\!-\!2]_{xz_3}\,[n\!-\!2]_{xz_4}\bigr)\,[n\!-\!3]_x!\;+\\
\bigl(1\,+\,x^{2n-3}(z_1z_2z_3z_4)^{n-1}\bigr)\,[n\!-\!2]_x!,\end{multline}
and the ASM generating functions of~\eqref{Zdef} are explicitly
\begin{align}
\notag
Z^\tri_n(x,0;z_1,z_2,z_3)&=\textstyle\bigl(xz_2\sum_{0\le i<j\le n-3}\bigl(x^{n+i-j-3}z_1^iz_2^{n-j-3}+x^{n-i+j-4}z_1^{n-i-3}z_2^j\bigr)\;+\\
\notag&\quad\qquad\qquad[n\!-\!2]_{xz_1}+z_2(xz_1)^{n-2}\,[n\!-\!2]_{xz_2}\bigr)\,xz_1z_3\,[n\!-\!2]_{xz_3}\,[n\!-\!3]_x!\;+\\
\notag&\hspace{-3mm}\bigl(z_1(xz_2z_3)^{n-1}\,[n\!-\!2]_{xz_1}+xz_2\,[n\!-\!2]_{xz_2}+1+x^{2n-3}(z_1z_2z_3)^{n-1}\bigr)\,[n\!-\!2]_x!\\
\notag Z^\adj_n(x,0;z_1,z_2)&=xz_1z_2\,[n\!-\!1]_{xz_1}\,[n\!-\!1]_{xz_2}\,[n\!-\!2]_x!\,+\,[n\!-\!1]_x!,\\
\notag Z^\opp_n(x,0;z_1,z_2)&=\textstyle\sum_{0\le i<j\le n-1}\bigl(x^{n+i-j-1}\,z_1^i\,z_2^{n-j-1}+x^{n-i+j-2}\,z_1^{n-i-1}\,
z_2^j\bigr)\,[n\!-\!2]_x!,\\
\notag Z_n(x,0;z)&=[n]_{xz}\,[n\!-\!1]_x!,\\
\label{perm2}Z_n(x,0)&=[n]_x!.\end{align}

Each of these formulae is either a standard result
for permutations, or a straightforward variation of such a result, and each formula
can be derived using simple combinatorial arguments.
See, for example, Stanley~\cite[Cor.~1.3.13]{Sta12} for a derivation of the last equation of~\eqref{perm2}.

Note that if~\eqref{perm1} is shown first to be valid, then each equation of~\eqref{perm2} can subsequently be obtained
by setting certain boundary parameters in~\eqref{perm1} to~$1$.
Alternatively, if the last two equations of~\eqref{perm2}
are shown first to be valid, then~\eqref{perm1} and the remaining equations of~\eqref{perm2} can subsequently be obtained by
using the general results~\eqref{quadrel},~\eqref{triprel} (or~\eqref{triprelalt}),~\eqref{opprel} and~\eqref{adjdoub2},
which will be given in Section~\ref{newres}.

In interpreting~\eqref{perm1}, note also that the seven main terms on the RHS of~\eqref{perm1} (i.e., the term ending in $[n-4]_x!$,
the four terms ending
in $[n-3]_x!$ and the two terms ending in $[n-2]_x!$) correspond to sums over
the sets of $n\times n$ permutation matrices $A$ with $(A_{11},A_{1n},A_{nn},A_{n1})$ equal to
$(0,0,0,0)$, $(0,0,1,0)$, $(0,0,0,1)$, $(1,0,0,0)$, $(0,1,0,0)$, $(1,0,1,0)$ or $(0,1,0,1)$, respectively.
In subsequent formulae for quadruply-refined ASM generating functions (e.g.,~\eqref{ff1},~\eqref{unwquad}
and~\eqref{quadrel}),
the RHS will again consist of seven or eight main terms, but these terms will no longer correspond simply to sums over
sets of ASMs with fixed values of the four corner entries.

\subsection{Bulk parameters satisfying $y=x+1$}\label{prev3}
The case in which the bulk parameters are related by $y=x+1$ is closely related to
$x$-determinants (or, in the notation usually used, $\lambda$-determinants) of matrices,
domino tilings of an Aztec diamond, tournaments, and the free fermion case of the six-vertex model.
(Note that the sub-case $x=1$ and $y=2$, which corresponds to the so-called $2$-enumeration of ASMs,
is often considered separately.)
Due to certain combinatorial and algebraic simplifications, it is again possible to obtain explicit formulae.

For $y=x+1$, the quadruply-refined ASM generating function is given explicitly by
\begin{multline}\label{ff1}
(xz_4z_1\!+\!z_4\!+\!z_1\!-\!1)(xz_1z_2\!-\!xz_1\!-\!xz_2\!-\!1)\;\times\\
(xz_2z_3\!+\!z_2\!+\!z_3\!-\!1)(xz_3z_4\!-\!xz_3\!-\!xz_4\!-\!1)\;Z^\qua_n(x,x\!+\!1;z_1,z_2,z_3,z_4)=\\
\shoveleft{z_1z_2z_3z_4(xz_1z_3\!+\!1)(xz_2z_4\!+\!1)\bigl((xz_1\!+\!1)(xz_2\!+\!1)(xz_3\!+\!1)(xz_4\!+\!1)\bigr)^{n-2}(x\!+\!1)^{(n-4)(n-5)/2}\;-}\\
\Bigl((z_2\!-\!1)(z_3\!-\!1)z_4z_1\bigl((xz_4\!+\!1)(xz_1\!+\!1)\bigr)^{n-2}\;+\qquad\qquad\qquad\qquad\qquad\\[-1mm]
\qquad\qquad\qquad\qquad\qquad(z_4\!-\!1)(z_1\!-\!1)z_2z_3\bigl((xz_2\!+\!1)(xz_3\!+\!1)\bigr)^{n-2}\Bigr)\,\times\\
\qquad(xz_1z_2\!-\!xz_1\!-\!xz_2\!-\!1)(xz_3z_4\!-\!xz_3\!-\!xz_4\!-\!1)(x\!+\!1)^{(n-3)(n-4)/2}\;-\\
\Bigr((z_3\!-\!1)(z_4\!-\!1)z_1z_2(z_3z_4)^{n-1}\bigl((xz_1\!+\!1)(xz_2\!+\!1)\bigr)^{n-2}\;+\qquad\qquad\qquad\\[-1mm]
\qquad\qquad\qquad(z_1\!-\!1)(z_2\!-\!1)z_3z_4(z_1z_2)^{n-1}\bigl((xz_3\!+\!1)(xz_4\!+\!1)\bigr)^{n-2}\Bigr)\;\times\\
\qquad(xz_4z_1\!+\!z_4\!+\!z_1\!-\!1)(xz_2z_3\!+\!z_2\!+\!z_3\!-\!1)\,x^n(x\!+\!1)^{(n-3)(n-4)/2}\;+\\
(z_1\!-\!1)(z_2\!-\!1)(z_3\!-\!1)(z_4\!-\!1)\Bigl((xz_1z_2\!-\!xz_1\!-\!xz_2\!-\!1)(xz_3z_4\!-\!xz_3\!-\!xz_4\!-\!1)\;+\qquad\qquad\\
(xz_4z_1\!+\!z_4\!+\!z_1\!-\!1)(xz_2z_3\!+\!z_2\!+\!z_3\!-\!1)(z_1z_2z_3z_4)^{n-1}x^{2n-1}\Bigr)(x\!+\!1)^{(n-2)(n-3)/2},
\end{multline}
and the ASM generating functions of~\eqref{Zdef} are given explicitly by
\begin{align}
\notag&\hspace*{-45mm}(xz_1z_3\!+\!z_1\!+\!z_3\!-\!1)(xz_2z_3\!-\!xz_2\!-\!xz_3\!-\!1)Z^\tri_n(x,x\!+\!1;z_1,z_2,z_3)=\\
\notag&\hspace*{-35mm}-z_1z_3(xz_1z_2\!+\!1)\bigl((xz_1\!+\!1)(xz_2\!+\!1)\bigr)^{n-2}(xz_3\!+\!1)^{n-1}(x\!+\!1)^{(n-3)(n-4)/2}\;+\\
\notag&\hspace*{-35mm}(z_2\!-\!1)(z_3\!-\!1)(xz_1z_3\!+\!z_1\!+\!z_3\!-\!1)z_1(z_2z_3)^{n-1}(xz_1\!+\!1)^{n-2}x^n(x\!+\!1)^{(n-2)(n-3)/2}\;-\\
\notag&\hspace*{-30mm}(z_1\!-\!1)(z_3\!-\!1)(xz_2z_3\!-\!xz_2\!-\!xz_3\!-\!1)(xz_2\!+\!1)^{n-2}(x\!+\!1)^{(n-2)(n-3)/2},\\
\notag&\hspace*{-45mm}(xz_1z_2\!+\!z_1\!+\!z_2\!-\!1)Z^\adj_n(x,x\!+\!1;z_1,z_2)=
z_1z_2\bigl((xz_1\!+\!1)(xz_2\!+\!1)\bigr)^{n-1}(x\!+\!1)^{(n-2)(n-3)/2}\hspace{13mm}\\
\notag&\hspace*{45mm}-\,(z_1\!-\!1)(z_2\!-\!1)\,(x\!+\!1)^{(n-1)(n-2)/2},\qquad\\
\notag\hspace*{15mm}Z^\opp_n(x,x\!+\!1;z_1,z_2)&=(xz_1z_2\!+\!1)\,\bigl((xz_1\!+\!1)(xz_2\!+\!1)\bigr)^{n-2}\,(x\!+\!1)^{(n-2)(n-3)/2},\\
\notag Z_n(x,x\!+\!1;z)&=(xz\!+\!1)^{n-1}\,(x\!+\!1)^{(n-1)(n-2)/2},\\
\label{ff2}Z_n(x,x\!+\!1)&=(x\!+\!1)^{n(n-1)/2}.\end{align}

Each of these formulae is either a previously-known result,
or can be obtained relatively easily from such results.  More specifically, the last three cases of~\eqref{ff2} have appeared,
in various forms, in the literature (see the references at the end of this section), but it seems that~\eqref{ff1} and the first
two cases of~\eqref{ff2} have not.  A formula which can be regarded as a generalization of the last three cases of~\eqref{ff2}
will be given near the end of Section~\ref{gensect}.

A derivation of~\eqref{ff1}--\eqref{ff2}, based on a result of Robbins and Rumsey~\cite[Sec.~5]{RobRum86} for $x$-determinants,
will now be given, since this derivation involves the use of a quadratic relation satisfied by the quadruply-refined ASM
generating function, and thereby
shares some features with new material which will be presented in
Section~\ref{newres}.
An alternative derivation, based on the Izergin--Korepin formula for the partition function of the six-vertex model with DWBC
and the Cauchy double alternant evaluation, will be given in Section~\ref{add3}.

For an $n\times n$ matrix of indeterminates, $\bigl(M_{ij}\bigr)_{1\le i,j\le n}$, define
\begin{equation}\label{lambdadet}\mathcal{Z}_n(x,M)=
\textstyle\sum_{A\in\ASM(n)}x^{\nu(A)}\,(x+1)^{\mu(A)}\,\prod_{i,j=1}^n(M_{ij})^{A_{ij}},\end{equation}
this being the so-called $x$-determinant of $M$,
as introduced by Robbins and Rumsey~\cite[Sec.~5]{RobRum86}.
It can be seen that~$\mathcal{Z}_n(-1,M)$ is the standard determinant of $M$,
since the RHS of~\eqref{lambdadet} is then a sum over all $n\times n$ permutation matrices $A$, with~$\nu(A)$
being the number of inversions in the permutation associated with $A$.

It follows from a result of Robbins and Rumsey~\cite[Sec.~5]{RobRum86} that
\begin{equation}\label{lambdaDesJac}\mathcal{Z}_n(x,M)\,\mathcal{Z}_{n-2}(x,M_\mathrm{C})=\mathcal{Z}_{n-1}(x,M_\mathrm{TL})\,
\mathcal{Z}_{n-1}(x,M_\mathrm{BR})+x\,\mathcal{Z}_{n-1}(x,M_\mathrm{TR})\,
\mathcal{Z}_{n-1}(x,M_\mathrm{BL}),\end{equation}
where $M_\mathrm{TL}$, $M_\mathrm{TR}$, $M_\mathrm{BR}$ and $M_\mathrm{BL}$ denote the $(n-1)\times(n-1)$ submatrices
corresponding to the top-left, top-right, bottom-right and bottom-left corners of~$M$, respectively,
and~$M_\mathrm{C}$ denotes the central $(n-2)\times(n-2)$ submatrix of~$M$.
For $x=-1$,~\eqref{lambdaDesJac} is the Desnanot--Jacobi determinant identity, which
will be discussed in more detail in Section~\ref{DesJacsect}.

Taking $M$ in~\eqref{lambdaDesJac} to be
\begin{equation}M=\begin{pmatrix}1&z_1&\ldots&z_1^{n-2}&(z_1z_2)^{n-1}\\
z_4&1&\ldots&1&z_2^{n-2}\\
\vdots&\vdots&&\vdots&\vdots\\
z_4^{n-2}&1&\ldots&1&z_2\\
(z_3z_4)^{n-1}&z_3^{n-2}&\ldots&z_3&1\end{pmatrix},\end{equation}
it follows, using definitions from~\eqref{Zquad}, \eqref{Zdef},~\eqref{Zaltdef} and~\eqref{lambdadet}, that
the quadruply-refined ASM generating function satisfies the quadratic relation
\begin{multline}\label{ffquad}Z^\qua_n(x,x\!+\!1;z_1,z_2,z_3,z_4)\,Z_{n-2}(x,x\!+\!1)=
Z^\adj_{n-1}(x,x\!+\!1;z_4,z_1)\,Z^\adj_{n-1}(x,x\!+\!1;z_2,z_3)\;+\\
xz_1z_2z_3z_4\,\Z^\adj_{n-1}(x,x\!+\!1;z_1,z_2)\,\Z^\adj_{n-1}(x,x\!+\!1;z_3,z_4).\end{multline}
Note that a quadratic relation for arbitrary $x$ and $y$ will be given in~\eqref{quadrel}, and that a
different quadratic relation for the present case can be obtained by setting $y=x+1$ in~\eqref{quadrel}.

Setting $z_1=z_2=z_3=z_4=1$ in~\eqref{ffquad} gives
\begin{equation}Z_n(x,x\!+\!1)\,Z_{n-2}(x,x\!+\!1)=(x\!+\!1)\,Z_{n-1}(x,x\!+\!1)^2,\end{equation}
which, together with~$Z_1(x,x\!+\!1)=1$ and $Z_2(x,x\!+\!1)=x\!+\!1$, gives the formula in~\eqref{ff2} for $Z_n(x,x\!+\!1)$.

Setting $z_2=z_3=z_4=1$ in~\eqref{ffquad}, and relabelling $z_1$ as $z$, gives
\begin{equation}Z_n(x,x\!+\!1;z)\,Z_{n-2}(x,x\!+\!1)=(xz\!+\!1)\,Z_{n-1}(x,x\!+\!1;z)\,Z_{n-1}(x,x\!+\!1),\end{equation}
which, together with $Z_1(x,x\!+\!1;z)=1$ and the formula for $Z_n(x,x\!+\!1)$, gives the formula in~\eqref{ff2} for $Z_n(x,x\!+\!1;z)$.
The formula for $Z^\opp_n(x,x\!+\!1;z_1,z_2)$ can be obtained similarly, by setting $z_2=z_4=1$ in~\eqref{ffquad},
and using the formulae for $Z_n(x,x\!+\!1;z)$ and $Z_n(x,x\!+\!1)$.

Setting $z_2=z_3=1$ in~\eqref{ffquad}, and relabelling $z_4$ as $z_2$, gives
\begin{multline}
Z^\adj_n(x,x\!+\!1;z_1,z_2)\,Z_{n-2}(x,x\!+\!1)=\\
Z^\adj_{n-1}(x,x\!+\!1;z_1,z_2)\,Z_{n-1}(x,x\!+\!1)+
xz_1z_2\,Z_{n-1}(x,x\!+\!1;z_1)\,Z_{n-1}(x,x\!+\!1;z_2),\end{multline}
which, together with $Z^\adj_1(x,x\!+\!1;z_1,z_2)=1$, gives
\begin{equation}\label{ffZadj}Z^\adj_n(x,x\!+\!1;z_1,z_2)=Z_{n-1}(x,x\!+\!1)\,\biggl(1+
xz_1z_2\,\sum_{i=1}^{n-1}\frac{Z_i(x,x\!+\!1;z_1)\,Z_i(x,x\!+\!1;z_2)}
{Z_{i-1}(x,x\!+\!1)\,Z_i(x,x\!+\!1)}\biggr),\end{equation}
where, in the sum over $i$, $Z_0(x,x\!+\!1)$ is taken to be 1.
Substituting the formulae for $Z_i(x,x\!+\!1;z)$ and $Z_i(x,x\!+\!1)$ into~\eqref{ffZadj},
and then performing the sum over $i$ and simplifying,
gives the formula in~\eqref{ff2} for $Z^\adj_n(x,x\!+\!1;z_1,z_2)$.

Finally, the formulae in~\eqref{ff1}--\eqref{ff2} for $Z^\qua_n(x,x\!+\!1;z_1,z_2,z_3,z_4)$ and
$Z^\tri_n(x,x\!+\!1;z_1,z_2,z_3)$ can be obtained
using~\eqref{ffquad}, together with the formulae for $Z^\adj_n(x,x\!+\!1;z_1,z_2)$ and
$Z_n(x,x\!+\!1)$, and a similar formula for $\Z^\adj_n(x,x\!+\!1;z_1,z_2)$
(which can be obtained from the formula for $Z^\adj_n(x,x\!+\!1;z_1,z_2)$
and the relevant equation from~\eqref{Zsymm}).

For further information regarding various aspects of the case $y=x+1$ and closely related cases,
including details of several alternative methods for obtaining formulae such as~\eqref{ff1}--\eqref{ff2},
see, for example,
Bogoliubov, Pronko and Zvonarev~\cite[Sec.~5]{BogProZvo02},
Bosio and van Leeuwen~\cite{BosVan12},
Bousquet-M\'elou and Habsieger~\cite[Sec.~5]{BouHab93},
Bressoud~\cite{Bre01},
Brualdi and Kirkland~\cite{BruKir05},
Brubaker, Bump and Friedberg~\cite{BruBumFri11a,BruBumFri11b},
Bump, McNamara and Nakasuji~\cite{BumMcnNak11},
Chapman~\cite{Cha01},
Ciucu~\cite{Ciu96},~\cite[Thm.~6.1]{Ciu98},
Colomo and Pronko~\cite[Secs.~4.2 \&~5.2]{ColPro05a},~\cite[Sec.~4.3]{ColPro06},
Di Francesco~\cite{Dif13},
Elkies, Kuperberg, Larsen and Propp~\cite{ElkKupLarPro92a,ElkKupLarPro92b},
Eu and Fu~\cite{EuFu05},
Ferrari and Spohn~\cite{FerSpo06},
Hamel and King~\cite[Cor.~5.1]{HamKin07},
Kuo~\cite[Sec.~3]{Kuo04},
Kuperberg~\cite[Thm.~3]{Kup02},
Langer~\cite{Lan13},
Lascoux~\cite{Las07a},
McNamara~\cite{Mcn10},
Mills, Robbins and Rumsey~\cite[Sec.~6]{MilRobRum83},
Okada~\cite[Thm.~2.4(1), third eq.]{Oka06},
Rosengren~\cite[Sec.~9]{Ros09},
Striker~\cite[Sec.~5]{Str09},~\cite[Sec.~6]{Str11b},
Tokuyama~\cite[Cor.~3.4]{Tok88}, and Yang~\cite{Yan91}.

\subsection{Bulk parameters $x=y=1$}\label{prev4}
The case in which the bulk parameters~$x$ and~$y$ are both~1 involves
numbers of ASMs with prescribed configurations on certain boundaries, but without
prescribed values of bulk statistics.  The cases of zero, one or
two boundaries will be considered together first, followed by the cases of three or four boundaries.

Define unrefined, singly-refined, opposite-boundary doubly-refined and adjacent-boundary doubly-refined ASM numbers as, respectively,
\begin{align}\notag\A_n&=|\ASM(n)|,\\
\notag\A_{n,k}&=|\{A\in\ASM(n)\mid A_{1,k+1}=1\}|,\\
\notag\A^\opp_{n,k_1,k_2}&=|\{A\in\ASM(n)\mid A_{1,k_1+1}=A_{n,n-k_2}=1\}|,\\
\label{AA}\A^\adj_{n,k_1,k_2}&=|\{A\in\ASM(n)\mid A_{1,k_1+1}=A_{k_2+1,1}=1\}|,\end{align}
for $0\le k,k_1,k_2\le n-1$, with the numbers being~0
for~$k$,~$k_1$ or~$k_2$ outside this range.
These numbers are therefore related to generating functions of~\eqref{Zdef}--\eqref{Zaltdef} by
\begin{align}\notag Z_n(1,1)&=\A_n,\\
\notag Z_n(1,1;z)&=\textstyle\sum_{k=0}^{n-1}\,\A_{n,k}\,z^k,\\
\notag Z^\opp_n(1,1;z_1,z_2)&=\textstyle\sum_{k_1,k_2=0}^{n-1}\,\A^\opp_{n,k_1,k_2}\,z_1^{k_1}z_2^{k_2},\\
\notag Z^\adj_n(1,1;z_1,z_2)&=\textstyle\sum_{k_1,k_2=0}^{n-1}\,\A^\adj_{n,k_1,k_2}\,z_1^{k_1}z_2^{k_2},\\
\label{Z1}\textstyle\Z^\adj_n(1,1;z_1,z_2)&=\textstyle\sum_{k_1,k_2=0}^{n-1}\,\A^\adj_{n,n-1-k_1,n-1-k_2}\,z_1^{k_1}z_2^{k_2},\end{align}
and to the boundary-parameter coefficients of~\eqref{Zcoeff} by
\begin{equation}\A_{n,k}=Z_n(1,1)_k,\qquad\A^\adj_{n,k_1,k_2}=Z^\adj_n(1,1)_{k_1,k_2}.\end{equation}

The ASM numbers of~\eqref{AA} satisfy various elementary identities related to
their definitions, symmetry properties of ASMs or properties of ASMs with a 1 in a corner,
and which correspond to identities, such as~\eqref{1}--\eqref{0} and~\eqref{Zcoeffid}, satisfied by the generating functions.
Some examples of these identities are
\begin{gather}
\notag\textstyle\sum_{k=0}^{n-1}\A_{n,k}=\A_n,\qquad \A_{n,k}=\A_{n,n-1-k},\qquad\A_{n,0}=\A_{n-1},\\
\notag\textstyle\sum_{k_2=0}^{n-1}\A^\opp_{n,k_1,k_2}=\sum_{k_2=0}^{n-1}\A^\adj_{n,k_1,k_2}=\A_{n,k_1},\\
\notag\A^\opp_{n,k_1,k_2}=\A^\opp_{n,k_2,k_1}=\A^\opp_{n,n-1-k_1,n-1-k_2},\quad\A^\adj_{n,k_1,k_2}=\A^\adj_{n,k_2,k_1},\\
\label{A1}\A^\adj_{n,k,0}=\A_{n-1}\,\delta_{k,0},\qquad\A^\opp_{n,k,0}=\A^\adj_{n,k+1,n-1}=\A_{n-1,k}.\end{gather}
Furthermore, some identities related to properties of ASMs in which a 1 on a boundary is separated from a corner by a single zero,
and which can be obtained from~\eqref{Zcoeffid2} and the first two identities of~\eqref{A1}, are
\begin{equation}\label{A1id}\A^\adj_{n,k,1}=\textstyle\sum_{i=k-1}^{n-2}\A_{n-1,i}-\A_{n-1}\,\delta_{k,0},\qquad
\A_{n,1}=\textstyle\A_{n-1}+\sum_{k=1}^{n-2}k\,\A_{n-1,k}=\tfrac{n}{2}\,\A_{n-1}.\end{equation}

Proceeding to more general identities, for which derivations using combinatorial arguments are not currently known,
formulae which give the ASM numbers of~\eqref{AA} explicitly are
\begin{align}\label{ASM}\A_n&=\textstyle\prod_{i=0}^{n-1}\frac{(3i+1)!}{(n+i)!},\\
\label{RASM}\A_{n,k}&=\begin{cases}\frac{(n+k-1)!\:(2n-k-2)!}{k!\:(n-k-1)!\:(2n-2)!}\:\prod_{i=0}^{n-2}\!\frac{(3i+1)!}{(n+i-1)!},&0\le k\le n-1,\\
0,&\text{otherwise,}\end{cases}\\
\notag\A^\opp_{n,k_1,k_2}&=\textstyle\frac{1}{\A_{n-1}}\sum_{i=0}^{\min(k_1,n-k_2-1)}\bigl(
\A_{n,k_1-i}\,\A_{n-1,k_2+i}+\A_{n-1,k_1-i-1}\,\A_{n,k_2+i}\;-\\[-1.7mm]
\label{oppRASM}&\hspace{54mm}\A_{n,k_1-i-1}\,\A_{n-1,k_2+i}-\A_{n-1,k_1-i-1}\,\A_{n,k_2+i+1}\bigr),\\
\label{adjRASM}\A^\adj_{n,k_1,k_2}&=\begin{cases}\A_{n-1},&k_1=k_2=0,\\
\binom{k_1+k_2-2}{k_1-1}\A_{n-1}-\sum_{i=1}^{k_1}\sum_{j=1}^{k_2}\binom{k_1+k_2-i-j}{k_1-i}\A^\opp_{n,i-1,n-j},&1\le k_1,k_2\le n-1,\\
0,&\text{otherwise,}\end{cases}\end{align}

It follows from~\eqref{ASM}--\eqref{RASM} that the unrefined and singly-refined ASM numbers
satisfy simple recursion relations, such as
\begin{align}
\notag\A_n&=\tfrac{(n-1)!\,(3n-2)!}{(2n-2)!\,(2n-1)!}\,\A_{n-1},\\
\notag\A_n\,\A_{n-2}&=\tfrac{3(3n-2)(3n-4)}{4(2n-1)(2n-3)}\,\A_{n-1}^2,\\
\label{MRRrec}k(2n\!-\!k\!-\!1)\A_{n,k}&=(n\!-\!k)(n\!+\!k\!-\!1)\,\A_{n,k-1}.\end{align}
It also follows that the singly-refined ASM generating function at $x=y=1$ can be expressed in terms of the Gaussian hypergeometric
function as
\begin{equation}\label{hyp}Z_n(1,1;z)=\A_{n-1}\:{}_2F_1\!\left[\begin{matrix}1\!-\!n,\,n\\\,2\!-\!2n\end{matrix}\,;\,z\right],\end{equation}
and satisfies the hypergeometric differential equation
\begin{equation}\label{diff}z(1\!-\!z)\tfrac{d^2}{dz^2}Z_n(1,1;z)+2(1\!-\!n\!-\!z)\tfrac{d}{dz}Z_n(1,1;z)+n(n\!-\!1)Z_n(1,1;z)=0.\end{equation}

Furthermore, the singly-refined ASM numbers satisfy, for $0\le k\le n-1$,
\begin{equation}\label{RASMlinear}\textstyle\A_{n,k}=\sum_{i=0}^k(-1)^i\binom{n+k-1}{k-i}\A_{n,i}.\end{equation}

Recursion relations which involve the opposite-boundary and adjacent-boundary doubly-refined ASM numbers are
\begin{align}
\notag(\A^\opp_{n,k_1-1,k_2}-\A^\opp_{n,k_1,k_2-1})\,\A_{n-1}&=\A_{n,k_1-1}\,\A_{n-1,k_2-1}-\A_{n,k_1}\,\A_{n-1,k_2-1}\;-\\
\label{oppRASMrec}&\hspace{28mm}\A_{n-1,k_1-1}\,\A_{n,k_2-1}+\A_{n-1,k_1-1}\,\A_{n,k_2},\\
\label{adjRASMrec}\A^\adj_{n,k_1-1,k_2}+\A^\adj_{n,k_1,k_2-1}-\A^\adj_{n,k_1,k_2}&=
\A^\opp_{n,k_1-1,n-k_2}-(\delta_{k_1,1}-\delta_{k_1,0})(\delta_{k_2,1}-\delta_{k_2,0})\,\A_{n-1}.\end{align}
It follows that, in terms of generating functions,~\eqref{oppRASMrec} can be written as
\begin{multline}\label{Zoppid}
(z_1\!-\!z_2)\,Z^\opp_n(1,1;z_1,z_2)\,\A_{n-1}=(z_1\!-\!1)\,z_2\,Z_n(1,1;z_1)\,Z_{n-1}(1,1;z_2)\;-\\
z_1\,(z_2\!-\!1)\,Z_{n-1}(1,1;z_1)\,Z_n(1,1;z_2),\end{multline}
and~\eqref{adjRASMrec} can be written as
\begin{equation}\label{Zoppadjid}(z_1\!+\!z_2\!-\!1)\,Z^\adj_n(1,1;z_1,z_2)=
z_1\,z_2^n\,Z^\opp_n(1,1;z_1,\tfrac{1}{z_2})
-(z_1\!-\!1)(z_2\!-\!1)\,\A_{n-1}.\end{equation}
The relation~\eqref{Zoppadjid} will be used in Section~\ref{new3} for the derivation of results for the quadruply-refined
ASM generating function at $x=y=1$.

The opposite-boundary doubly-refined ASM generating function at $x=y=1$ can also be expressed as
\begin{multline}\label{schur}
Z^\opp_n(1,1;z_1,z_2)=3^{-n(n-1)/2}\,\bigl(q^2(z_1+q)(z_2+q)\bigr)^{n-1}\times\\
s_{(n-1,n-1,\ldots,2,2,1,1)}\bigl(\tfrac{qz_1+1}{z_1+q},\tfrac{qz_2+1}{z_2+q},
\,\underbrace{1,\ldots,1}_{2n-2}\,\bigr)\big|_{q=e^{\pm 2\pi i/3}},\end{multline}
where $s_{(n-1,n-1,\ldots,2,2,1,1)}\bigl(\tfrac{qz_1+1}{z_1+q},\tfrac{qz_2+1}{z_2+q},
1,\ldots,1\bigr)$ is the Schur function indexed by the double-staircase
partition $(n-1,n-1,\ldots,2,2,1,1)$, evaluated at the $2n$ parameters $\frac{qz_1+1}{z_1+q},\frac{qz_2+1}{z_2+q},$ $1,\ldots,1$.
Setting $z_1=1$ or $z_2=1$ in~\eqref{schur} gives an expression for the singly-refined ASM generating function at $x=y=1$, while setting
$z_1=z_2=1$ in~\eqref{schur}, and using~\eqref{Z1} and~\eqref{A1}, it follows that
\begin{multline}\label{SSYT}\A_n=3^{-n(n-1)/2}\times\bigl(\text{number of semistandard Young tableaux of shape}\\
(n-1,n-1,\ldots,2,2,1,1)\text{ with entries from }\{1,\ldots,2n\}\bigr).\end{multline}
The product formula~\eqref{ASM} for $\A_n$ can be obtained from~\eqref{SSYT} using the hook-content formula for
semistandard Young tableaux.

A further identity satisfied by the opposite-boundary doubly-refined and unrefined ASM numbers is
\begin{equation}\label{BCS}\det_{0\le k_1,k_2\le n-1}(\A^\opp_{n,k_1,k_2})=(-1)^{n(n+1)/2+1}\,(\A_{n-1})^{n-3}.\end{equation}

Additional expressions for some of the previous ASM numbers or associated ASM generating functions will be discussed in Section~\ref{TSSCPP}.

Proceeding now to the enumeration of ASMs with prescribed configurations on three or four boundaries,
the alternative quadruply-refined ASM generating function at $x=y=1$ is given by
\begin{multline}\label{unwquad}
(z_4z_1\!-\!z_4\!+\!1)(z_1z_2\!-\!z_1\!+\!1)(z_2z_3\!-\!z_2\!+\!1)(z_3z_4\!-\!z_3\!+\!1)\,\Z^\qua_n(1,1;z_1,z_2,z_3,z_4)=\\
\frac{z_1\,z_2\,z_3\,z_4\,\det_{1\le i,j\le 4}
\bigl(z_i^{\;j-1}\,(z_i\!-\!1)^{4-j}\,Z_{n-j+1}(1,1;z_i)\bigr)}{\rule{0ex}{2.3ex}\A_{n-1}\,\A_{n-2}\,\A_{n-3}\,\prod_{1\le i<j\le 4}(z_i-z_j)}\;+
\qquad\qquad\qquad\qquad\qquad\qquad\\
(z_2\!-\!1)(z_3\!-\!1)(z_4z_1\!-\!z_4\!+\!1)(z_1z_2\!-\!z_1\!+\!1)(z_3z_4\!-\!z_3\!+\!1)(z_2z_4)^{n-1}Z^\adj_{n-1}(1,1;\tfrac{1}{z_4},z_1)\;+\\
(z_3\!-\!1)(z_4\!-\!1)(z_1z_2\!-\!z_1\!+\!1)(z_2z_3\!-\!z_2\!+\!1)(z_4z_1\!-\!z_4\!+\!1)(z_1z_3)^{n-1}Z^\adj_{n-1}(1,1;\tfrac{1}{z_1},z_2)\;+\\
(z_4\!-\!1)(z_1\!-\!1)(z_2z_3\!-\!z_2\!+\!1)(z_3z_4\!-\!z_3\!+\!1)(z_1z_2\!-\!z_1\!+\!1)(z_2z_4)^{n-1}Z^\adj_{n-1}(1,1;\tfrac{1}{z_2},z_3)\;+\\
(z_1\!-\!1)(z_2\!-\!1)(z_3z_4\!-\!z_3\!+\!1)(z_4z_1\!-\!z_4\!+\!1)(z_2z_3\!-\!z_2\!+\!1)(z_1z_3)^{n-1}Z^\adj_{n-1}(1,1;\tfrac{1}{z_3},z_4)\;-\\
\\
(z_1\!-\!1)(z_2\!-\!1)(z_3\!-\!1)(z_4\!-\!1)\bigl((z_1z_2\!-\!z_1\!+\!1)(z_3z_4\!-\!z_3\!+\!1)(z_2z_4)^{n-1}\;+\\
(z_2z_3\!-\!z_2\!+\!1)(z_4z_1\!-\!z_4\!+\!1)(z_1z_3)^{n-1}\bigr)\A_{n-2}.\end{multline}

Setting $z_2=1$ in~\eqref{unwquad} (and then relabelling $z_3$ as $z_2$ and $z_4$ as~$z_3$), it follows that the
triply-refined ASM generating function at $x=y=1$ is given by
\begin{multline}\label{unwtrip}
(z_1z_3\!-\!z_3\!+\!1)(z_2z_3\!-\!z_2\!+\!1)z_3^{n-1}Z^\tri_n(1,1;z_1,z_2,\tfrac{1}{z_3})=\\
\frac{z_1\,z_3\,\det_{1\le i,j\le3}
\bigl(z_i^{\;j-1}\,(z_i\!-\!1)^{3-j}\,Z_{n-j+1}(1,1;z_i)\bigr)}{\rule{0ex}{2.3ex}\A_{n-1}\,\A_{n-2}\,\prod_{1\le i<j\le3}(z_i-z_j)}\;+
\qquad\qquad\qquad\qquad\qquad\\
(z_2\!-\!1)(z_3\!-\!1)(z_1z_3\!-\!z_3\!+\!1)z_1z_2^{n-1}Z_{n-1}(1,1;z_1)\;+\\
(z_1\!-\!1)(z_3\!-\!1)(z_2z_3\!-\!z_2\!+\!1)z_3^{n-1}Z_{n-1}(1,1;z_2).\end{multline}

Note that the identities~\eqref{Zoppid} and~\eqref{Zoppadjid}, satisfied by the doubly-refined ASM generating functions at $x=y=1$,
can be regarded as special cases of the identity~\eqref{unwtrip} satisfied by the triply-refined ASM generating function.
More specifically, setting $z_3=1$ in~\eqref{unwtrip} gives~\eqref{Zoppid}, while setting~$z_2=1$
in~\eqref{unwtrip}, and using~\eqref{Zoppid}, gives~\eqref{Zoppadjid}.

Note also that, by using~\eqref{Zoppadjid} to replace each case of an adjacent-boundary doubly-refined ASM generating function
in~\eqref{unwquad} by an opposite-boundary doubly-refined ASM generating function,~\eqref{unwquad} can be
restated as
\begin{multline}\label{unwquadopp}
(z_4z_1\!-\!z_4\!+\!1)(z_1z_2\!-\!z_1\!+\!1)(z_2z_3\!-\!z_2\!+\!1)(z_3z_4\!-\!z_3\!+\!1)\,\Z^\qua_n(1,1;z_1,z_2,z_3,z_4)=\\
\frac{z_1\,z_2\,z_3\,z_4\,\det_{1\le i,j\le 4}
\bigl(z_i^{\;j-1}\,(z_i\!-\!1)^{4-j}\,Z_{n-j+1}(1,1;z_i)\bigr)}{\rule{0ex}{2.3ex}\A_{n-1}\,\A_{n-2}\,\A_{n-3}\,\prod_{1\le i<j\le 4}(z_i-z_j)}\;+
\qquad\qquad\qquad\qquad\qquad\qquad\\
(z_2\!-\!1)(z_3\!-\!1)(z_1z_2\!-\!z_1\!+\!1)(z_3z_4\!-\!z_3\!+\!1)z_4z_1z_2^{n-1}Z^\opp_{n-1}(1,1;z_4,z_1)\;+\\
(z_3\!-\!1)(z_4\!-\!1)(z_2z_3\!-\!z_2\!+\!1)(z_4z_1\!-\!z_4\!+\!1)z_1z_2z_3^{n-1}Z^\opp_{n-1}(1,1;z_1,z_2)\;+\\
(z_4\!-\!1)(z_1\!-\!1)(z_3z_4\!-\!z_3\!+\!1)(z_1z_2\!-\!z_1\!+\!1)z_2z_3z_4^{n-1}Z^\opp_{n-1}(1,1;z_2,z_3)\;+\\
(z_1\!-\!1)(z_2\!-\!1)(z_4z_1\!-\!z_4\!+\!1)(z_2z_3\!-\!z_2\!+\!1)z_3z_4z_1^{n-1}Z^\opp_{n-1}(1,1;z_3,z_4)\;+\\
\shoveleft{\quad(z_1\!-\!1)(z_2\!-\!1)(z_3\!-\!1)(z_4\!-\!1)\bigl((z_1z_2\!-\!z_1\!+\!1)(z_3z_4\!-\!z_3\!+\!1)(z_2z_4)^{n-1}\;+}\\
(z_2z_3\!-\!z_2\!+\!1)(z_4z_1\!-\!z_4\!+\!1)(z_1z_3)^{n-1}\bigr)\A_{n-2}.\end{multline}

The first terms (i.e., the determinant terms) on the RHS of~\eqref{unwquad}--\eqref{unwquadopp}
are proportional to certain cases of a function which will be defined in~\eqref{Zmult}.
Specifically, using~\eqref{Zmult},
the first terms on the RHS of~\eqref{unwquad} (or~\eqref{unwquadopp}) and~\eqref{unwtrip}
are $z_1z_2z_3z_4\,X_n(1,1;z_1,z_2,z_3,z_4)$ and $z_1z_2z_3\,X_n(1,1;z_1,z_2,z_3)$, respectively.
It will also be seen, in~\eqref{schurgen}, that the function~\eqref{Zmult} at $x=y=1$, and hence the first terms
on the RHS of~\eqref{unwquad}--\eqref{unwquadopp}, are related to a certain Schur function.

Further expressions for $\Z^\qua_n(1,1;z_1,z_2,z_3,z_4)$ and $Z^\tri_n(1,1;z_1,z_2,\tfrac{1}{z_3})$,
which differ from \eqref{unwtrip}--\eqref{unwquadopp} in the first terms on each RHS,
will be obtained in Corollaries~\ref{cor9} and \ref{cor10}.

The origins of the results given in this section will now be outlined.

The product formula~\eqref{ASM} was conjectured by Mills, Robbins and Rumsey~\cite[Conj.~1]{MilRobRum82,MilRobRum83},
and first proved by Zeilberger~\cite{Zei96a} and, shortly thereafter, but using a different method, by Kuperberg~\cite{Kup96}.
The product formula~\eqref{RASM} was first proved by Zeilberger~\cite{Zei96b},
and confirms the validity of further conjectures of Mills, Robbins and Rumsey~\cite[Conj.~2]{MilRobRum82,MilRobRum83}.

Alternative proofs of~\eqref{ASM}--\eqref{RASM} have been given by
Colomo and Pronko~\cite[Sec.~5.3]{ColPro05a},~\cite[Sec.~4.2]{ColPro06}, Fischer~\cite{Fis07}, and Stroganov~\cite[Sec.~4]{Str06}.
(See also Razumov and Stroganov~\cite[Sec.~2]{RazStr04b},~\cite[Sec.~2]{RazStr09} for additional details related to the third of these proofs.)

The formulae~\eqref{ASM}--\eqref{RASM} also follow from certain other known results.  For example,
as already indicated,~\eqref{ASM} can be obtained from the result~\eqref{SSYT} and the hook-content formula.
Alternatively,~\eqref{ASM} can be obtained as a special case of a result of Rosengren~\cite[Cor.~8.4]{Ros09}.
Also,~\eqref{RASM}
can be obtained by combining a result of Behrend, Di Francesco and Zinn-Justin~\cite[Thm.~1]{BehDifZin12}
that~$\A_{n,k}$ is the number of descending plane partitions with each part at most~$n$ and exactly~$k$ parts equal to~$n$,
with a result of Mills, Robbins and Rumsey~\cite[Sec.~5]{MilRobRum82} which gives a product formula for the number of such objects.

Among these various derivations of~\eqref{ASM}--\eqref{RASM}, all make essential use of the
Izergin--Korepin formula~\cite[Eq.~(5)]{Ize87} for the partition function of the six-vertex model with DWBC,
or related properties of integrability,
except for
Zeilberger's proof~\cite{Zei96a} of~\eqref{ASM}, which uses
a result of Andrews~\cite{And94} for the number
of totally symmetric self-complementary plane partitions in a $2n\times2n\times2n$ box,
and Fischer's proof~\cite{Fis07}, which uses
an operator formula obtained by Fischer~\cite{Fis06,Fis10}
(see also Colomo and Pronko~\cite[Eq.~(3.3)]{ColPro12}).

The hypergeometric function
expression~\eqref{hyp} is given by Colomo and Pronko \cite[Eq. (2.16)]{ColPro05c}, \cite[Eq. (5.43)]{ColPro05a},
\cite[Eq.~(4.19)]{ColPro06},
and the differential equation~\eqref{diff} is given by Stroganov~\cite[Eq.~(26)]{Str06}.

The relation~\eqref{RASMlinear} was obtained by Fischer~\cite[Sec.~3]{Fis07}, with a corresponding relation for
descending plane partitions, containing a further parameter associated with the sum of the parts of a
descending plane partition, having been obtained previously by Mills, Robbins and Rumsey~\cite[Sec.~5]{MilRobRum82}.
It was also shown by Fischer~\cite[Sec.~4]{Fis07} that the formula~\eqref{RASM} follows from~\eqref{RASMlinear},
the first three equations of~\eqref{A1}, and $\A_{1,0}=1$.

The relations~\eqref{oppRASMrec}--\eqref{Zoppadjid}
for the opposite-boundary and adjacent-boundary doubly-refined ASM numbers
were first obtained by Stroganov~\cite[Sec.~5]{Str06}.  Furthermore,~\eqref{adjRASMrec} is a special
case of a formula obtained by Fischer~\cite[Thm.~1]{Fis12}, and~\eqref{Zoppid} is a special case of
a relation,~\eqref{propeq1}, which will be discussed in Section~\ref{prev1}.  As indicated elsewhere in this section,~\eqref{Zoppid} and~\eqref{Zoppadjid}
are also special cases of~\eqref{unwquad}.

The formula~\eqref{oppRASM} follows easily from~\eqref{oppRASMrec}.
The formula~\eqref{adjRASM} was first obtained by Fischer~\cite[p.~570]{Fis12}, and
can be derived by dividing both sides of~\eqref{Zoppadjid} by
$z_1+z_2-1$ and equating coefficients of $z_1^{k_1}z_2^{k_2}$ on each side.

The Schur function expression~\eqref{schur} was obtained by Di Francesco and Zinn-Justin~\cite[Eqs.~(2.2) \& (2.4)]{DifZin05},
using a result of Okada~\cite[Thm.~2.4(1), second equation]{Oka06}.
A generalization of~\eqref{schur} will be given in~\eqref{schurgen}, and a derivation of~\eqref{schurgen}
will be given in Section~\ref{add2}.

The relation~\eqref{SSYT} was first obtained by Okada~\cite[Thm.~1.2 (A1)]{Oka06}.

The identity~\eqref{BCS} was obtained by Biane, Cantini and Sportiello~\cite[Thm.~1]{BiaCanSpo12}, by
combining~\eqref{schur} with a certain Schur function identity, also obtained by
Biane, Cantini and Sportiello~\cite[Thm.~2]{BiaCanSpo12}.

Some further aspects of the opposite-boundary doubly-refined ASM numbers are discussed by Fischer~\cite[Sec.~6]{Fis10}.

Identities for the quadruply- and triply-refined ASM generating functions at $x=y=1$,
which are essentially equivalent to~\eqref{unwquad}--\eqref{unwtrip},
were obtained recently by Ayyer and Romik~\cite[Thms.~1--3]{AyyRom13}.
The forms of \eqref{unwquad}--\eqref{unwtrip} given here were observed by Colomo~\cite{Col12},
and can be derived by combining a result
which will be given in~\eqref{ZmultASM}, with a result
that the partition function of the six-vertex model with DWBC
is symmetric in all of its spectral parameters at its so-called combinatorial point.
A derivation of~\eqref{unwquad} using this approach will be given in Section~\ref{add2}.
(Roughly speaking, this derivation involves associating spectral parameters~$t_1$,~$t_2$,~$t_3$ and~$t_4$ with
the first row, last column, last row and first column, respectively, of an ASM.  The symmetry of the partition
function in all of its spectral parameters then enables $t_2$ and~$t_4$ to be associated instead with
rows, so that the result~\eqref{ZmultASM},
which involves parameters associated with several rows, can be applied.)

Note that, since setting certain boundary parameters in~\eqref{unwquad} to~1
gives~\eqref{Zoppid},~\eqref{Zoppadjid} and~\eqref{unwtrip}, Section~\ref{add2}
also provides derivations of these other identities.
(For the identity~\eqref{Zoppadjid}, such a derivation essentially depends
only the symmetry of the partition function in all of its spectral parameters, and not on~\eqref{ZmultASM}.)

In the versions of~\eqref{unwquad}--\eqref{unwtrip} obtained by Ayyer and Romik~\cite[Thms.~1--3]{AyyRom13},
slightly different quadruply-refined and triply-refined ASM generating functions are used,
and matrices appear which are related by transposition and column operations to those in the first terms on the RHS of~\eqref{unwquad}--\eqref{unwtrip}.
In particular, the quadruply-refined generating function used there is
$Z'_n(x,y;z_1,z_2,z_3,z_4)=\rule[-2.7ex]{0ex}{2.7ex}\sum_{\substack{A\in\ASM(n)\\A_{11}=A_{1n}=A_{nn}=A_{n1}=0}}\!x^{\nu(A)}\,y^{\mu(A)}\,
z_1^{\rT(A)-1}\,z_2^{n-2-\rR(A)}\,z_3^{\rB(A)-1}\,z_4^{n-2-\rL(A)}$
(at $x=y=1$), which
can be shown straightforwardly to be related
to generating functions used here by
\begin{multline}
z_1\,z_2\,z_3\,z_4\,Z'_n(x,y;z_1,z_2,z_3,z_4)=\\
\Z^\qua_n(x,y;z_1,z_2,z_3,z_4)-(z_2z_4)^{n-1}\,Z^\adj_{n-1}(x,y;\tfrac{1}{z_4},z_1)-(z_2z_4)^{n-1}\,Z^\adj_{n-1}(x,y;\tfrac{1}{z_2},z_3)\;-\\
z_1z_2^{\;n-2}(xz_3)^{n-1}\,\Z^\adj_{n-1}(x,y;z_1,\tfrac{1}{z_2})-z_3z_4^{\;n-2}(xz_1)^{n-1}\,\Z^\adj_{n-1}(x,y;z_3,\tfrac{1}{z_4})\;+\\
\bigl((z_2z_4)^{n-1}+x^{2n-3}(z_1z_3)^{n-1}\bigr)\,Z_{n-2}(x,y).
\end{multline}
The matrices used in some of the formulae of Ayyer and Romik~\cite[Eqs.~(1.8) \&~(1.11)]{AyyRom13}
involve functions $\gamma_n(z)$ and $\delta_n(z)$ (see Ayyer and Romik~\cite[Eq.~(1.7)]{AyyRom13} for definitions), which are
related to the functions used in the matrices in the first terms on the RHS of~\eqref{unwquad}--\eqref{unwtrip} by
\begin{align}\notag
\gamma_n(z)&=\tfrac{4(n-1)!\,(3n-5)!}{(2n-3)!\,(2n-4)!}\,z^2\,Z_{n-2}(1,1;z)-\tfrac{4(2n-3)!\,(2n-5)!}{(n-3)!\,(3n-5)!}\,(z\!-\!1)^2\,Z_n(1,1;z),\\
\label{AR}\delta_n(z)&=\tfrac{4(n-1)!\,(n-4)!\,(3n-5)!\,(3n-8)!}{((2n-4)!)^2\,(2n-5)!\,(2n-7)!}\,z^3\,Z_{n-3}(1,1;z)-
\tfrac{9(9n^2-30n+20)}{2n-5}\,z(z\!-\!1)^2\,Z_{n-1}(1,1;z).\end{align}
These relations can be confirmed (perhaps with the assistance of a computer to verify certain underlying identities)
by applying the definitions of $\gamma_n(z)$, $\delta_n(z)$ and $Z_n(1,1;z)$, and using the product formula~\eqref{RASM} for the
singly-refined ASM numbers which appear in these functions. The fact that the matrices used by Ayyer and Romik
are related to those used here by (transposition and) column operations follows immediately from~\eqref{AR}.
See also Ayyer and Romik~\cite[Thm.~3 \& App.~A]{AyyRom13} for a discussion of this matter.

\subsection{Arbitrary bulk parameters $x$ and $y$}\label{prev1}
The case in which the bulk parameters $x$ and~$y$ are arbitrary is of primary interest in this paper,
and several new results will be presented in Section~\ref{newres}.  In this section, the previously-known results for this case,
which involve the unrefined, singly-refined and opposite-boundary doubly-refined ASM generating functions,
are reviewed.

A determinant formula for the opposite-boundary doubly-refined ASM generating function is
\begin{equation}\label{Zdetdoub}Z^\opp_n(x,y;z_1,z_2)=\det_{0\le i,j\le n-1}\bigl(K_n(x,y;z_1,z_2)_{ij}\bigr),\end{equation}
where
\begin{multline}\label{K}K_n(x,y;z_1,z_2)_{ij}=\\
-\delta_{i,j+1}+\begin{cases}
\sum_{k=0}^{\min(i,j+1)}\binom{i-1}{i-k}\binom{j+1}{k}x^ky^{i-k},&j\le n-3,\\[1.5mm]
\sum_{k=0}^i\sum_{l=0}^k\binom{i-1}{i-k}\binom{n-l-2}{k-l}x^ky^{i-k}z_2^{l+1},&j=n-2,\\[1.5mm]
\sum_{k=0}^i\sum_{l=0}^k\sum_{m=0}^l\binom{i-1}{i-k}\binom{n-l-2}{k-l}x^ky^{i-k}z_1^mz_2^{l-m},&j=n-1.
\end{cases}\end{multline}

Setting $z_2=1$ or $z_1=z_2=1$ in~\eqref{Zdetdoub}--\eqref{K},
and using standard binomial coefficient identities, it follows that
determinant formulae for the singly-refined and unrefined ASM generating functions are
\begin{align}\label{Zdetsing}Z_n(x,y;z)&=\det_{0\le i,j\le n-1}\bigl(K_n(x,y;z)_{ij}\bigr),\\
\label{Zdetun}Z_n(x,y)&=\det_{0\le i,j\le n-1}\bigl(K_n(x,y)_{ij}\bigr),\end{align}
where
\begin{align}
\notag K_n(x,y;z)_{ij}&=K_n(x,y;z,1)_{ij}\\
&=-\delta_{i,j+1}+\begin{cases}
\sum_{k=0}^{\min(i,j+1)}\binom{i-1}{i-k}\binom{j+1}{k}x^ky^{i-k},&j\le n-2,\\
\sum_{k=0}^i\sum_{l=0}^k\binom{i-1}{i-k}\binom{n-l-1}{k-l}x^ky^{i-k}z^l,&j=n-1,\end{cases}\\
\label{K1}K_n(x,y)_{ij}&=K_n(x,y;1,1)_{ij}=\textstyle-\delta_{i,j+1}+\sum_{k=0}^{\min(i,j+1)}\binom{i-1}{i-k}\binom{j+1}{k}x^ky^{i-k}.\end{align}

There are further simplifications if $x=y=1$. For example,
\begin{equation}K_n(1,1)_{ij}=-\delta_{i,j+1}+\tbinom{i+j}{i}.\end{equation}

The opposite-boundary doubly-refined ASM generating function satisfies
\begin{multline}\label{propeq2}
(z_1\!-\!z_2)\,(z_3\!-\!z_4)\,Z^\opp_n(x,y;z_1,z_2)\,Z^\opp_n(x,y;z_3,z_4)\;-\\
(z_1\!-\!z_3)\,(z_2\!-\!z_4)\,Z^\opp_n(x,y;z_1,z_3)\,Z^\opp_n(x,y;z_2,z_4)\;+\\
(z_1\!-\!z_4)\,(z_2\!-\!z_3)\,Z^\opp_n(x,y;z_1,z_4)\,Z^\opp_n(x,y;z_2,z_3)=0,
\end{multline}
and it can be expressed in terms of singly-refined and unrefined ASM generating functions as
\begin{multline}\label{propeq1}
(z_1\!-\!z_2)\,Z^\opp_n(x,y;z_1,z_2)\,Z_{n-1}(x,y)=(z_1\!-\!1)\,z_2\,Z_n(x,y;z_1)\,Z_{n-1}(x,y;z_2)\;-\\
z_1\,(z_2\!-\!1)\,Z_{n-1}(x,y;z_1)\,Z_n(x,y;z_2).\end{multline}

The two previous identities are essentially equivalent,
since~\eqref{propeq1} can be obtained from~\eqref{propeq2} by setting $z_3=1$ and $z_4=0$, and then
applying relations
from~\eqref{1} and~\eqref{0}, while~\eqref{propeq2} can be obtained from~\eqref{propeq1} by
expressing each of the six cases of $(z_i-z_j)Z^\opp_n(x,y;z_i,z_j)$
in terms of singly-refined and unrefined ASM generating functions,
and then checking that the resulting expression for the LHS of~\eqref{propeq2} vanishes.

The formula~\eqref{Zdetdoub} was derived by Behrend, Di Francesco and Zinn-Justin~\cite[Eqs.~(21)--(22)]{BehDifZin13},
as part of the proof of an equality~\cite[Thm.~1]{BehDifZin13} between the opposite-boundary
doubly-refined ASM generating function and a four-parameter generating function for descending plane partitions.
(See also Section~\ref{DPPs}.)
The formula~\eqref{Zdetsing} was derived by Behrend, Di Francesco and Zinn-Justin~\cite[Eqs.~(97)--(98)]{BehDifZin12},
prior to the derivation of~\eqref{Zdetdoub}, as part of the proof of an equality~\cite[Thm.~1]{BehDifZin12}
(conjectured previously by Mills, Robbins and Rumsey~\cite[Conj.~3]{MilRobRum83}) between the singly-refined
ASM generating function and a three-parameter generating function for descending plane partitions.

For an alternative version
of~\eqref{Zdetdoub}, involving a transformation of the matrix~$K_n(x,y;z_1,z_2)$,
see Behrend, Di Francesco and Zinn-Justin~\cite[Eqs.~(65)--(66)]{BehDifZin13},
and for alternative versions
of~\eqref{Zdetsing} or~\eqref{Zdetun}, involving transformations of the matrices~$K_n(x,y;z)$ or~$K_n(x,y)$,
see Behrend, Di Francesco and Zinn-Justin~\cite[Props.~1 \& 4, Eqs.~(65), (66), (87) \& (88), \& Sec.~4.1]{BehDifZin12}.

For additional information on determinants closely related to those of~\eqref{Zdetsing}--\eqref{Zdetun},
including formulae for their evaluation or factorization in certain special cases, see, for example,
Andrews~\cite{And79},
Andrews and Stanton~\cite{AndSta98},
Ciucu, Eisenk{\"o}lbl, Krattenthaler and Zare~\cite[Thms.~10--13]{CiuEisKraZar01},
Ciucu and Krattenthaler~\cite{CiuKra00},
Colomo and Pronko~\cite[Eqs.~(23)--(24)]{ColPro03},~\cite[Eqs.~(4.3)--(4.7)]{ColPro04},
de Gier~\cite[Sec.~2.1]{Deg09},
Gessel and Xin~\cite[Sec.~5]{GesXin06},
Krattenthaler~\cite[e.g., Thms.~25--37]{Kra99},~\cite[Sec.~5.5]{Kra05},
Lalonde~\cite[Thm.~3.1]{Lal02},
Mills, Robbins and Rumsey~\cite{MilRobRum82},~\cite[p.~346]{MilRobRum83},~\cite[Secs.~3--4]{MilRobRum87},
Robbins~\cite[Sec.~2]{Rob00},
and Rosengren~\cite{Ros12}.

The case of~\eqref{propeq1} with $x=y=1$ (as given in~\eqref{Zoppid}) was first obtained by Stroganov~\cite[Sec.~5]{Str06},
and this relation for arbitrary $x$ and $y$ was first obtained,
as an identity involving one- and two-point boundary correlation functions for
the six-vertex model with DWBC, by Colomo and Pronko~\cite[Eq.~(5.32)]{ColPro05b},~\cite[Eq.~(3.32)]{ColPro06}.
An alternative proof of~\eqref{propeq1}
was given by Behrend, Di Francesco and Zinn-Justin~\cite[Sec.~5]{BehDifZin13}.
This identity will also be obtained from Theorem~\ref{thm} of this paper, as Corollary~\ref{cor4},
and it will be found to be a special case of the identity~\eqref{ZmultASM} given in Section~\ref{gensect}.

\subsection{ASM enumeration involving statistics associated with several rows or several columns}\label{gensect}
In this section, a generating function involving statistics associated with several rows (or several columns) of ASMs
is defined, and (in~\eqref{ZmultASM}) a result is stated which,
for the case of particular assignments of certain parameters, provides a determinantal expression for this generating
function in terms of singly-refined and unrefined ASM generating functions.  A derivation of this result will be given in Section~\ref{add1}.

For $A\in\ASM(n)$, define ASM statistics, associated with row $i$ or column~$j$ of $A$, as
\begin{align}\notag\nu^{\row\,i}(A)&=\textstyle\bigl|\bigl\{j\in\{1,\ldots,n\}\,\big|\,\sum_{i'=1}^{i-1}\!A_{i'j}=\sum_{j'=1}^j\!A_{ij'}\bigr\}\bigr|,\\
\notag\nu^{\col\,j}(A)&=\textstyle\bigl|\bigl\{i\in\{1,\ldots,n\}\,\big|\,\sum_{i'=1}^{i-1}\!A_{i'j}=\sum_{j'=1}^j\!A_{ij'}\bigr\}\bigr|,\\
\notag\mu^{\row\,i}(A)&=\text{number of $-1$'s in row $i$ of }A,\\
\label{genstat}\mu^{\col\,j}(A)&=\text{number of $-1$'s in column $j$ of }A.\end{align}

It follows from the defining properties of ASMs that if, for an ASM $A$,
$\sum_{i'=1}^{i-1}\!A_{i'j}=\sum_{j'=1}^j\!A_{ij'}$ (or equivalently $\sum_{i'=1}^i\!A_{i'j}=\sum_{j'=1}^{j-1}\!A_{ij'}$),
then $A_{ij}=0$. Hence,~$\nu^{\row\,i}(A)$ and $\nu^{\col\,j}(A)$ can be regarded as
the numbers of certain 0's in row $i$ and column~$j$, respectively, of~$A$.
It will be seen in Section~\ref{ASMbijsect} that, in terms of the configuration of the six-vertex model with DWBC which
corresponds to $A$,~$\nu^{\row\,i}(A)$ and $\nu^{\col\,j}(A)$ are simply the numbers of vertex configurations
of types~(1) and~(2) in row~$i$ and column~$j$, respectively, of the grid.

It can also be seen that the statistics $\nu^{\row\,i}(A)$ and
$\nu^{\col\,j}(A)$ can be written in various other ways, for example
as
\begin{align}\notag\nu^{\row\,i}(A)&=\textstyle
\sum_{j=1}^n\bigl(\sum_{i'=1}^{i-1}\sum_{j'=1}^jA_{i'j}A_{ij'}+\sum_{i'=i}^n\sum_{j'=j+1}^nA_{i'j}A_{ij'}\bigr),\\
\label{genstat2}\nu^{\col\,j}(A)&=\textstyle
\sum_{i=1}^n\bigl(\sum_{i'=1}^{i-1}\sum_{j'=1}^jA_{i'j}A_{ij'}+\sum_{i'=i}^n\sum_{j'=j+1}^nA_{i'j}A_{ij'}\bigr).\end{align}

It can be checked easily that, for any $A\in\ASM(n)$,
the statistics of~\eqref{numuA}--\eqref{rho} and those of~\eqref{genstat} are related by
\begin{gather}\notag\textstyle\sum_{i=1}^n\nu^{\row\,i}(A)=\sum_{j=1}^n\nu^{\col\,j}(A)=2\nu(A),\\
\notag\textstyle\nu^{\row\,1}(A)=\rT(A),\quad\nu^{\col\,n}(A)=\rR(A),\quad\nu^{\row\,n}(A)=\rB(A),\quad\nu^{\col\,1}(A)=\rL(A),\\
\notag\textstyle\sum_{i=2}^{n-1}\mu^{\row\,i}(A)=\sum_{j=2}^{n-1}\mu^{\col\,j}(A)=\mu(A),\\
\label{multstat}\mu^{\row\,1}(A)=\mu^{\col\,n}(A)=\mu^{\row\,n}(A)=\mu^{\col\,1}(A)=0,
\end{gather}
and that the statistics of~\eqref{genstat} behave under the operations of transposition or anticlockwise quarter-turn rotation
of ASMs according to
\begin{align}\notag\nu^{\row\,i}(A)&=\nu^{\col\,i}(A^T)=n-1-\nu^{\col\,i}(A^Q)-2\mu^{\row\,i}(A),\\
\label{genTQ}\mu^{\row\,i}(A)&=\mu^{\col\,i}(A^T)=\mu^{\col\,i}(A^Q),\end{align}
using the same notation as in~\eqref{TQ}.

For any $0\le m\le n$ and $1\le k_1<\ldots<k_m\le n$, define
an ASM generating function associated with rows
(or, as will be seen in~\eqref{gensymm}, columns) $k_1,\ldots,k_m$  as
\begin{multline}\label{ZrowASM}Z^{k_1,\ldots,k_m}_n(x,y;z_1,\ldots,z_m;w_1,\ldots,w_m)=\\
\sum_{A\in\ASM(n)}x^{\nu(A)}\,y^{\mu(A)-\sum_{i=1}^m\mu^{\row\,k_i}(A)}\,
\prod_{i=1}^m z_i^{\;\nu^{\row\,k_i}(A)}\,w_i^{\;\mu^{\row\,k_i}(A)},\end{multline}
for indeterminates $x$, $y$, $z_1,\ldots,z_m$ and $w_1,\ldots,w_m$.

It can be seen that
\begin{multline}\label{mult1}
Z^{k_1,\ldots,k_m}_n(x,y;z_1,\ldots,z_{i-1},1,z_{i+1},\ldots,z_m;w_1,\ldots,w_{i-1},y,w_{i+1},\ldots,w_m)=\\
Z^{k_1,\ldots,k_{i-1},k_{i+1},\ldots,k_m}_n(x,y;z_1,\ldots,z_{i-1},z_{i+1},\ldots,z_m;w_1,\ldots,w_{i-1},w_{i+1},\ldots,w_m),\end{multline}
and that $Z^{k_1,\ldots,k_m}_n(x,y;z_1,\ldots,z_m;w_1,\ldots,w_m)$ is
the unrefined ASM generating function $Z_n(x,y)$ for $m=0$, the singly-refined ASM generating function
$Z_n(x,y;z_1)$ for $m=1$ and $k_1=1$ or $k_1=n$, and the opposite-boundary doubly-refined ASM generating function
$Z^\opp_n(x,y;z_1,z_2)$ for $m=2$, $k_1=1$ and $k_2=n$.

Using~\eqref{genTQ}, it also follows that
\begin{multline}\label{gensymm}Z^{k_1,\ldots,k_m}_n(x,y;z_1,\ldots,z_m;w_1,\ldots,w_m)=\\
\sum_{A\in\ASM(n)}x^{\nu(A)}\,y^{\mu(A)-\sum_{j=1}^m\mu^{\col\,k_j}(A)}\,
\prod_{j=1}^m z_j^{\;\nu^{\col\,k_j}(A)}\,w_j^{\;\mu^{\col\,k_j}(A)}=\\
x^{n(n-1)/2}\,(z_1\ldots z_m)^{n-1}\,Z^{k_1,\ldots,k_m}_n\bigl(\tfrac{1}{x},\tfrac{y}{x};\,\tfrac{1}{z_1},\ldots,\tfrac{1}{z_m};\,
\tfrac{w_1}{x z_1^{\:2}},\ldots,\tfrac{w_m}{x z_m^{\;2}}\bigr).\end{multline}

Now define, for $0\le m\le n$, a further function
\begin{equation}\label{Zmult}X_n(x,y;z_1,\ldots,z_m)=\begin{cases}
Z_n(x,y),&m=0,\\
\displaystyle\frac{\det_{1\le i,j\le m}
\bigl(z_i^{\;j-1}\,(z_i\!-\!1)^{m-j}\,Z_{n-j+1}(x,y;z_i)\bigr)}{\prod_{1\le i<j\le m}(z_i\!-\!z_j)\:\prod_{i=1}^{m-1}Z_{n-i}(x,y)},&
1\le m\le n.\end{cases}
\end{equation}
This definition is based on a definition of Colomo and Pronko~\cite[Eq.~(6.6)]{ColPro08},~\cite[Eq.~(4.18)]{ColPro12} of a closely-related function
for the six-vertex model with DWBC.

Some properties of $X_n(x,y;z_1,\ldots,z_m)$, which follow straightforwardly from~\eqref{Zmult}
and the latter equations of~\eqref{1},~\eqref{Zsymm} and~\eqref{0}, are
\begin{gather}
\notag X_n(x,y;z_1,\ldots,z_m)\text{ is symmetric in }z_1,\ldots,z_m,\\
\notag X_n(x,y;z_1,\ldots,z_m)=x^{n(n-1)/2}\,(z_1\ldots z_m)^{n-1}\,X_n\bigl(\tfrac{1}{x},\tfrac{y}{x};\tfrac{1}{z_1},\ldots,\tfrac{1}{z_m}\bigr),\\
\notag X_n(x,y;z_1,\ldots,z_{i-1},1,z_{i+1},\ldots,z_m)=X_n(x,y;z_1,\ldots,z_{i-1},z_{i+1},\ldots,z_m),\\
\notag X_n(x,y;z_1,\ldots,z_{i-1},0,z_{i+1},\ldots,z_m)=X_{n-1}(x,y;z_1,\ldots,z_{i-1},z_{i+1},\ldots,z_m),\\
\label{Zmultsymm}X_n(x,y;z)=Z_n(x,y;z).
\end{gather}
Certain further properties can be found by combining~\eqref{Zmult} with
general identities for minors of a matrix.
For example, for $1\le m\le n$, and indeterminates $z_1,\ldots,z_m$, $u_1,\ldots,u_{m-1}$ and $v_1,\ldots,v_{m-1}$, the identity
\begin{multline}\label{Xid}\textstyle\prod_{1\le i<j\le m}(z_i\!-\!z_j)\:\prod_{1\le i\le j\le m-1}(u_i\!-\!v_j)\;\times\\
\textstyle X_n(x,y;z_1,\ldots,z_m)\:\prod_{i=1}^{m-1}X_n(x,y;u_1,\ldots,u_{m-i},v_{m-i},\ldots,v_{m-1})=\\
\det_{1\le i,j\le m}\Bigl(\textstyle\prod_{k=1}^{m-j}(z_i\!-\!u_k)\,\prod_{k=m-j+1}^{m-1}(z_i\!-\!v_k)
\:X_n(x,y;z_i,u_1,\ldots,u_{m-j},v_{m-j+1},\ldots,v_{m-1})\Bigr)\end{multline}
is satisfied.
This result can be obtained from~\eqref{Zmult} by applying a certain identity for minors of a matrix, which will be given in~\eqref{PolBazin},
to the $(3m-2)\times m$ matrix $N$ with entries $N_{ij}=f_j(z_i)$ for $1\le i\le m$, $N_{ij}=f_j(u_{i-m})$ for $m+1\le i\le 2m-1$,
and $N_{ij}=f_j(v_{i-2m+1})$ for $2m\le i\le 3m-2$, where $f_j(z)=z^{j-1}(z-1)^{m-j}Z_{n-j+1}(x,y;z)$ (although the
actual form of the function $f_j(z)$ in this derivation is immaterial).

Note that the $1\le m\le n$ case of~\eqref{Zmult} can be retrieved from~\eqref{Xid}, using only the last three properties in~\eqref{Zmultsymm}
and the $m=0$ case of~\eqref{Zmult}, by setting $u_1=\ldots=u_{m-1}=1$ and $v_1=\ldots=v_{m-1}=0$.

The main result of this section is that, for any $0\le m\le n$ and $1\le k_1<\ldots<k_m\le n$,
\begin{multline}\label{ZmultASM}Z^{k_1,\ldots,k_m}_n(x,y;z_1,\ldots,z_m;xz_1^{\:2}\!+\!(y\!-\!x\!-\!1)z_1\!+\!1,\ldots,xz_m^{\:2}\!+\!
(y\!-\!x\!-\!1)z_m\!+\!1)=\\
X_n(x,y;z_1,\ldots,z_m).\end{multline}
A derivation of~\eqref{ZmultASM}, based on a result of Colomo and Pronko~\cite[Eq.~(6.8)]{ColPro08},~\cite[Eq.~(A.13)]{ColPro10},
will be given in Section~\ref{add1}.  It will also be shown in Section~\ref{add1} that~\eqref{ZmultASM} can be derived by
combining a general identity for minors of a matrix with the
Izergin--Korepin formula for the partition function of the six-vertex model with DWBC.

Note that, using~\eqref{mult1} and the third property of~\eqref{Zmultsymm}, the cases of~\eqref{ZmultASM} with $m<n$
all follow from the case $m=n$.

It can now be seen, using~\eqref{ZrowASM} and~\eqref{ZmultASM}, that the function~\eqref{Zmult}
is a polynomial in $x$, $y$ and $z_1,\ldots,z_m$ with integer coefficients.

It follows from~\eqref{ZmultASM} that, for fixed $n$ and $m$, the LHS of~\eqref{ZmultASM} is independent of $k_1,\ldots,k_m$.
This implies, for example, that the
singly-refined and opposite-boundary doubly-refined ASM generating functions can be written, for any $1\le k\le n$
and $1\le k_1<k_2\le n$, as
\begin{align}\notag Z_n(x,y;z)&=Z_n^k(x,y;z;xz^2\!+\!(y\!-\!x\!-\!1)z\!+\!1),\\
Z^\opp_n(x,y;z_1,z_2)&=Z_n^{k_1,k_2}(x,y;z_1,z_2;xz_1^{\:2}\!+\!(y\!-\!x\!-\!1)z_1\!+\!1,
xz_2^{\:2}\!+\!(y\!-\!x\!-\!1)z_2\!+\!1).\end{align}

It can also be seen that the case $m=2$, $k_1=1$ and $k_2=n$ of~\eqref{ZmultASM} gives the identity~\eqref{propeq1}
 satisfied by the opposite-boundary doubly-refined
ASM generating function, and that the case $m=2$ of~\eqref{Xid}
and the result $Z^\opp_n(x,y;z_1,z_2)=X_n(x,y;z_1,z_2)$ give the identity~\eqref{propeq2}.

For the case $y=x+1$, using the last two formulae of~\eqref{ff2} in~\eqref{Zmult} gives 
$X_n(x,x+1;z_1,\ldots,z_m)=\bigl(\prod_{i=1}^m(xz_i+1)\bigr)^{n-m}\,(x+1)^{(n-m)(n-m-1)/2}\,
\det_{1\le i,j\le m}\bigl(z_i^{\;j-1}((z_i-1)(xz_i+1))^{m-j}\bigr)\big/\prod_{1\le i<j\le m}(z_i-z_j)
=\prod_{1\le i<j\le m}(xz_iz_j+1)\,\bigl(\prod_{i=1}^m(xz_i+1)\bigr)^{n-m}\,(x+1)^{(n-m)(n-m-1)/2}$.
(In the last step, the determinant evaluation $\det_{1\le i,j\le m}\bigl(z_i^{\;j-1}((z_i-1)(xz_i+1))^{m-j}\bigr)=
\prod_{1\le i<j\le m}(z_i-z_j)(xz_iz_j+1)$ can be obtained using the method of identification of factors. See
Krattenthaler~\cite[p.~5 \& Sec.~2.4]{Kra99}, \cite[Sec.~4, Method~3]{Kra05}.  For example, it follows that $xz_iz_j+1$ is a factor of this 
determinant, for each $i\ne j$, by setting $z_j=-1/(xz_i)$, and observing that row $j$ of the matrix is then row $i$ divided by $(-x z_i^2)^{m-1}$.)
This explicit expression for $X_n(x,x+1;z_1,\ldots,z_m)$ can be seen to generalize the last three formulae of~\eqref{ff2}.

Finally, it will be shown in Section~\ref{add2} that the function~\eqref{Zmult} at $x=y=1$ can be expressed as
\begin{multline}\label{schurgen}X_n(1,1;z_1,\ldots,z_m)=3^{-n(n-1)/2}\,(-q)^{m(n-1)}\,\bigl((z_1\!+\!q)\ldots(z_m\!+\!q)\bigr)^{n-1}\times\\
s_{(n-1,n-1,\ldots,2,2,1,1)}\bigl(\tfrac{qz_1+1}{z_1+q},\ldots,\tfrac{qz_m+1}{z_m+q},
\,\underbrace{1,\ldots,1}_{2n-m}\,\bigr)\big|_{q=e^{\pm 2\pi i/3}},\end{multline}
where this uses the same notation as, and can be seen to generalize, the identity~\eqref{schur}.

\subsection{ASMs with several rows or columns closest to two opposite boundaries prescribed}
In this section, the enumeration of ASMs with prescribed configurations of
several rows or columns closest to two opposite boundaries (but not involving any further statistics) is discussed.

Let $K_1$ and $K_2$ be subsets of $\{1,\ldots,n\}$ with $|K_1|+|K_2|\le n$,
and let $A_{K_1}$ and~$A_{K_2}$ be matrices with $n$ columns and~$|K_1|$ or $|K_2|$ rows respectively,
in which each entry is~$0$,~$1$ or~$-1$, along each row and column the nonzero entries alternate in sign, in each row the entries sum to~1,
in column $j$ of $A_{K_1}$ the entries sum to~$1$ if $j\in K_1$ and the entries
sum to~$0$ with the first nonzero entry (if there is one) being~$1$ if $j\not\in K_1$,
and in column~$j$ of~$A_{K_2}$ the entries sum to~$1$ if $j\in K_2$ and the entries
sum to~$0$ with the first nonzero entry (if there is one) being~$-1$ if $j\not\in K_2$.
(For example, the entries of these matrices could be taken as simply
$(A_{K_1})_{ij}=\delta_{j,K_{1,i}}$ and $(A_{K_2})_{ij}=\delta_{j,K_{2,i}}$,
where~$K_{1,i}$ and~$K_{2,i}$ are the~$i$th smallest elements of~$K_1$ and~$K_2$, respectively.)
For a fixed choice of matrices $A_{K_1}$ and~$A_{K_2}$ which satisfy these conditions,
define
\begin{multline}\label{Amult}\A_{n,K_1,K_2}=\text{number of $n\times n$ ASMs whose first~$|K_1|$ rows are given by~$A_{K_1}$,}\\
\text{and last~$|K_2|$ rows are given by $A_{K_2}$,}\end{multline}
where it can be checked easily that~$\A_{n,K_1,K_2}$ is independent of the choice of~$A_{K_1}$ and~$A_{K_2}$.
(Note that if $K_1=\emptyset$ or $K_2=\emptyset$, then there is no restriction on the first or last rows, respectively, of the ASMs in~\eqref{Amult}.)
Thus, $\A_{n,K_1,K_2}$ is the number of $n\times n$ ASMs in which certain rows or columns closest to two opposite boundaries are prescribed.

Alternatively, $\A_{n,K_1,K_2}$ can be written, without reference to~$A_{K_1}$ and~$A_{K_2}$, as
the number of $(n-|K_1|-|K_2|)\times n$ matrices in which each entry is~$0$, $1$ or~$-1$,
along each row and column the nonzero entries alternate in sign, in each row the entries sum to~1,
and in column $j$ the entries sum to~$1$ if $j\in\{1,\ldots,n\}\setminus(K_1\cup K_2)$,
the entries sum to~$-1$ if $j\in K_1\cap K_2$,
the entries sum to~$0$ with the first nonzero entry (if there is one) being~$-1$ if $j\in K_1\setminus K_2$,
and the entries sum to~$0$ with the first nonzero entry (if there is one) being~$1$ if $j\in K_2\setminus K_1$.

It can be seen that
\begin{equation}\A_{n,K_1,K_2}=\A_{n,K_2,K_1}=\A_{n,\{n+1-k\mid k\in K_1\},\{n+1-k\mid k\in K_2\}}\end{equation}
and that
\begin{equation}\A_{n,\emptyset,\emptyset}=\A_n,\quad\A_{n,\{k+1\},\emptyset}=
\A_{n,k},\quad\A_{n,\{k_1+1\},\{n-k_2\}}=\A^\opp_{n,k_1,k_2},\end{equation}
for any $0\le k,k_1,k_2\le n-1$.

It can also be seen that $\A_{n,\{1,\ldots,n\}\setminus {K_1},\emptyset}$ or
$\A_{n,\{1,\ldots,n\}\setminus {K_2},\emptyset}$ give the numbers of possible matrices $A_{K_1}$ or
$A_{K_2}$, respectively, which satisfy the conditions outlined above.  Furthermore,
for a fixed set $K$ of positive integers, $\A_{n,\{1,\ldots,n\}\setminus K,\emptyset}$ is independent
of $n$, for $n\ge\max(K)$, with an operator formula for this number having been obtained by Fischer~\cite[Thm.~1]{Fis06}.
Some related functions for the six-vertex model with DWBC have been studied by Colomo and Pronko~\cite{ColPro12}.

The numbers $\A_{n,K_1,K_2}$ have been studied by Fischer~\cite{Fis11,Fis12}, Fischer and Romik~\cite{FisRom09}, and
Karklinsky and Romik~\cite{KarRom10}.
Results which have been obtained include
linear relations between the numbers $\A_{n,K,\emptyset}$ with $|K|=2$ and the opposite-boundary doubly-refined ASM numbers
(Fischer~\cite[Eq.~(5.3)]{Fis11}, and Karklinsky and Romik~\cite[Eq.~(7)]{KarRom10}),
explicit formulae for $\A_{n,K,\emptyset}$ with $|K|=2$
(Fischer~\cite[Eq.~(1.3)]{Fis11}, and Karklinsky and Romik~\cite[Thm.~1]{KarRom10}),
and an expression for~$\A_{n,K,\{k\}}$ with $|K|=2$ in terms of numbers of $n\times n$ ASMs with prescribed configurations on three boundaries
(Fischer~\cite[Thm.~1]{Fis12}).  Also, related numbers have been defined (Fischer~\cite[Eq.~(1.2)]{Fis11},~\cite[Eq.~(2.9)]{Fis12}),
and have been shown to give the numbers $\A_{n,K_1,K_2}$ in certain cases (Fischer~\cite[Thm.~1]{Fis11}),
and to satisfy certain linear relations (Fischer~\cite[Eq.~(5.1) \& p.~253, first equation]{Fis11},~\cite[Eqs.~(1.5) \& (1.6)]{Fis12}).

\subsection{The cases $Z_n(1,3)$ and $Z_n(1,3;z)$}
Results involving $Z_n(1,3)$ and $Z_n(1,3;z)$,
some of which were previously conjectured by Mills, Robbins and Rumsey~\cite[Conj.~6--7]{MilRobRum83},~\cite[Conj.~5]{MilRobRum86},
and which correspond to the so-called $3$-enumeration of ASMs,
have been obtained by Cantini~\cite[Sec.~4.1]{Can07}, Colomo and
Pronko~\cite{ColPro05c},~\cite[Secs.~4.4 \&~5.4]{ColPro05a},~\cite[Sec.~4.3]{ColPro06},
Kuperberg~\cite[Thm.~2]{Kup96},~\cite[Thm.~3]{Kup02},
Okada~\cite[Thm.~2.4(1), fourth eq.]{Oka06}, and
Stroganov~\cite{Str03}.

\subsection{The case $Z_n(1,1;-1)$}
It was observed by Di Francesco~\cite[Eqs.~(2.7)--(2.8)]{Dif06},
and shown explicitly by Williams~\cite[Thm.~4]{Wil08}, that
\begin{equation}Z_n(1,1;-1)=\mathcal{V}_n^2,\end{equation}
where $\mathcal{V}_n$ is the number of
vertically-symmetric $n\times n$ ASMs (i.e., ASMs $A$ with $A_{ij}=A_{i,n+1-j}$).  Note that for $n$ even,
it can be seen easily (for example, using the second equation of~\eqref{A1}) that both sides of the equation are zero.
For $n$ odd, $\mathcal{V}_n=\prod_{i=1}^{(n-1)/2}(6i-2)!/(n+2i-1)!$,
as conjectured by Robbins~\cite{Rob91,Rob00}, and first shown by Kuperberg~\cite[Thm.~2, second equation]{Kup02}.
It is known that $\mathcal{V}_n$ is also
the number of descending plane partitions invariant under a certain operation
(see Behrend, Di Francesco and Zinn-Justin~\cite[Sec.~4.2 \&~Eq.~(102)]{BehDifZin12}),
and the number of totally symmetric self-complementary plane partitions invariant under a certain operation
(see Ishikawa~\cite[Thm.~7.11(i)]{Ish06b}), and it has been conjectured by
Di Francesco~\cite[Eq.~(2.8)]{Dif06},~\cite[Eq.~(4.5)]{Dif07} that $\mathcal{V}_n^{\;2}$ (up to sign) and $\mathcal{V}_n^{\;4}$ are obtained when a
parameter is set to~$-1$ in certain
generating functions for totally symmetric self-complementary and cyclically symmetric transpose-complementary plane partitions.

\subsection{ASMs with a fixed number of generalized inversions}
Expressions for the coefficients of~$x^p$ in $Z_n(x,1)$, i.e.,
for the number of $n\times n$ ASMs with $p$ generalized inversions,
can be obtained using a result of Behrend~\cite[Cor.~14]{Beh08}, and are given
by Behrend, Di Francesco and Zinn-Justin~\cite[Eqs.~(24)--(25)]{BehDifZin12}. Some particular cases of $p$,
for ASMs with fixed values of certain further statistics,
are also considered by Behrend, Di Francesco and Zinn-Justin~\cite[pp.~337--338]{BehDifZin12},~\cite[Sec.~7.1]{BehDifZin13}.

\subsection{ASMs with a fixed number of $-1$'s}
For the case of $Z_n(1,y)$, an expression for the
coefficient of $y^m$, i.e., for the number of $n\times n$ ASMs with~$m$ $-1$s,
has been obtained by Cori, Duchon and Le Gac~\cite{CorDucLeg10,Leg11a},~\cite[Ch.~3]{Leg11b} (and, for $m=1,2$, by Aval~\cite[Prop.~4]{Ava08}),
a certain factorization property (previously conjectured by
Mills, Robbins and Rumsey~\cite[Conj.~4 \& Conj.~5]{MilRobRum83},~\cite[pp.~50 \&~54]{MilRobRum87} and Robbins~\cite[Sec.~2]{Rob00})
has been established by Kuperberg~\cite[Thm.~3]{Kup96},~\cite[Thm.~4, first two eqs.]{Kup02},
and a certain operator formula containing a parameter which corresponds to the number of $-1$'s in
an ASM has been obtained by Fischer~\cite[Thm.~1]{Fis10}.

Results involving the coefficient of $y^1$ in $Z_n(x,y;z)$,
i.e., for a certain generating function for $n\times n$ ASMs with a single $-1$, have
been obtained by Lalonde~\cite{Lal02,Lal06}.

\subsection{Objects in simple bijection with ASMs}\label{bij}
Simple bijections are known between ASMs and various other combinatorial objects.  Some examples,
all of which are reviewed by Propp~\cite[Secs.~2--4 \&~7]{Pro01}, are
bijections between $\ASM(n)$ and sets of monotone (or Gog) triangles with bottom row $1,\ldots,n$,
$(n+1)\times(n+1)$ corner-sum matrices, $(n+1)\times(n+1)$ height-function matrices, 3-colourings of an
$(n+1)\times(n+1)$ grid with certain boundary conditions,
configurations of the six-vertex model on an $n\times n$ grid with DWBC, or
fully packed loop configurations on an $n\times n$ grid with certain boundary conditions.
Some further examples are a bijection between $\ASM(n)$ and the set of sets of $n$ osculating paths on an $n\times n$ grid in which
all paths start and end on two adjacent boundaries
(see, for example, Behrend~\cite[Secs.~2--4]{Beh08}, Bousquet-M\'elou and Habsieger~\cite[pp.~68--69]{BouHab93}, Bressoud~\cite[pp.~226--227]{Bre99},
or E\u{g}ecio\u{g}lu, Redmond and Ryavec~\cite[pp.~35--36]{EgeRedRya01}),
and a bijection between $\ASM(n)$ and the set of alternating paths
for any fixed fully packed loop configuration on an $n\times n$ grid
(see Ng~\cite[Prop.~3.1]{Ng12}).

In each of these cases, the statistics for the other combinatorial object which correspond,
under the bijection, to the ASM statistics~\eqref{numuA}--\eqref{rho} or~\eqref{genstat}
can be obtained relatively straightforwardly. For example, this will be done for configurations of the
six-vertex model with DWBC in~\eqref{6VDWBCstat}.  Accordingly, all of the results
reviewed or obtained in this paper for~$\ASM(n)$ with the statistics~\eqref{numuA}--\eqref{rho} or~\eqref{genstat}
could alternatively be expressed in
terms of any of these other combinatorial objects and their associated statistics.
However, it should be noted that certain statistics will seem more natural when expressed in terms of some objects than others.

It should also be noted that, by applying some of these simple bijections between ASMs and other objects, further statistics or characteristics,
which would not seem natural in terms of the ASMs themselves, may become apparent.
For instance, a natural statistic for the monotone triangle which corresponds to an ASM is the number of diagonals whose entries all have certain minimal values
(see, for example, Zeilberger~\cite[Lem.~1]{Zei96a} for a result, previously conjectured by Mills, Robbins and Rumsey~\cite[Conj.~7]{MilRobRum86},
involving this statistic),
natural statistics for the 3-colouring of a grid which corresponds to an ASM are the numbers of appearances of each colour (see, for example,
Rosengren~\cite[Cor.~8.4]{Ros09},~\cite[Sec.~3]{Ros11} for results involving these statistics),
and a natural characteristic for the fully packed loop configuration which corresponds to an ASM is its associated link pattern
(see, for example, Cantini and Sportiello~\cite[Eq.~(23)]{CanSpo11} for a result, previously conjectured by Razumov
and Stroganov~\cite{RazStr04a}, involving this characteristic).

In the case of the six-vertex model with DWBC, ASM generating functions are
closely related to certain partition functions, expectation values, probabilities
or correlation functions for the model, and the results of this paper could alternatively
be expressed in terms of the latter quantities.  For example, the quadruply-refined ASM generating function is
closely associated with a certain four-point boundary correlation function for the six-vertex model with DWBC.
For studies of various correlation functions, and related quantities,
for the six-vertex model with DWBC, see, for example,
Bogoliubov, Kitaev and Zvonarev~\cite{BogKitZvo02},
Bogoliubov, Pronko and Zvonarev~\cite{BogProZvo02},
Colomo and Pronko~\cite{ColPro05b,ColPro06,ColPro08,ColPro12},
Foda and Preston~\cite{FodPre04}, and Motegi~\cite{Mot11,Mot12}.

\subsection{Descending plane partitions}\label{DPPs}
Descending plane partitions are certain combinatorial objects first defined by Andrews~\cite{And79,And80}.
It was shown by Behrend, Di Francesco and Zinn-Justin~\cite[Thm.~1]{BehDifZin13}
that the opposite-boundary doubly-refined ASM generating function
$Z^\opp_n(x,y;z_1,z_2)$ is equal to the generating function for descending plane partitions with largest part at most~$n$,
in which the statistics associated with $x$, $y$,~$z_1$ and~$z_2$ are, respectively,
the number of nonspecial parts, the number of special parts (as first defined by Mills, Robbins
and Rumsey~\cite[p.~344]{MilRobRum83}), the number of~$n$'s,
and the number of $(n-1)$'s plus the number of rows of length~$n-1$, in a descending plane partition.
(Note, however, that no general bijection is currently known between $\ASM(n)$ and the set of descending plane partitions with
largest part at most~$n$, and
that no further statistics for descending plane partitions are currently known which, together with the previous four statistics, lead to
generating functions which are equal to the adjacent-boundary doubly-refined, triply-refined or quadruply-refined ASM generating functions.)

It follows that certain known results for descending plane partitions
correspond to results involving the opposite-boundary doubly-refined, singly-refined
or unrefined ASM generating functions.  For example, certain such results were outlined
in Section~\ref{prev1}.

Results are also known for the enumeration of descending plane partitions with a prescribed sum of parts.
In particular, it follows from work of Mills, Robbins and Rumsey~\cite[Sec.~5]{MilRobRum82} that, for $0\le k\le n-1$,
\begin{align}\notag\textstyle\sum_{D\in\mathrm{DPP}(n)}q^{|D|}&=\textstyle\prod_{i=0}^{n-1}\frac{[3i+1]_q!}{[n+i]_q!},\\
\label{DPP}\textstyle\sum_{D\in\mathrm{DPP}(n,k)}q^{|D|}&=\textstyle
\frac{q^{kn}\:[n+k-1]_q!\:[2n-k-2]_q!}{[k]_q!\:[n-k-1]_q!\:[2n-2]_q!}\:\prod_{i=0}^{n-2}\!\frac{[3i+1]_q!}{[n+i-1]_q!},\end{align}
where $\mathrm{DPP}(n)$ denotes the set of descending plane partitions with largest part at most~$n$,
$\mathrm{DPP}(n,k)$ denotes the set of elements of $\mathrm{DPP}(n)$ with exactly $k$ parts equal to $n$,
$|D|$ denotes the sum of parts of a descending plane partition $D$, and $[n]_q!$ is the usual $q$-factorial of $n$.
As expected, each RHS of~\eqref{DPP} reduces to each RHS of~\eqref{ASM}--\eqref{RASM} for $q=1$.  However, no ASM statistic
is currently known which corresponds to the sum of parts of a descending plane
partition, and leads to generating functions matching those of~\eqref{DPP}.

For further information, and related results or conjectures, regarding descending plane partitions, see, for
example,
Andrews~\cite{And79,And80},
Ayyer~\cite{Ayy10},
Behrend, Di Francesco and Zinn-Justin~\cite{BehDifZin12,BehDifZin13},
Bressoud~\cite{Bre99},
Bressoud and Propp~\cite{BrePro99},
Krattenthaler~\cite{Kra06},
Lalonde~\cite{Lal02,Lal03,Lal06},
Mills, Robbins and Rumsey~\cite{MilRobRum82,MilRobRum83,MilRobRum87},
Robbins~\cite{Rob91,Rob00},
and Striker~\cite{Str11a}.

\subsection{Totally symmetric self-complementary plane partitions}\label{TSSCPP}
Totally symmetric self-complementary plane partitions are certain combinatorial objects first defined
by Mills, Robbins and Rumsey~\cite[Sec.~1]{MilRobRum86}, and Stanley~\cite[Sec.~2]{Sta86a}.
It was shown by Fonseca and Zinn-Justin~\cite{FonZin08},
following conjectures of Mills, Robbins and Rumsey~\cite[Conj.~2 \&~3]{MilRobRum86},
that the opposite-boundary doubly-refined ASM generating function
with both bulk parameters set to~1, i.e., $Z^\opp_n(1,1;z_1,z_2)$, is equal to a generating function for
totally symmetric self-complementary plane partitions in a $2n\times2n\times2n$ box
in which the statistics associated with~$z_1$ and~$z_2$ can be certain pairs from among several
statistics defined by Doran~\cite[Sec.~7]{Dor93}, and Mills, Robbins and Rumsey~\cite[Sec.~3]{MilRobRum86} (see also Robbins~\cite[p.~16]{Rob91}).
The singly-refined case of this result was obtained previously by Razumov, Stroganov and
Zinn-Justin~\cite[Sec.~5.5]{RazStrZin07}.
(Note, however, that no general bijection is currently known between $\ASM(n)$ and the set of totally symmetric self-complementary plane partitions
in a $2n\times2n\times2n$ box, and that no statistics for totally symmetric self-complementary plane partitions are currently known which
have the same enumerative behaviour as the number of generalized inversions or the number of $-1$'s in an ASM.)

It follows that certain known results for totally symmetric self-complementary plane partitions,
or obtained while studying these objects,
correspond to results involving the ASM generating functions $Z^\opp_n(1,1;z_1,z_2)$ or
$Z_n(1,1;z)$, or the ASM numbers $\A^\opp_{n,k_1,k_2}$, $\A_{n,k}$ or~$\A_n$.
For example, for $Z^\opp_n(1,1;z_1,z_2)$, $Z_n(1,1;z)$ or $\A_n$,
Pfaffian expressions follow from
results of Ishikawa~\cite[Thms.~1.2 \&~1.4, \& Sec.~7]{Ish06a} and Stembridge~\cite[Thm.~8.3]{Ste90},
constant-term expressions follow from
results of Ishikawa~\cite[Sec.~8]{Ish06a}, Krattenthaler~\cite[Thm.]{Kra96}\ and Zeilberger~\cite{Zei94},~\cite[Sublems.~1.1 \&~1.2]{Zei96a},
and integral expressions (which can easily be converted to constant-term expressions)
follow from results of Fonseca and Zinn-Justin~\cite[Eqs.~(4.9) \&~(4.14)]{FonZin08} and 
Zinn-Justin and Di Francesco~\cite[Eqs.~(37) \&~(39)]{ZinDif08}.
Note that many of these results are expressed in terms of certain triangles of positive integers (specifically, monotone or Gog triangles for ASMs,
and Magog triangles for totally symmetric self-complementary plane partitions), or closely related integer arrays.
Also, many such results are stated in more general forms
which contain additional parameters associated
with certain entries of such arrays being prescribed to take certain values, or being bounded by certain values.

For the case of adjacent-boundary doubly-refined ASM enumeration, it has been conjectured by Cheballah~\cite[Conj.~4.3.1]{Che11}
that $\A^\adj_{n,k_1,k_2}$ equals the number of totally symmetric self-complementary plane partitions
in a $2n\times2n\times2n$ box for which a certain pair of statistics have values $k_1$ and $k_2$.

For further information, and related results or conjectures, regarding totally symmetric
self-complementary plane partitions, see, for example,
Andrews~\cite{And94},
Ayyer, Cori and Gouyou-Beauchamps~\cite{AyyCorGou11},
Bressoud~\cite[Sec.~6.2]{Bre99},
Bressoud and Propp~\cite{BrePro99},
Biane and Cheballah~\cite{BiaChe13},
Cheballah~\cite{Che11},
Cheballah and Biane~\cite{CheBia12},
Di Francesco~\cite{Dif06,Dif07},
Doran~\cite{Dor93},
Fonseca~\cite[Secs.~3--4]{Fon10},
Fonseca and Zinn-Justin~\cite{FonZin08},
Ishikawa~\cite{Ish06a,Ish06b},
Krattenthaler~\cite{Kra96},
Mills, Robbins and Rumsey~\cite{MilRobRum86},
Robbins~\cite{Rob91},
Stanley~\cite{Sta86a},
Stembridge~\cite[Sec.~8]{Ste90},
Striker~\cite{Str09,Str11b},
Zeilberger~\cite{Zei94,Zei96a},
Zinn-Justin~\cite{Zin09},
and Zinn-Justin and Di Francesco~\cite{ZinDif08}.

\subsection{Loop models}\label{loops}
Numbers of certain ASMs, or plane partitions, have been found to appear also as
appropriately normalized entries of particular eigenvectors (or sums of such entries, or norms of such eigenvectors)
associated with certain
cases of integrable loop models (or associated quantum spin chains).  A wide variety of such cases are known,
and these often involve ASMs
with prescribed link patterns of associated fully packed loop configurations,
with prescribed values of some of the statistics of~\eqref{numuA}-\eqref{rho} or of related statistics or parameters,
or subject to invariance under certain symmetry operations.

Of the many papers which study the confirmed or conjectured
appearances of numbers of ASMs in such contexts, a few examples are
Batchelor, de Gier and Nienhuis~\cite{BatDegNie01},
Cantini and Sportiello~\cite{CanSpo11,CanSpo12},
Di Francesco~\cite{Dif04b,Dif04a},
Di Francesco and Zinn-Justin~\cite{DifZin05},
Pasquier~\cite{Pas06},
Razumov and Stroganov~\cite{RazStr01,RazStr04a,RazStr06c}, and
Stroganov~\cite{Str01}. For reviews of some of these matters, see, for example,
de Gier~\cite{Deg05,Deg07,Deg09}, or Zinn-Justin~\cite{Zin09}.

\subsection{ASMs invariant under symmetry operations}\label{invsymm}
This paper is focused on the enumeration of ASMs with prescribed values of the bulk statistics~\eqref{numuA}
and boundary statistics~\eqref{rho}, but with no other conditions applied.
However, other studies have focused on the enumeration of ASMs which are invariant under certain symmetry operations,
or subject to related conditions,
and which in some cases also have prescribed values of some of the statistics of~\eqref{numuA}-\eqref{rho}.
For results of this type (some of which
remain conjectural), see, for example, Aval and Duchon~\cite{AvaDuc09,AvaDuc10},
Bousquet-M\'elou and Habsieger~\cite{BouHab93},
Bressoud~\cite[Sec.~6.1]{Bre99},
Kuperberg~\cite{Kup02}, Okada~\cite{Oka06},
Robbins~\cite{Rob91,Rob00}, Razumov and Stroganov~\cite{RazStr04b,RazStr06b,RazStr06a},
Stanley~\cite{Sta86b},
and Stroganov~\cite{Str04,Str08}.

Numbers of certain such ASMs have also been related (again, in some cases, only conjecturally)
to numbers of certain descending plane partitions
(see, for example, de Gier, Pyatov and Zinn-Justin~\cite[Prop.~3, first equation]{DegPyaZin09}, and
Mills, Robbins and Rumsey~\cite[Conj.~3S]{MilRobRum83}),
numbers of certain totally-symmetric self-complementary plane partitions
(see, for example, Ishikawa~\cite{Ish06a,Ish06b}, and Mills, Robbins and Rumsey~\cite[Conjs.~4 \& 6]{MilRobRum86}), or
appropriately normalized entries of eigenvectors associated with certain cases of loop models
(see, for example, the references given in Section~\ref{loops}).

\section{Main results}\label{newres}
In this section, the main result of this paper (Theorem~\ref{thm}) is stated,
and some of its consequences (Corollaries~\ref{cor1}--\ref{cor10}) are derived and discussed.
The proof of Theorem~\ref{thm} is deferred to Section~\ref{pf}.

In particular, new results are given for the quadruply-refined ASM generating function (Theorem~\ref{thm} and Corollary~\ref{cor9}),
triply-refined ASM generating function (Corollaries~\ref{cor1},~\ref{cor3} and~\ref{cor10}),
adjacent-boundary doubly-refined ASM generating functions (Corollaries~\ref{cor2},~\ref{cor5} and~\ref{cor6}),
singly-refined ASM generating function (Corollary~\ref{cor7}), and unrefined ASM generating function (Corollary~\ref{cor8}),
and a previously-known result is obtained for the opposite-boundary doubly-refined ASM generating function (Corollary~\ref{cor4}).

\subsection{Main theorem}\label{new1}
The primary result of this paper is as follows.
\begin{theorem}\label{thm}
The quadruply-refined ASM generating function satisfies
\begin{multline}\label{quadrel}y(z_4\!-\!z_2)(z_1\!-\!z_3)\,Z^\qua_n(x,y;z_1,z_2,z_3,z_4)\,Z_{n-2}(x,y)=\\
\bigl((z_1\!-\!1)(z_2\!-\!1)\!-\!yz_1z_2\bigr)\bigl((z_3\!-\!1)(z_4\!-\!1)\!-\!yz_3z_4\bigr)\,Z^\adj_{n-1}(x,y;z_4,z_1)\,Z^\adj_{n-1}(x,y;z_2,z_3)\;-\\
\bigl(x(z_4\!-\!1)(z_1\!-\!1)\!-\!y\bigr)\bigl(x(z_2\!-\!1)(z_3\!-\!1)\!-\!y\bigr)\,z_1z_2z_3z_4\,
\Z^\adj_{n-1}(x,y;z_1,z_2)\,\Z^\adj_{n-1}(x,y;z_3,z_4)\;-\\
(z_2\!-\!1)(z_3\!-\!1)\bigl((z_4\!-\!1)(z_1\!-\!1)\!-\!yz_4z_1\bigr)\,Z^\adj_{n-1}(x,y;z_4,z_1)\,Z_{n-2}(x,y)\;+\\
(z_3\!-\!1)(z_4\!-\!1)\bigl(x(z_1\!-\!1)(z_2\!-\!1)\!-\!y\bigr)\,z_1z_2\,(xz_3z_4)^{n-1}\,\Z^\adj_{n-1}(x,y;z_1,z_2)\,Z_{n-2}(x,y)\;-\\
(z_4\!-\!1)(z_1\!-\!1)\bigl((z_2\!-\!1)(z_3\!-\!1)\!-\!yz_2z_3\bigr)\,Z^\adj_{n-1}(x,y;z_2,z_3)\,Z_{n-2}(x,y)\;+\\
(z_1\!-\!1)(z_2\!-\!1)\bigl(x(z_3\!-\!1)(z_4\!-\!1)\!-\!y\bigr)\,z_3z_4\,(xz_1z_2)^{n-1}\,\Z^\adj_{n-1}(x,y;z_3,z_4)\,Z_{n-2}(x,y)\;+\\
(z_1\!-\!1)(z_2\!-\!1)(z_3\!-\!1)(z_4\!-\!1)\bigl(1-(x^2z_1z_2z_3z_4)^{n-1}\bigr)\,Z_{n-2}(x,y)^2.\end{multline}\end{theorem}
The identity~\eqref{quadrel} holds for each $n\ge3$.
Furthermore, if $Z_0(x,y)$ is taken to be~1, then it can be seen, using
$Z^\adj_1(x,y;z_1,z_2)=\Z^\adj_1(x,y;z_1,z_2)=1$ and $Z^\qua_2(x,y;z_1,z_2,z_3,z_4)=1+xz_1z_2z_3z_4$ (from~\eqref{Zquad123}--\eqref{Zaltdef}),
that~\eqref{quadrel} also holds for $n=2$.

The proof of~\eqref{quadrel}, which will be given in Section~\ref{pf}, will involve
using a relation between a certain generalized ASM generating function and
the partition function of the six-vertex model with DWBC, and
then combining the Desnanot--Jacobi determinant
identity with the Izergin--Korepin formula for this partition function.

An alternative form of~\eqref{quadrel} will be obtained, and some further related results will be discussed,
in Section~\ref{add0}.

It can be seen that~\eqref{quadrel} enables $Z^\qua_n(x,y;z_1,z_2,z_3,z_4)$ to be obtained recursively,
using the initial conditions (from~\eqref{Zquad123})
$Z^\qua_1(x,y;z_1,z_2,z_3,z_4)=1$ and $Z^\qua_2(x,y;z_1,z_2,z_3,z_4)=1+xz_1z_2z_3z_4$, and the definitions (from~\eqref{Zdef}--\eqref{Zaltdef})
$Z^\adj_n(x,y;z_1,z_2)=Z^\qua_n(x,y;z_1,1,1,z_2)$, $\Z^\adj_n(x,y;z_1,z_2)=Z^\qua_n(x,y;z_1,z_2,1,1)$ and
$Z_n(x,y)=Z^\qua_n(x,y;1,1,1,1)$.  Accordingly, in this sense, $Z^\qua_n(x,y;z_1,z_2,z_3,z_4)$,
and all of the ASM generating functions of~\eqref{Zdef}--\eqref{Zaltdef} which are defined in terms of $Z^\qua_n(x,y;z_1,z_2,z_3,z_4)$,
are determined by~\eqref{quadrel}.

Note, however, that if the generating functions are computed recursively in this way,
then, for each successive~$n$, $Z^\qua_n(x,y;z_1,z_2,z_3,z_4)$ should first
be computed for arbitrary~$z_1$,~$z_2$,~$z_3$, and~$z_4$, with the factor $(z_1-z_3)(z_4-z_2)$
being explicitly cancelled from both sides of~\eqref{quadrel}, so
that division by zero is avoided when boundary parameters need to be set to~$1$ in subsequent computations.
Alternatively, certain expressions in which this cancellation has effectively been done, will be given in Section~\ref{new2}.

By replacing $z_2$ and $z_4$ by $\frac{1}{z_2}$ and $\frac{1}{z_4}$ respectively,
Theorem~\ref{thm} can be restated for the alternative quadruply-refined ASM generating function of~\eqref{Zaltdef} as
\begin{multline}\label{altquadrel}
y(z_1\!-\!z_3)(z_2\!-\!z_4)\,\Z^\qua_n(x,y;z_1,z_2,z_3,z_4)\,Z_{n-2}(x,y)=\\
\shoveleft{\bigl((z_1\!-\!1)(z_2\!-\!1)\!+\!yz_1\bigr)\bigl((z_3\!-\!1)(z_4\!-\!1)\!+\!yz_3\bigr)\,(z_2z_4)^{n-1}\;\times}\\
\shoveright{Z^\adj_{n-1}(x,y;\tfrac{1}{z_4},z_1)\,Z^\adj_{n-1}(x,y;\tfrac{1}{z_2},z_3)\;-}\\
\shoveleft{\bigl(x(z_4\!-\!1)(z_1\!-\!1)\!+\!yz_4\bigr)\bigl(x(z_2\!-\!1)(z_3\!-\!1)\!+\!yz_2\bigr)\,z_1z_3(z_2z_4)^{n-2}\;\times}\\
\shoveright{\Z^\adj_{n-1}(x,y;z_1,\tfrac{1}{z_2})\,\Z^\adj_{n-1}(x,y;z_3,\tfrac{1}{z_4})\;-}\\
(z_2\!-\!1)(z_3\!-\!1)\bigl((z_4\!-\!1)(z_1\!-\!1)\!+\!yz_1\bigr)\,(z_2z_4)^{n-1}\,Z^\adj_{n-1}(x,y;\tfrac{1}{z_4},z_1)\,Z_{n-2}(x,y)\;+\\
(z_3\!-\!1)(z_4\!-\!1)\bigl(x(z_1\!-\!1)(z_2\!-\!1)\!+\!yz_2\bigr)
\,z_1z_2^{\;n-2}(xz_3)^{n-1}\,\Z^\adj_{n-1}(x,y;z_1,\tfrac{1}{z_2})\,Z_{n-2}(x,y)\;-\\
(z_4\!-\!1)(z_1\!-\!1)\bigl((z_2\!-\!1)(z_3\!-\!1)\!+\!yz_3\bigr)\,(z_2z_4)^{n-1}\,Z^\adj_{n-1}(x,y;\tfrac{1}{z_2},z_3)\,Z_{n-2}(x,y)\;+\\
(z_1\!-\!1)(z_2\!-\!1)\bigl(x(z_3\!-\!1)(z_4\!-\!1)\!+\!yz_4\bigr)
\,z_3z_4^{\;n-2}(xz_1)^{n-1}\,\Z^\adj_{n-1}(x,y;z_3,\tfrac{1}{z_4})\,Z_{n-2}(x,y)\;+\\
(z_1\!-\!1)(z_2\!-\!1)(z_3\!-\!1)(z_4\!-\!1)\bigl((z_2z_4)^{n-1}-(x^2z_1z_3)^{n-1}\bigr)\,Z_{n-2}(x,y)^2.\end{multline}

\subsection{Corollaries for arbitrary bulk parameters $x$ and $y$}\label{new2}
In this section, some consequences of Theorem~\ref{thm}, for the case in which the bulk
parameters $x$ and $y$ remain arbitrary, are derived.
These results are obtained from~\eqref{quadrel} using only definitions
and elementary properties
of ASM generating functions from Section~\ref{genfunc},
and the initial conditions (from~\eqref{Zquad123}--\eqref{Zaltdef}) $Z^\adj_1(x,y;z_1,z_2)=
\Z^\adj_1(x,y;z_1,z_2)=1$ and  $Z^\adj_2(x,y;z_1,z_2)=
\Z^\adj_2(x,y;z_1,z_2)=1+xz_1z_2$.
In particular, most of the results are obtained by setting boundary parameters to~$1$,
using~\eqref{1} for the corresponding specializations, and, in some cases, solving recursion relations.

Alternative forms of some of the results of this section will be given in Section~\ref{add0}.

\begin{corollary}\label{cor1}The triply-refined ASM generating function satisfies
\begin{multline}\label{triprel}
(z_2\!-\!z_1)(z_3\!-\!1)\,Z^\tri_n(x,y;z_1,z_2,z_3)\,Z_{n-2}(x,y)=\\
\bigl((z_2\!-\!1)(z_3\!-\!1)\!-\!yz_2z_3\bigr)\,z_1\,Z^\adj_{n-1}(x,y;z_1,z_3)\,Z_{n-1}(x,y;z_2)\;-\\
\bigl(x(z_1\!-\!1)(z_3\!-\!1)\!-\!y\bigr)\,z_1z_2z_3\,\Z^\adj_{n-1}(x,y;z_2,z_3)\,Z_{n-1}(x,y;z_1)\;-\\
(z_1\!-\!1)(z_3\!-\!1)\,z_2\,Z_{n-1}(x,y;z_2)\,Z_{n-2}(x,y)\;+\\
(z_2\!-\!1)(z_3\!-\!1)\,z_1\,(xz_2z_3)^{n-1}\,Z_{n-1}(x,y;z_1)\,Z_{n-2}(x,y).
\end{multline}
\end{corollary}
\begin{proof}Set $z_2=1$ in~\eqref{quadrel} (and then relabel $z_3$ as $z_2$ and $z_4$ as~$z_3$).\end{proof}
\begin{corollary}\label{cor2}The adjacent-boundary doubly-refined ASM generating functions satisfy the recursion relations
\begin{gather}
\notag(z_1\!-\!1)(z_2\!-\!1)\,Z^\adj_n(x,y;z_1,z_2)\,Z_{n-2}(x,y)=yz_1z_2\,Z^\adj_{n-1}(x,y;z_1,z_2)\,Z_{n-1}(x,y)\;+\qquad\qquad\\
\notag\bigl(x(z_1\!-\!1)(z_2\!-\!1)\!-\!y\bigr)\,z_1z_2\,Z_{n-1}(x,y;z_1)\,Z_{n-1}(x,y;z_2)\;+\\
\label{adjdoub1}\qquad\qquad\qquad\qquad\qquad\qquad\qquad(z_1\!-\!1)(z_2\!-\!1)\,Z_{n-1}(x,y)\,Z_{n-2}(x,y),\\
\notag(z_1\!-\!1)(z_2\!-\!1)\,\Z^\adj_n(x,y;z_1,z_2)\,Z_{n-2}(x,y)=yz_1z_2\,\Z^\adj_{n-1}(x,y;z_1,z_2)\,Z_{n-1}(x,y)\;+\qquad\qquad\\
\notag\bigl((z_1\!-\!1)(z_2\!-\!1)\!-\!yz_1z_2\bigr)\,Z_{n-1}(x,y;z_1)\,Z_{n-1}(x,y;z_2)\;+\\
\label{adjdoub1alt}\qquad\qquad\qquad\qquad\qquad\qquad(z_1\!-\!1)(z_2\!-\!1)\,(xz_1z_2)^{n-1}\,Z_{n-1}(x,y)\,Z_{n-2}(x,y).\end{gather}
\end{corollary}
\begin{proof}
To obtain~\eqref{adjdoub1}, set $z_2=1$ in~\eqref{triprel} (and then relabel $z_3$ as $z_2$).
To obtain~\eqref{adjdoub1alt}, set $z_1=1$ in~\eqref{triprel} (and then relabel $z_2$ as $z_1$ and $z_3$ as $z_2$).
\end{proof}
Note that, if $Z_0(x,y)$ is taken to be~$1$, then~\eqref{adjdoub1} and~\eqref{adjdoub1alt} hold for all $n\ge2$.
\begin{corollary}\label{cor3}The triply-refined ASM generating function also satisfies
\begin{multline}\label{triprelalt}
y(z_2\!-\!z_1)z_3\,Z^\tri_n(x,y;z_1,z_2,z_3)\,Z_{n-1}(x,y)=\\
(z_1\!-\!1)\bigl((z_2\!-\!1)(z_3\!-\!1)\!-\!yz_2z_3\bigr)\,Z^\adj_n(x,y;z_1,z_3)\,Z_{n-1}(x,y;z_2)\;-\\
(z_2\!-\!1)\bigl(x(z_1\!-\!1)(z_3\!-\!1)\!-\!y\bigr)\,z_1z_3\,\Z^\adj_n(x,y;z_2,z_3)\,Z_{n-1}(x,y;z_1)\;-\\
(z_1\!-\!1)(z_2\!-\!1)(z_3\!-\!1)\,Z_{n-1}(x,y;z_2)\,Z_{n-1}(x,y)\;+\\
(z_1\!-\!1)(z_2\!-\!1)(z_3\!-\!1)\,z_1z_2^{n-1}(xz_3)^n\,Z_{n-1}(x,y;z_1)\,Z_{n-1}(x,y).
\end{multline}
\end{corollary}
\begin{proof}Use~\eqref{adjdoub1}--\eqref{adjdoub1alt} to
replace $Z^\adj_{n-1}(x,y;z_1,z_3)$ and $\Z^\adj_{n-1}(x,y;z_2,z_3)$ in~\eqref{triprel}
by terms which instead contain $Z^\adj_n(x,y;z_1,z_3)$ and $\Z^\adj_n(x,y;z_2,z_3)$,
and then cancel an overall factor which contains a term $z_3-1$.\end{proof}
\begin{corollary}\label{cor4}The opposite-boundary doubly-refined ASM generating function satisfies
\begin{multline}\label{opprel}
(z_1\!-\!z_2)\,Z^\opp_n(x,y;z_1,z_2)\,Z_{n-1}(x,y)=(z_1\!-\!1)\,z_2\,Z_n(x,y;z_1)\,Z_{n-1}(x,y;z_2)\;-\\
z_1\,(z_2\!-\!1)\,Z_{n-1}(x,y;z_1)\,Z_n(x,y;z_2).\end{multline}
\end{corollary}
\begin{proof}Set $z_3=1$ in~\eqref{triprelalt}.\end{proof}
Note that Corollary~\ref{cor4} is a previously-known result,
as already given in~\eqref{propeq1} and discussed in Section~\ref{prev1}.
\begin{corollary}\label{cor5}
The adjacent-boundary doubly-refined ASM generating functions can be expressed as
\begin{align}\notag Z^\adj_n(x,y;z_1,z_2)&=Z_{n-1}(x,y)\,\biggl(1+
\sum_{i=1}^{n-1}\biggl(\frac{y\,z_1\,z_2}{(z_1\!-\!1)(z_2\!-\!1)}\biggr)^{n-i}\times\\
\label{adjdoub2}&\qquad\qquad\qquad\quad
\biggl(1+\frac{(x(z_1\!-\!1)(z_2\!-\!1)\!-\!y)\,Z_i(x,y;z_1)\,Z_i(x,y;z_2)}{y\,Z_{i-1}(x,y)\,Z_i(x,y)}\biggr)\!\biggr),\\[2mm]
\notag\Z^\adj_n(x,y;z_1,z_2)&=Z_{n-1}(x,y)\,\biggl((xz_1z_2)^{n-1}+
\sum_{i=1}^{n-1}\biggl(\frac{y}{(z_1\!-\!1)(z_2\!-\!1)}\biggr)^{n-i}\times\\
\label{adjdoub2alt}&\hspace{-15mm}
\biggl(x^{i-1}(z_1z_2)^{n-1}+\frac{(z_1z_2)^{n-i-1}((z_1\!-\!1)(z_2\!-\!1)\!-\!yz_1z_2)\,Z_i(x,y;z_1)\,Z_i(x,y;z_2)}
{y\,Z_{i-1}(x,y)\,Z_i(x,y)}\biggr)\!\biggr),
\end{align}
where, in the sums over $i$, $Z_0(x,y)$ is taken to be 1.\end{corollary}
\begin{proof}
It can be checked straightforwardly (again taking $Z_0(x,y)=1$), that the initial conditions $Z^\adj_1(x,y;z_1,z_2)=\Z^\adj_1(x,y;z_1,z_2)=1$
and recursion relations~\eqref{adjdoub1}--\eqref{adjdoub1alt} are satisfied by~\eqref{adjdoub2}--\eqref{adjdoub2alt}.
\end{proof}
\begin{corollary}\label{cor6}
The two types of adjacent-boundary doubly-refined ASM generating function are related by
\begin{multline}\label{Zadjadj}
((z_1\!-\!1)(z_2\!-\!1)\!-\!yz_1z_2)\,Z^\adj_n(x,y;z_1,z_2)-(x(z_1\!-\!1)(z_2\!-\!1)\!-\!y)\,z_1z_2\,\Z^\adj_n(x,y;z_1,z_2)\\
=(z_1\!-\!1)(z_2\!-\!1)\bigl(1-(xz_1z_2)^n\bigr)\,Z_{n-1}(x,y).\end{multline}
\end{corollary}
\begin{proof}
This can be obtained directly from~\eqref{adjdoub2}--\eqref{adjdoub2alt}.
Alternatively, it can be obtained from~\eqref{adjdoub1}--\eqref{adjdoub1alt},
by showing that $((z_1\!-\!1)(z_2\!-\!1)\!-\!yz_1z_2)\,Z^\adj_n(x,y;z_1,z_2)-(z_1\!-\!1)(z_2\!-\!1)\,Z_{n-1}(x,y)$
and $(x(z_1\!-\!1)(z_2\!-\!1)\!-\!y)\,z_1z_2\,\Z^\adj_n(x,y;z_1,z_2)-(z_1\!-\!1)(z_2\!-\!1)(xz_1z_2)^n\,Z_{n-1}(x,y)$
satisfy the same recursion relation and initial condition.
\end{proof}
Note that, from~\eqref{Zsymm}, the two types of adjacent-boundary doubly-refined ASM generating function are also related by
\begin{equation}\label{Zadj}
Z^\adj_n(x,y;z_1,z_2)=x^{n(n-1)/2}\,(z_1z_2)^{n-1}\,\Z^\adj_n\bigl(\tfrac{1}{x},\tfrac{y}{x};\tfrac{1}{z_1},\tfrac{1}{z_2}\bigr).\end{equation}
Combining~\eqref{Zadjadj} and~\eqref{Zadj}, it follows that the adjacent-boundary doubly-refined ASM generating function satisfies
\begin{multline}\label{Zadjsymm}
((z_1\!-\!1)(z_2\!-\!1)\!-\!yz_1z_2)\,Z^\adj_n(x,y;z_1,z_2)=\\
(x(z_1\!-\!1)(z_2\!-\!1)\!-\!y)\,(z_1z_2)^n\,x^{n(n-1)/2}\,Z^\adj_n(\tfrac{1}{x},\tfrac{y}{x};\tfrac{1}{z_1},\tfrac{1}{z_2})\;+\\
(z_1\!-\!1)(z_2\!-\!1)\bigl(1-(xz_1z_2)^n\bigr)\,Z_{n-1}(x,y).\end{multline}

By using~\eqref{quadrel} and \eqref{Zadjadj}, together with identities from~\eqref{Zsymm}, it is also possible to obtain
a certain identity which involves $Z^\qua_n(x,y;z_1,z_2,z_3,z_4)$ and $Z^\qua_n(x,y;z_3,z_2,z_1,z_4)$
(or $Z^\qua_n(x,y;z_1,z_2,z_3,z_4)$ and $Z^\qua_n(x,y;z_1,z_4,z_3,z_2)$).
A form of this identity will be given in~\eqref{Yaddsymm}.
\begin{corollary}\label{cor7}The boundary parameter coefficients in the singly-refined ASM generating function,
as defined in~(\ref{Zcoeff}), satisfy
\begin{multline}\label{singZcoeff}
Z_n(x,y)_k=Z_{n-1}(x,y)\,\delta_{k,0}\,+\,
Z_{n-1}(x,y)\,\sum_{i=0}^{k-1}\Biggl(y^{i+1}\binom{k\!-\!1}{i}\binom{n\!-\!1}{i\!+\!1}\;+\\
y^i\sum_{j_1=0}^{k-i-1}\,\sum_{j_2=0}^{n-i-2}\frac{Z_{n-i-1}(x,y)_{j_1}\,Z_{n-i-1}(x,y)_{j_2}}{Z_{n-i-1}(x,y)\,Z_{n-i-2}(x,y)}\,
\Biggr(x\binom{k\!-\!j_1\!-\!2}{i\!-\!1}\binom{n\!-\!j_2\!-\!2}{i}\;-\\
y\binom{k\!-\!j_1\!-\!1}{i}\binom{n\!-\!j_2\!-\!1}{i\!+\!1}\Biggr)\Biggr),\end{multline}
where $Z_0(x,y)$, if it appears, is taken to be 1.
\end{corollary}
\begin{proof}Expanding the factors $1/((z_1-1)(z_2-1))^{n-i}$ in~\eqref{adjdoub2}
as binomial series, equating coefficients of $z_1^{k_1}z_2^{k_2}$ on both sides of~\eqref{adjdoub2},
and using the definitions~\eqref{Zcoeff}, gives
\begin{multline}\label{doubZcoeff}
Z^\adj_n(x,y)_{k_1,k_2}=Z_{n-1}(x,y)\,\delta_{k_1,0}\,\delta_{k_2,0}\,+\,
Z_{n-1}(x,y)\,\sum_{i=0}^{\min(k_1,k_2)-1}\Biggl(y^{i+1}\binom{k_1\!-\!1}{i}\binom{k_2\!-\!1}{i}\;+\\
y^i\sum_{j_1=0}^{k_1-i-1}\,\sum_{j_2=0}^{k_2-i-1}\frac{Z_{n-i-1}(x,y)_{j_1}\,Z_{n-i-1}(x,y)_{j_2}}{Z_{n-i-1}(x,y)\,Z_{n-i-2}(x,y)}\,
\Biggl(x\binom{k_1\!-\!j_1\!-\!2}{i\!-\!1}\binom{k_2\!-\!j_2\!-\!2}{i\!-\!1}\;-\\
y\binom{k_1\!-\!j_1\!-\!1}{i}\binom{k_2\!-\!j_2\!-\!1}{i}\Biggr)\Biggr).
\end{multline}
Summing \eqref{doubZcoeff} over $k_2$, using a standard binomial coefficient summation identity, and relabelling $k_1$ as $k$, then
gives~\eqref{singZcoeff}.
\end{proof}
\begin{corollary}\label{cor8}The unrefined ASM generating function satisfies
\begin{multline}\label{unrefZ}
Z_n(x,y)=Z_{n-1}(x,y)\Biggl(1\,+\,
\sum_{i=0}^{n-2}\Biggl(y^{i+1}\binom{n\!-\!1}{i\!+\!1}^2\;+\\
\frac{xy^i\bigl(\sum_{j=0}^{n-i-2}\binom{n-j-2}{i}Z_{n-i-1}(x,y)_j\bigr)^2-
y^{i+1}\bigl(\sum_{j=0}^{n-i-2}\binom{n-j-1}{i+1}Z_{n-i-1}(x,y)_j\bigr)^2}{Z_{n-i-1}(x,y)\,Z_{n-i-2}(x,y)}\Biggr)\Biggr),\end{multline}
where $Z_0(x,y)$ is taken to be 1.
\end{corollary}
\begin{proof}Sum~\eqref{singZcoeff} over $k$, and use a standard binomial coefficient summation identity.
\end{proof}
Note that various other expressions for $Z_n(x,y)_k$ and $Z_n(x,y)$, which provide alternatives to~\eqref{singZcoeff} and~\eqref{unrefZ},
can be obtained by taking the limits $z_1\rightarrow1$ and $z_2\rightarrow1$ of~\eqref{adjdoub2} in ways somewhat
different from those used in Corollaries~\ref{cor7} and~\ref{cor8}.

As observed in Section~\ref{new1}, the main result~\eqref{quadrel} enables all of the ASM generating functions of~\eqref{Zquad},
\eqref{Zdef} and~\eqref{Zaltdef}
to be computed recursively.  The corollaries of this section essentially comprise special cases
of~\eqref{quadrel} which apply to specific ASM
generating functions, and can alternatively be used for their computation.

For example, it can be seen that~\eqref{singZcoeff} and \eqref{unrefZ} give
$Z_n(x,y)_k$ and $Z_n(x,y)$ in terms of $Z_i(x,y)_k$ and $Z_i(x,y)$ for $i=1,\ldots,n-1$, and thereby
enable the singly-refined and unrefined ASM generating functions to be computed recursively.
The adjacent-boundary doubly-refined, opposite-boundary doubly-refined, triply-refined and
quadruply-refined ASM generating functions can then be computed using~\eqref{adjdoub2}--\eqref{adjdoub2alt},~\eqref{opprel},~\eqref{triprel}
or~\eqref{triprelalt}, and~\eqref{quadrel}, respectively.

Note that the opposite-boundary doubly-refined, singly-refined and unrefined ASM generating functions
can instead be computed using the determinant formulae~\eqref{Zdetdoub},~\eqref{Zdetsing} and~\eqref{Zdetun}.

\subsection{Corollaries for bulk parameters $x=y=1$}\label{new3}
In this section, some consequences of Theorem~\ref{thm}, for the case in which the bulk
parameters $x$ and $y$ are both set to~$1$, are derived.
In contrast to the consequences of Theorem~\ref{thm} given in Section~\ref{new2} for the case of arbitrary~$x$ and~$y$,
which were obtained essentially using only~\eqref{quadrel},
the consequences in this section are obtained using~\eqref{quadrel} together with
the nontrivial relation~\eqref{Zoppadjid} between the adjacent-boundary and opposite-boundary doubly-refined ASM generating functions
at $x=y=1$.

\begin{corollary}\label{cor9}
The alternative quadruply-refined ASM generating function at $x=y=1$ satisfies
\begin{multline}\label{unwquadalt}
(z_4z_1\!-\!z_4\!+\!1)(z_1z_2\!-\!z_1\!+\!1)(z_2z_3\!-\!z_2\!+\!1)(z_3z_4\!-\!z_3\!+\!1)\,\Z^\qua_n(1,1;z_1,z_2,z_3,z_4)=\\
\shoveleft{\frac{z_1\,z_2\,z_3\,z_4}{\A_{n-2}\,(z_1\!-\!z_3)(z_2\!-\!z_4)}\;\times}\\
\shoveright{\bigl((z_1z_2\!-\!z_1\!+\!1)(z_1z_2\!-\!z_2\!+\!1)(z_3z_4\!-\!z_3\!+\!1)(z_3z_4\!-\!z_4\!+\!1)\,
Z^\opp_{n-1}(1,1;z_4,z_1)\,Z^\opp_{n-1}(1,1;z_2,z_3)\;-}\\
\shoveright{(z_4z_1\!-\!z_4\!+\!1)(z_4z_1\!-\!z_1\!+\!1)(z_2z_3\!-\!z_2\!+\!1)(z_2z_3\!-\!z_3\!+\!1)\,
Z^\opp_{n-1}(1,1;z_1,z_2)\,Z^\opp_{n-1}(1,1;z_3,z_4)\bigr)\;+}\\
(z_2\!-\!1)(z_3\!-\!1)(z_1z_2\!-\!z_1\!+\!1)(z_3z_4\!-\!z_3\!+\!1)z_4z_1z_2^{n-1}Z^\opp_{n-1}(1,1;z_4,z_1)\;+\\
(z_3\!-\!1)(z_4\!-\!1)(z_2z_3\!-\!z_2\!+\!1)(z_4z_1\!-\!z_4\!+\!1)z_1z_2z_3^{n-1}Z^\opp_{n-1}(1,1;z_1,z_2)\;+\\
(z_4\!-\!1)(z_1\!-\!1)(z_3z_4\!-\!z_3\!+\!1)(z_1z_2\!-\!z_1\!+\!1)z_2z_3z_4^{n-1}Z^\opp_{n-1}(1,1;z_2,z_3)\;+\\
(z_1\!-\!1)(z_2\!-\!1)(z_4z_1\!-\!z_4\!+\!1)(z_2z_3\!-\!z_2\!+\!1)z_3z_4z_1^{n-1}Z^\opp_{n-1}(1,1;z_3,z_4)\;+\\
\shoveleft{\quad(z_1\!-\!1)(z_2\!-\!1)(z_3\!-\!1)(z_4\!-\!1)\bigl((z_1z_2\!-\!z_1\!+\!1)(z_3z_4\!-\!z_3\!+\!1)(z_2z_4)^{n-1}\;+}\\
(z_2z_3\!-\!z_2\!+\!1)(z_4z_1\!-\!z_4\!+\!1)(z_1z_3)^{n-1}\bigr)\A_{n-2}.\end{multline}
\end{corollary}
\begin{proof}Setting $x=y=1$ in~\eqref{altquadrel}, and using~\eqref{Zadj}, gives
\begin{multline}\label{unwquadalt1}
(z_1\!-\!z_3)(z_2\!-\!z_4)\Z^\qua_n(1,1;z_1,z_2,z_3,z_4)\,\A_{n-2}=\\
(z_1z_2\!-\!z_2\!+\!1)(z_3z_4\!-\!z_4\!+\!1)(z_2z_4)^{n-1}\,Z^\adj_{n-1}(1,1;\tfrac{1}{z_4},z_1)\,Z^\adj_{n-1}(1,1;\tfrac{1}{z_2},z_3)\;-\\
(z_4z_1\!-\!z_1\!+\!1)(z_2z_3\!-\!z_3\!+\!1)(z_1z_3)^{n-1}\,Z^\adj_{n-1}(1,1;\tfrac{1}{z_1},z_2)\,Z^\adj_{n-1}(1,1;\tfrac{1}{z_3},z_4)\;-\\
(z_2\!-\!1)(z_3\!-\!1)(z_4z_1\!-\!z_4\!+\!1)(z_2z_4)^{n-1}\,Z^\adj_{n-1}(1,1;\tfrac{1}{z_4},z_1)\,\A_{n-2}\;+\\
(z_3\!-\!1)(z_4\!-\!1)(z_1z_2\!-\!z_1\!+\!1)(z_1z_3)^{n-1}\,Z^\adj_{n-1}(1,1;\tfrac{1}{z_1},z_2)\,\A_{n-2}\;-\\
(z_4\!-\!1)(z_1\!-\!1)(z_2z_3\!-\!z_2\!+\!1)(z_2z_4)^{n-1}\,Z^\adj_{n-1}(1,1;\tfrac{1}{z_2},z_3)\,\A_{n-2}\;+\\
(z_1\!-\!1)(z_2\!-\!1)(z_3z_4\!-\!z_3\!+\!1)(z_1z_3)^{n-1}\,Z^\adj_{n-1}(1,1;\tfrac{1}{z_3},z_4)\,\A_{n-2}\;+\\
(z_1\!-\!1)(z_2\!-\!1)(z_3\!-\!1)(z_4\!-\!1)\bigl((z_2z_4)^{n-1}-(z_1z_3)^{n-1}\bigr)\,\A_{n-2}^2.\end{multline}
Multiplying both sides of~\eqref{unwquadalt1} by
$(z_4z_1\!-\!z_4\!+\!1)(z_1z_2\!-\!z_1\!+\!1)(z_2z_3\!-\!z_2\!+\!1)(z_3z_4\!-\!z_3\!+\!1)/((z_1\!-\!z_3)(z_2\!-\!z_4)\A_{n-2})$,
and using~\eqref{Zoppadjid}, in the form
$(z_1z_2\!-\!z_1\!+\!1)z_1^{n-2}Z^\adj_{n-1}(1,1;\tfrac{1}{z_1},z_2)=
z_2Z^\opp_{n-1}(1,1;z_1,z_2)+z_1^{n-2}(z_1\!-\!1)(z_2\!-\!1)\A_{n-2}$,
then gives~\eqref{unwquadalt}.
\end{proof}
\begin{corollary}\label{cor10}
The triply-refined ASM generating function at $x=y=1$ satisfies
\begin{multline}\label{unwtripalt}
(z_1z_3\!-\!z_3\!+\!1)(z_2z_3\!-\!z_2\!+\!1)z_3^{n-1}Z^\tri_n(1,1;z_1,z_2,\tfrac{1}{z_3})=\\
\shoveleft{\quad\frac{z_1\,z_3}{\A_{n-2}\,(z_1\!-\!z_2)(z_3\!-\!1)}\,\bigl((z_1z_3\!-\!z_1\!+\!1)(z_1z_3\!-\!z_3\!+\!1)z_2\,
Z_{n-1}(1,1;z_1)\,Z^\opp_{n-1}(1,1;z_2,z_3)\;-}\\
\shoveright{(z_2z_3\!-\!z_2\!+\!1)(z_2z_3\!-\!z_3\!+\!1)z_1\,Z_{n-1}(1,1;z_2)\,Z^\opp_{n-1}(1,1;z_1,z_3)\bigr)\;+}\\
(z_2\!-\!1)(z_3\!-\!1)(z_1z_3\!-\!z_3\!+\!1)z_1z_2^{n-1}Z_{n-1}(1,1;z_1)\;+\\
(z_1\!-\!1)(z_3\!-\!1)(z_2z_3\!-\!z_2\!+\!1)z_3^{n-1}Z_{n-1}(1,1;z_2).\end{multline}
\end{corollary}
\begin{proof}Set $z_2=1$ in~\eqref{unwquadalt} (and then relabel $z_3$ as $z_2$ and $z_4$ as~$z_3$).\end{proof}
It can be seen that~\eqref{unwquadalt} and~\eqref{unwtripalt} provide alternative formulae
to~\eqref{unwquad}--\eqref{unwquadopp} for the quad\-ruply- and triply-refined ASM generating functions at $x=y=1$.
In fact,~\eqref{unwquadalt} and~\eqref{unwtripalt}
differ from~\eqref{unwquadopp} and~\eqref{unwtrip}, respectively,
only in the first terms on each RHS.
It can also be seen that the last five terms on the RHS of~\eqref{unwquadalt} could be replaced by the last
five terms on the RHS of~\eqref{unwquad}.

\section{Proofs}\label{proof}
In this section, proofs of Theorem~\ref{thm}, and of several of the results of Section~\ref{prev}, are given.
Preliminary results which are needed for these proofs are obtained or stated in Sections~\ref{ASMgengf}--\ref{DesJacsect},
while the main steps of the proofs are given in Sections~\ref{pf} and~\ref{add3}--\ref{add2}.  In Section~\ref{add0},
alternative forms of certain results of Section~\ref{newres}, and some further results, are discussed.

\subsection{A generalized ASM generating function}\label{ASMgengf}
In this section, a generating function which generalizes the
quadruply-refined ASM generating function is introduced,
and some of its elementary properties are identified.  This generating function will be shown, in Section~\ref{partfunc},
to be proportional to a certain case of the partition function of the six-vertex model with DWBC.

The generalized ASM generating function involves the six statistics of~\eqref{numuA}--\eqref{rho}, together
with four further statistics associated with the entries in the corners of an ASM, and is defined,
for indeterminates $x$, $y$, $z_1$, $z_2$, $z_3$, $z_4$, $z_{41}$, $z_{12}$, $z_{23}$ and $z_{34}$, as
\begin{multline}\label{Z}Z^\gen_n(x,y;z_1,z_2,z_3,z_4;z_{41},z_{12},z_{23},z_{34})=\\[1.5mm]
\sum_{A\in\ASM(n)}x^{\nu(A)}\,y^{\mu(A)}\,z_1^{\rT(A)}\,z_2^{\rR(A)}
\,z_3^{\rB(A)}\,z_4^{\rL(A)}\,z_{41}^{\;\;1-A_{11}}\,z_{12}^{\;\;1-A_{1n}}
\,z_{23}^{\;\;1-A_{nn}}\,z_{34}^{\;\;1-A_{n1}}.\end{multline}
For example, for $n=3$, the function is
\begin{multline}Z^\gen_3(x,y;z_1,z_2,z_3,z_4;z_{41},z_{12},z_{23},z_{34})=
z_{12}\,z_{34}+
x\,z_1\,z_4\,z_{41}\,z_{12}\,z_{34}\;+\\
x\,z_2\,z_3\,z_{12}\,z_{23}\,z_{34}+
x^2\,z_1\,z_2\,z_3^2\,z_4^2\,z_{41}\,z_{12}\,z_{23}+
x^2\,z_1^2\,z_2^2\,z_3\,z_4\,z_{41}\,z_{23}\,z_{34}\;+\\
x^3\,z_1^2\,z_2^2\,z_3^2\,z_4^2\,z_{41}\,z_{23}+
x\,y\,z_1\,z_2\,z_3\,z_4\,z_{41}\,z_{12}\,z_{23}\,z_{34},\end{multline}
where the terms are written in an order which corresponds to that used in~\eqref{ASM123}.

It can be seen immediately that the quadruply-refined ASM generating function is
\begin{equation}\label{Zquagen}Z^\qua_n(x,y;z_1,z_2,z_3,z_4)=Z^\gen_n(x,y;z_1,z_2,z_3,z_4;1,1,1,1).\end{equation}

By acting on $\ASM(n)$ with transposition or anticlockwise quarter-turn rotation, and using~\eqref{TQ},
it follows that
\begin{multline}\label{Zrefrot}
Z^\gen_n(x,y;z_1,z_2,z_3,z_4;z_{41},z_{12},z_{23},z_{34})=Z^\gen_n(x,y;z_4,z_3,z_2,z_1;z_{41},z_{34},z_{23},z_{12})\\
=x^{n(n-1)/2}\,(z_1z_2z_3z_4)^{n-1}\,Z^\gen_n\bigl(\tfrac{1}{x},\tfrac{y}{x};
\tfrac{1}{z_2},\tfrac{1}{z_3},\tfrac{1}{z_4},\tfrac{1}{z_1};z_{12},z_{23},z_{34},z_{41}\bigr).
\end{multline}

The properties of ASMs with a 1 in the top-left corner, as outlined in Section~\ref{stat}, imply that
\begin{align}\notag Z^\gen_n(x,y;z_1,z_2,z_3,z_4;0,z_{12},z_{23},z_{34})&=Z^\gen_n(x,y;0,z_2,z_3,z_4;z_{41},z_{12},z_{23},z_{34})\\
\label{Z0}&=z_{12}\,z_{34}\,Z^\gen_{n-1}(x,y;1,z_2,z_3,1;1,1,z_{23},1),\end{align}
and the fact that an ASM cannot contain a 1 in both its top-left and top-right corner, implies that
\begin{equation}\label{Z00}Z^\gen_n(x,y;z_1,z_2,z_3,z_4;0,0,z_{23},z_{34})=0.\end{equation}

It can be seen from the definition~\eqref{Z} that $Z^\gen_n(x,y;z_1,z_2,z_3,z_4;z_{41},z_{12},z_{23},z_{34})$ is
linear in $z_{41}$,~$z_{12}$, $z_{23}$ and~$z_{34}$.
Therefore, using interpolation for these parameters at $0$ and $1$, it follows that
\begin{multline}\label{Zint}Z^\gen_n(x,y;z_1,z_2,z_3,z_4;z_{41},z_{12},z_{23},z_{34})=\\
\sum_{i_1,i_2,i_3,i_4\in\{0,1\}}\negmedspace
(1-z_{41})^{1-i_1}\,z_{41}^{\;\;\;i_1}\,(1-z_{12})^{1-i_2}\,z_{12}^{\;\;\;i_2}\,
(1-z_{23})^{1-i_3}\,z_{23}^{\;\;\;i_3}\,(1-z_{34})^{1-i_4}\,z_{34}^{\;\;\;i_4}\,\times\qquad\qquad\\[-3mm]
Z^\gen_n(x,y;z_1,z_2,z_3,z_4;i_1,i_2,i_3,i_3).\end{multline}

Applying cases of~\eqref{Zdef}--\eqref{Zaltdef}
and~\eqref{Zquagen}--\eqref{Z00} to~\eqref{Zint} (as a result of which, for example,~\eqref{Zrefrot} and~\eqref{Z00} imply
that nine of the sixteen terms on the RHS of~\eqref{Zint} vanish, while~\eqref{Zdef} and~\eqref{Zquagen}--\eqref{Z0} give
$Z^\gen_n(x,y;z_1,z_2,z_3,z_4;0,1,0,1)=Z_{n-2}(x,y)$), it now follows that
the generalized ASM generating function can be expressed entirely in terms of quadruply-refined, adjacent-boundary doubly-refined
and unrefined ASM generating functions as
\begin{multline}\label{Zexp}
Z^\gen_n(x,y;z_1,z_2,z_3,z_4;z_{41},z_{12},z_{23},z_{34})=
z_{41}\,z_{12}\,z_{23}\,z_{34}\,Z^\qua_n(x,y;z_1,z_2,z_3,z_4)\;+\\
z_{41}\,z_{12}\,(1\!-\!z_{23})\,z_{34}\,Z^\adj_{n-1}(x,y;z_4,z_1)\;+\\
z_{41}\,z_{12}\,z_{23}\,(1\!-\!z_{34})\,z_1z_2\,(xz_3z_4)^{n-1}\,\Z^\adj_{n-1}(x,y;z_1,z_2)\;+\\
(1\!-\!z_{41})\,z_{12}\,z_{23}\,z_{34}\,Z^\adj_{n-1}(x,y;z_2,z_3)\;+\\
z_{41}\,(1\!-\!z_{12})\,z_{23}\,z_{34}\,z_3z_4\,(xz_1z_2)^{n-1}\,\Z^\adj_{n-1}(x,y;z_3,z_4)\;+\\
(1\!-\!z_{41})\,z_{12}\,(1\!-\!z_{23})\,z_{34}\,Z_{n-2}(x,y)\;+\\
z_{41}\,(1\!-\!z_{12})\,z_{23}\,(1\!-\!z_{34})\,x^{2n-3}\,(z_1z_2z_3z_4)^{n-1}\,Z_{n-2}(x,y).\end{multline}

Some special cases of~\eqref{Zexp}, which will be used in Section~\ref{pf}, are
\begin{gather}
\notag Z^\gen_n(x,y;z_1,1,1,z_2;z_{12},1,1,1)=Z^\gen_n(x,y;1,z_1,z_2,1;1,1,z_{12},1)\\
\notag\qquad=z_{12}\,Z^\adj_n(x,y;z_1,z_2)+(1\!-\!z_{12})\,Z_{n-1}(x,y),\\
\notag Z^\gen_n(x,y;z_1,z_2,1,1;1,z_{12},1,1)=Z^\gen_n(x,y;1,1,z_1,z_2;1,1,1,z_{12})\\
\label{Zexpadj}\qquad=z_{12}\,\Z^\adj_n(x,y;z_1,z_2)+(1\!-\!z_{12})\,(xz_1z_2)^{n-1}\,Z_{n-1}(x,y).\end{gather}

\subsection{The bijection between ASMs and configurations of the six-vertex model with DWBC}\label{ASMbijsect}
In this section, the set of configurations of the six-vertex
model on an $n\times n$ grid with DWBC is described, and the
details of a natural bijection between this set and $\ASM(n)$ are summarized.  This is standard material,
with similar accounts having been given, for example, by Behrend, Di Francesco and
Zinn-Justin~\cite[Secs.~2.1 \&~3.1]{BehDifZin12},~\cite[Sec.~5.1]{BehDifZin13}.

The six-vertex, or square ice, model
is a much-studied integrable statistical mechanical model
(see, for example, Baxter~\cite[Chaps.~8~\&~9]{Bax82} for further information and references),
with DWBC for the model having been introduced and first studied by Korepin~\cite{Kor82}.
The bijection between ASMs and configurations of the model with DWBC was first
discussed by Elkies, Kuperberg, Larsen and Propp~\cite[Sec.~7]{ElkKupLarPro92b},
with the details having mostly been observed previously, but using different terminology, by Robbins and
Rumsey~\cite[pp.~179--180]{RobRum86}.
There exist closely-related bijections between ASMs and certain sets of osculating lattice paths
(see Section~\ref{bij} for references),
and between ASMs and certain fully-packed loop configurations (see, for example, Propp~\cite[Sec.~7]{Pro01}).

Let $\G$ be the $n\times n$ undirected grid with vertex set
$\{(i,j)\mid i,j=0,\ldots,n+1\}\setminus\{(0,0),(0,n+1),(n+1,0),(n+1,n+1)\}$,
where $(i,j)$ is taken to be in the $i$th row from the top and $j$th column from the left, and for which there are
horizontal edges between $(i,j)$ and $(i,j\pm1)$, and vertical edges between $(i,j)$ and $(i\pm1,j)$,
for each $i,j=1,\ldots,n$. This grid is shown in Figure~\ref{grid1}.
Each vertex or edge of~$\G$ can be described as either internal or external, i.e.,
the~$n^2$ vertices of degree~$4$ are internal,
the remaining~$4n$ vertices of degree~$1$ are external, the
$2n(n-1)$ edges which connect two internal vertices are internal,
and the remaining $4n$ edges are external.

\begin{figure}[h]\centering\psset{unit=7mm}\pspicture(-0.9,-0.1)(6.4,5.5)
\multips(0,1)(0,1){4}{\psline[linewidth=0.5pt](0,0)(5,0)}\multips(1,0)(1,0){4}{\psline[linewidth=0.5pt](0,0)(0,5)}
\multirput(1,0)(1,0){4}{$\scriptstyle\bullet$}\multirput(0,1)(1,0){6}{$\scriptstyle\bullet$}
\multirput(0,2)(1,0){6}{$\scriptstyle\bullet$}\multirput(0,3)(1,0){6}{$\scriptstyle\bullet$}
\multirput(0,4)(1,0){6}{$\scriptstyle\bullet$}\multirput(1,5)(1,0){4}{$\scriptstyle\bullet$}
\rput[b](1,5.1){$\scriptscriptstyle(0,1)$}\rput[b](4,5.1){$\scriptscriptstyle(0,n)$}
\rput[t](0.9,-0.1){$\scriptscriptstyle(n+1,1)$}\rput[t](4.1,-0.1){$\scriptscriptstyle(n+1,n)$}
\rput[r](-0.1,4){$\scriptscriptstyle(1,0)$}\rput[r](-0.1,1){$\scriptscriptstyle(n,0)$}
\rput[l](5.1,1){$\scriptscriptstyle(n,n+1)$}\rput[l](5.1,4){$\scriptscriptstyle(1,n+1)$}
\rput[br](0.99,4.09){$\scriptscriptstyle(\!1,1\!)$}\rput[bl](4.05,4.09){$\scriptscriptstyle(\!1,n\!)$}
\rput[tl](4.05,0.91){$\scriptscriptstyle(\!n,n\!)$}\rput[tr](0.99,0.91){$\scriptscriptstyle(\!n,1\!)$}
\rput(2.54,5.28){$\cdots$}\rput(-0.45,2.69){$\vdots$}\rput(2.54,-0.28){$\cdots$}\rput(5.45,2.69){$\vdots$}\endpspicture
\caption{The grid $\G$.}\label{grid1}\end{figure}
A configuration of the six-vertex model on $\G$ with DWBC is an assignment of arrows to the edges of
$\G$ such that the arrows on the external edges on the upper, right, lower and left boundaries of $\G$
are all directed upward, leftward, downward and rightward, respectively, while the arrows on the four edges incident to any
internal vertex satisfy the condition that two point towards and two point away from the vertex.

Now define $\SVDWBC(n)$ to be the set of all configurations of the six-vertex model on $\G$ with DWBC.
For example,
\psset{unit=4.53mm}
\begin{multline}\label{6VDBWC3}\SVDWBC(3)=\\
\left\{\raisebox{-7.7mm}{
\pspicture(0.4,0)(4.7,4.1)\multips(0.1,1)(0,1){3}{\psline[linewidth=0.5pt](0,0)(3.8,0)}\multips(1,0.1)(1,0){3}{\psline[linewidth=0.5pt](0,0)(0,3.8)}
\psdots[dotstyle=triangle*,dotscale=1,dotangle=0](1,3.5)(2,2.5)(2,3.5)(3,1.5)(3,2.5)(3,3.5)
\psdots[dotstyle=triangle*,dotscale=1,dotangle=180](1,0.5)(1,1.5)(1,2.5)(2,0.5)(2,1.5)(3,0.5)
\psdots[dotstyle=triangle*,dotscale=1,dotangle=90](1.5,3)(2.5,2)(2.5,3)(3.5,1)(3.5,2)(3.5,3)
\psdots[dotstyle=triangle*,dotscale=1,dotangle=-90](0.5,1)(0.5,2)(0.5,3)(1.5,1)(1.5,2)(2.5,1)
\multirput(1,0)(1,0){3}{$\scriptscriptstyle\bullet$}\multirput(0,1)(1,0){5}{$\scriptscriptstyle\bullet$}
\multirput(0,2)(1,0){5}{$\scriptscriptstyle\bullet$}\multirput(0,3)(1,0){5}{$\scriptscriptstyle\bullet$}
\multirput(1,4)(1,0){3}{$\scriptscriptstyle\bullet$}
\rput(4.2,0.7){,}\endpspicture
\pspicture(0.1,0)(4.7,4.1)\multips(0.1,1)(0,1){3}{\psline[linewidth=0.5pt](0,0)(3.8,0)}\multips(1,0.1)(1,0){3}{\psline[linewidth=0.5pt](0,0)(0,3.8)}
\psdots[dotstyle=triangle*,dotscale=1,dotangle=0](1,2.5)(1,3.5)(2,3.5)(3,1.5)(3,2.5)(3,3.5)
\psdots[dotstyle=triangle*,dotscale=1,dotangle=180](1,0.5)(1,1.5)(2,0.5)(2,1.5)(2,2.5)(3,0.5)
\psdots[dotstyle=triangle*,dotscale=1,dotangle=90](1.5,2)(2.5,2)(2.5,3)(3.5,1)(3.5,2)(3.5,3)
\psdots[dotstyle=triangle*,dotscale=1,dotangle=-90](0.5,1)(0.5,2)(0.5,3)(1.5,1)(1.5,3)(2.5,1)
\multirput(1,0)(1,0){3}{$\scriptscriptstyle\bullet$}\multirput(0,1)(1,0){5}{$\scriptscriptstyle\bullet$}
\multirput(0,2)(1,0){5}{$\scriptscriptstyle\bullet$}\multirput(0,3)(1,0){5}{$\scriptscriptstyle\bullet$}
\multirput(1,4)(1,0){3}{$\scriptscriptstyle\bullet$}
\rput(4.2,0.7){,}\endpspicture
\pspicture(0.1,0)(4.7,4.1)\multips(0.1,1)(0,1){3}{\psline[linewidth=0.5pt](0,0)(3.8,0)}\multips(1,0.1)(1,0){3}{\psline[linewidth=0.5pt](0,0)(0,3.8)}
\psdots[dotstyle=triangle*,dotscale=1,dotangle=0](1,3.5)(2,1.5)(2,2.5)(2,3.5)(3,2.5)(3,3.5)
\psdots[dotstyle=triangle*,dotscale=1,dotangle=180](1,0.5)(1,1.5)(1,2.5)(2,0.5)(3,0.5)(3,1.5)
\psdots[dotstyle=triangle*,dotscale=1,dotangle=90](1.5,3)(2.5,1)(2.5,3)(3.5,1)(3.5,2)(3.5,3)
\psdots[dotstyle=triangle*,dotscale=1,dotangle=-90](0.5,1)(0.5,2)(0.5,3)(1.5,1)(1.5,2)(2.5,2)
\multirput(1,0)(1,0){3}{$\scriptscriptstyle\bullet$}\multirput(0,1)(1,0){5}{$\scriptscriptstyle\bullet$}
\multirput(0,2)(1,0){5}{$\scriptscriptstyle\bullet$}\multirput(0,3)(1,0){5}{$\scriptscriptstyle\bullet$}
\multirput(1,4)(1,0){3}{$\scriptscriptstyle\bullet$}
\rput(4.2,0.7){,}\endpspicture
\pspicture(0.1,0)(4.7,4.1)\multips(0.1,1)(0,1){3}{\psline[linewidth=0.5pt](0,0)(3.8,0)}\multips(1,0.1)(1,0){3}{\psline[linewidth=0.5pt](0,0)(0,3.8)}
\psdots[dotstyle=triangle*,dotscale=1,dotangle=0](1,1.5)(1,2.5)(1,3.5)(2,3.5)(3,2.5)(3,3.5)
\psdots[dotstyle=triangle*,dotscale=1,dotangle=180](1,0.5)(2,0.5)(2,1.5)(2,2.5)(3,0.5)(3,1.5)
\psdots[dotstyle=triangle*,dotscale=1,dotangle=90](1.5,1)(2.5,1)(2.5,3)(3.5,1)(3.5,2)(3.5,3)
\psdots[dotstyle=triangle*,dotscale=1,dotangle=-90](0.5,1)(0.5,2)(0.5,3)(1.5,2)(1.5,3)(2.5,2)
\multirput(1,0)(1,0){3}{$\scriptscriptstyle\bullet$}\multirput(0,1)(1,0){5}{$\scriptscriptstyle\bullet$}
\multirput(0,2)(1,0){5}{$\scriptscriptstyle\bullet$}\multirput(0,3)(1,0){5}{$\scriptscriptstyle\bullet$}
\multirput(1,4)(1,0){3}{$\scriptscriptstyle\bullet$}
\rput(4.2,0.7){,}\endpspicture
\pspicture(0.1,0)(4.7,4.1)\multips(0.1,1)(0,1){3}{\psline[linewidth=0.5pt](0,0)(3.8,0)}\multips(1,0.1)(1,0){3}{\psline[linewidth=0.5pt](0,0)(0,3.8)}
\psdots[dotstyle=triangle*,dotscale=1,dotangle=0](1,2.5)(1,3.5)(2,1.5)(2,2.5)(2,3.5)(3,3.5)
\psdots[dotstyle=triangle*,dotscale=1,dotangle=180](1,0.5)(1,1.5)(2,0.5)(3,0.5)(3,1.5)(3,2.5)
\psdots[dotstyle=triangle*,dotscale=1,dotangle=90](1.5,2)(2.5,1)(2.5,2)(3.5,1)(3.5,2)(3.5,3)
\psdots[dotstyle=triangle*,dotscale=1,dotangle=-90](0.5,1)(0.5,2)(0.5,3)(1.5,1)(1.5,3)(2.5,3)
\multirput(1,0)(1,0){3}{$\scriptscriptstyle\bullet$}\multirput(0,1)(1,0){5}{$\scriptscriptstyle\bullet$}
\multirput(0,2)(1,0){5}{$\scriptscriptstyle\bullet$}\multirput(0,3)(1,0){5}{$\scriptscriptstyle\bullet$}
\multirput(1,4)(1,0){3}{$\scriptscriptstyle\bullet$}
\rput(4.2,0.7){,}\endpspicture
\pspicture(0.1,0)(4.7,4.1)\multips(0.1,1)(0,1){3}{\psline[linewidth=0.5pt](0,0)(3.8,0)}\multips(1,0.1)(1,0){3}{\psline[linewidth=0.5pt](0,0)(0,3.8)}
\psdots[dotstyle=triangle*,dotscale=1,dotangle=0](1,1.5)(1,2.5)(1,3.5)(2,2.5)(2,3.5)(3,3.5)
\psdots[dotstyle=triangle*,dotscale=1,dotangle=180](1,0.5)(2,0.5)(2,1.5)(3,0.5)(3,1.5)(3,2.5)
\psdots[dotstyle=triangle*,dotscale=1,dotangle=90](1.5,1)(2.5,1)(2.5,2)(3.5,1)(3.5,2)(3.5,3)
\psdots[dotstyle=triangle*,dotscale=1,dotangle=-90](0.5,1)(0.5,2)(0.5,3)(1.5,2)(1.5,3)(2.5,3)
\multirput(1,0)(1,0){3}{$\scriptscriptstyle\bullet$}\multirput(0,1)(1,0){5}{$\scriptscriptstyle\bullet$}
\multirput(0,2)(1,0){5}{$\scriptscriptstyle\bullet$}\multirput(0,3)(1,0){5}{$\scriptscriptstyle\bullet$}
\multirput(1,4)(1,0){3}{$\scriptscriptstyle\bullet$}
\rput(4.2,0.7){,}\endpspicture
\pspicture(0.1,0)(4,4.1)\multips(0.1,1)(0,1){3}{\psline[linewidth=0.5pt](0,0)(3.8,0)}\multips(1,0.1)(1,0){3}{\psline[linewidth=0.5pt](0,0)(0,3.8)}
\psdots[dotstyle=triangle*,dotscale=1,dotangle=0](1,2.5)(1,3.5)(2,1.5)(2,3.5)(3,2.5)(3,3.5)
\psdots[dotstyle=triangle*,dotscale=1,dotangle=180](1,0.5)(1,1.5)(2,0.5)(2,2.5)(3,0.5)(3,1.5)
\psdots[dotstyle=triangle*,dotscale=1,dotangle=90](1.5,2)(2.5,1)(2.5,3)(3.5,1)(3.5,2)(3.5,3)
\psdots[dotstyle=triangle*,dotscale=1,dotangle=-90](0.5,1)(0.5,2)(0.5,3)(1.5,1)(1.5,3)(2.5,2)
\multirput(1,0)(1,0){3}{$\scriptscriptstyle\bullet$}\multirput(0,1)(1,0){5}{$\scriptscriptstyle\bullet$}
\multirput(0,2)(1,0){5}{$\scriptscriptstyle\bullet$}\multirput(0,3)(1,0){5}{$\scriptscriptstyle\bullet$}
\multirput(1,4)(1,0){3}{$\scriptscriptstyle\bullet$}\endpspicture}
\right\}\!.\end{multline}

In an element of $\SVDWBC(n)$,
there are six possible configurations of arrows on the four edges incident to an internal vertex of $\G$.  These so-called
vertex configurations, are shown in Figure~\ref{vertices}, where the numbers 1, \ldots, 6 will be used to label the types, as indicated.

\begin{figure}[h]\centering\psset{unit=14mm}\pspicture(0.5,-0.3)(11.5,1)
\multips(1,0.5)(2,0){6}{\psline[linewidth=0.5pt](0,-0.5)(0,0.5)\psline[linewidth=0.5pt](-0.5,0)(0.5,0)}
\multirput(1,0.5)(2,0){6}{$\scriptstyle\bullet$}
\rput(1,-0.33){(1)}\psdots[dotstyle=triangle*,dotscale=1.5](1,0.75)\psdots[dotstyle=triangle*,dotscale=1.5](1,0.25)
\psdots[dotstyle=triangle*,dotscale=1.5,dotangle=-90](0.75,0.5)\psdots[dotstyle=triangle*,dotscale=1.5,dotangle=-90](1.25,0.5)
\rput(3,-0.33){(2)}\psdots[dotstyle=triangle*,dotscale=1.5,dotangle=180](3,0.75)\psdots[dotstyle=triangle*,dotscale=1.5,dotangle=180](3,0.25)
\psdots[dotstyle=triangle*,dotscale=1.5,dotangle=90](2.75,0.5)\psdots[dotstyle=triangle*,dotscale=1.5,dotangle=90](3.25,0.5)
\rput(5,-0.33){(3)}\psdots[dotstyle=triangle*,dotscale=1.5](5,0.75)\psdots[dotstyle=triangle*,dotscale=1.5](5,0.25)
\psdots[dotstyle=triangle*,dotscale=1.5,dotangle=90](4.75,0.5)\psdots[dotstyle=triangle*,dotscale=1.5,dotangle=90](5.25,0.5)
\rput(7,-0.33){(4)}\psdots[dotstyle=triangle*,dotscale=1.5,dotangle=180](7,0.75)\psdots[dotstyle=triangle*,dotscale=1.5,dotangle=180](7,0.25)
\psdots[dotstyle=triangle*,dotscale=1.5,dotangle=-90](6.75,0.5)\psdots[dotstyle=triangle*,dotscale=1.5,dotangle=-90](7.25,0.5)
\rput(9,-0.33){(5)}\psdots[dotstyle=triangle*,dotscale=1.5](9,0.75)\psdots[dotstyle=triangle*,dotscale=1.5,dotangle=180](9,0.25)
\psdots[dotstyle=triangle*,dotscale=1.5,dotangle=-90](8.75,0.5)\psdots[dotstyle=triangle*,dotscale=1.5,dotangle=90](9.25,0.5)
\rput(11,-0.33){(6)}\psdots[dotstyle=triangle*,dotscale=1.5,dotangle=180](11,0.75)\psdots[dotstyle=triangle*,dotscale=1.5](11,0.25)
\psdots[dotstyle=triangle*,dotscale=1.5,dotangle=90](10.75,0.5)\psdots[dotstyle=triangle*,dotscale=1.5,dotangle=-90](11.25,0.5)
\endpspicture\caption{The possible arrow configurations on edges incident to an internal vertex.}\label{vertices}\end{figure}
For $C\in\SVDWBC(n)$, let
$C_{ij}\in\{1,\ldots,6\}$ denote the type of vertex configuration in $C$ at internal vertex $(i,j)$ of $\G$,
and let $\N_{(k)}(C)$, $\N^{\row\,i}_{(k)}(C)$, $\N^{\col\,j}_{(k)}(C)$
and $\N^{ij}_{(k)}(C)$ denote the numbers of type-$k$ vertex configurations in~$C$
in the whole of $\G$, in row $i$ of $\G$, in column~$j$ of~$\G$,
or at vertex $(i,j)$ of $\G$, respectively.
Thus, for example, $\N^{ij}_{(k)}(C)=\delta_{C_{ij},k}$, and
$\N_{(k)}(C)=|\{(i,j)\mid 1\le i,j\le n,\;C_{ij}=k\}|=\sum_{i,j=1}^n\N^{ij}_{(k)}(C)$.

It can be shown easily that, for any $C\in\SVDWBC(n)$,
\begin{gather}\label{Nrel}\N_{(1)}(C)=\N_{(2)}(C),\quad\N_{(3)}(C)=\N_{(4)}(C),\quad
\N_{(5)}(C)=\N_{(6)}(C)+n,\\
\label{Nrelbound}\N^{\row\,1}_{(5)}(C)=\N^{\col\,n}_{(5)}(C)=
\N^{\row\,n}_{(5)}(C)=\N^{\col\,1}_{(5)}(C)=1.\end{gather}
(See, for example, Bressoud~\cite[p.~228]{Bre99},~\cite[p.~290]{Bre01} for a
discussion of the first two equations of~\eqref{Nrel}.)
It can also be seen that, for any $C\in\SVDWBC(n)$ and $1\le i,j\le n$,
\begin{equation}\label{Cbound}
C_{1j}\in\{1,3,5\},\quad C_{in}\in\{2,3,5\},\quad C_{nj}\in\{2,4,5\},\quad C_{i1}\in\{1,4,5\},\end{equation}
and therefore that
\begin{equation}\label{Ccorn}C_{11}\in\{1,5\},\quad C_{1n}\in\{3,5\},\quad C_{nn}\in\{2,5\},\quad C_{n1}\in\{4,5\}.\end{equation}

It can be shown straightforwardly that there is a natural bijection between $\ASM(n)$ and $\SVDWBC(n)$,
and that, for each $A\in\ASM(n)$ and $C\in\SVDWBC(n)$ which correspond under this bijection,
the statistics \eqref{numuA}--\eqref{rho} and~\eqref{genstat} satisfy
\begin{gather}\notag\nu(A)=\N_{(1)}(C)\;(=\N_{(2)}(C)),\quad\mu(A)=\N_{(6)}(C)\;(=\N_{(5)}(C)-n),\\
\notag\rT(A)=\N^{\row\,1}_{(1)}(C),\ \ \rR(A)=\N^{\col\,n}_{(2)}(C),\ \ \rB(A)=\N^{\row\,n}_{(2)}(C),\ \ \rL(A)=\N^{\col\,1}_{(1)}(C),\\
\notag\nu^{\row\,i}(A)=\N^{\row\,i}_{(1)}(C)+\N^{\row\,i}_{(2)}(C),\quad\nu^{\col\,j}(A)=\N^{\col\,j}_{(1)}(C)+\N^{\col\,j}_{(2)}(C),\\
\label{6VDWBCstat}\mu^{\row\,i}(A)=\N^{\row\,i}_{(6)}(C),\quad\mu^{\col\,j}(A)=\N^{\col\,j}_{(6)}(C).\end{gather}

The details of this bijection are as follows.
To map $A\in\ASM(n)$ to $C\in\SVDWBC(n)$,
for each $i=1,\ldots,n$ and $j=0,\ldots,n$,
assign a right or left arrow to the horizontal edge of~$\G$ between $(i,j)$ and $(i,j+1)$,
according to whether the partial row sum $\sum_{j'=1}^{j}A_{ij'}$ is~0 or 1, respectively.
Similarly, for each $i=0,\ldots,n$ and $j=1,\ldots,n$,
assign an upward or downward arrow to the vertical edge of~$\G$ between $(i,j)$ and $(i+1,j)$,
according to whether the partial column sum $\sum_{i'=1}^{i}A_{i'j}$ is 0 or 1, respectively.
To map $C\in\SVDWBC(n)$ to $A\in\ASM(n)$, let
\begin{equation}
A_{ij}=\begin{cases}1,&C_{ij}=5,\\
-1,&C_{ij}=6,\\
0,&C_{ij}\in\{1,2,3,4\}.\end{cases}
\end{equation}

The behaviour of the statistics~$\mu$,~$\rT$,~$\rR$,~$\rB$ and~$\rR$ under this bijection can be seen
easily. To obtain the behaviour of the statistics $\nu$, $\nu^{\row\,i}$ and $\nu^{\col\,j}$, note that,
for each $A\in\ASM(n)$ and $C\in\SVDWBC(n)$ which correspond under the bijection,
and each $1\le i,j\le n$, $\sum_{i'=1}^{i-1}\!A_{i'j}=\sum_{j'=1}^j\!A_{ij'}=0$
(or equivalently $\sum_{i'=1}^i\!A_{i'j}=\sum_{j'=1}^{j-1}\!A_{ij'}=0$)
if and only if $C_{ij}=1$, and
$\sum_{i'=1}^{i-1}\!A_{i'j}=\sum_{j'=1}^j\!A_{ij'}=1$
(or equivalently $\sum_{i'=1}^i\!A_{i'j}=\sum_{j'=1}^{j-1}\!A_{ij'}=1$)
if and only if $C_{ij}=2$.

As examples of this bijection, in~\eqref{ASM123} and~\eqref{6VDBWC3} the elements of $\ASM(3)$ and $\SVDWBC(3)$
are listed in an order for which respective elements correspond under the bijection.

\subsection{The partition function of the six-vertex model with DWBC}\label{partfunc}
In this section, the partition function of the six-vertex model with DWBC is introduced.  A relation between this
partition function, for certain assignments of its parameters, and the generalized
ASM generating function~\eqref{Z}, for certain assignments of its parameters,
is then derived using the bijection of Section~\ref{ASMbijsect}.

Let a weight $W_{(k)}(u,v)$ be associated with the vertex configuration of type~$k$, where $u$ and~$v$ are so-called spectral parameters.

The partition function for the case of the six-vertex model
of relevance here depends on these weights, and on spectral parameters~$u_i$ and~$v_j$
associated with row~$i$ and column~$j$ of~$\G$, for each $1\le i,j\le n$.  Specifically, this partition function is defined as
\begin{equation}\label{ZSVDWBC}Z^\SV_n(u_1,\ldots,u_n;v_1,\ldots,v_n)=
\sum_{C\in\SVDWBC(n)}\;\prod_{i,j=1}^n\,W_{(C_{ij})}(u_i,v_j).\end{equation}

Let the weights now satisfy
\begin{gather}\notag W_{(1)}(u,v)=W_{(2)}(u,v)=a(u,v),\quad W_{(3)}(u,v)=W_{(4)}(u,v)=b(u,v),\\
\label{Wabc}W_{(5)}(u,v)=W_{(6)}(u,v)=c(u,v),\end{gather}
for functions $a$, $b$ and $c$.

Taking the spectral parameters in~\eqref{ZSVDWBC} to be
\begin{equation}
\label{specpar}u_2=\ldots=u_{n-1}=r,\ \ v_2=\ldots=v_{n-1}=s,\ \ u_1=t_1,\ \ v_n=t_2,\ \ u_n=t_3,\ \ v_1=t_4,\end{equation}
for indeterminates $r$, $s$ and $t_1,\ldots,t_4$, and using~\eqref{Nrel}--\eqref{Ccorn}, gives
\begin{multline}\label{Z6VASM1}
Z^\SV_n(t_1,r,\ldots,r,t_3;t_4,s,\ldots,s,t_2)=
\sum_{C\in\SVDWBC(n)}\!\!a(r,s)^{2\N_{(1)}(C)}\,b(r,s)^{2\N_{(3)}(C)}\,c(r,s)^{2\N_{(6)}(C)+n}\;\times\\
\shoveright{\bigl(\tfrac{a(t_1,s)}{a(r,s)}\bigr)^{\N^{\row\,1}_{(1)}(C)-\N^{11}_{(1)}(C)}\,
\bigl(\tfrac{b(t_1,s)}{b(r,s)}\bigr)^{\N^{\row\,1}_{(3)}(C)-\N^{1n}_{(3)}(C)}\,
\bigl(\tfrac{c(t_1,s)}{c(r,s)}\bigr)^{1-\N^{11}_{(5)}(C)-\N^{1n}_{(5)}(C)}\;\times\qquad}\\
\shoveright{\bigl(\tfrac{a(r,t_2)}{a(r,s)}\bigr)^{\N^{\col\,n}_{(2)}(C)-\N^{nn}_{(2)}(C)}\,
\bigl(\tfrac{b(r,t_2)}{b(r,s)}\bigr)^{\N^{\col\,n}_{(3)}(C)-\N^{1n}_{(3)}(C)}\,
\shoveright{\bigl(\tfrac{c(r,t_2)}{c(r,s)}\bigr)^{1-\N^{1n}_{(5)}(C)-\N^{nn}_{(5)}(C)}\;\times\qquad}\\
\bigl(\tfrac{a(t_3,s)}{a(r,s)}\bigr)^{\N^{\row\,n}_{(2)}(C)-\N^{nn}_{(2)}(C)}\,
\bigl(\tfrac{b(t_3,s)}{b(r,s)}\bigr)^{\N^{\row\,n}_{(4)}(C)-\N^{n1}_{(4)}(C)}\,
\shoveright{\bigl(\tfrac{c(t_3,s)}{c(r,s)}\bigr)^{1-\N^{n1}_{(5)}(C)-\N^{nn}_{(5)}(C)}\;\times\qquad}\\
\bigl(\tfrac{a(r,t_4)}{a(r,s)}\bigr)^{\N^{\col\,1}_{(1)}(C)-\N^{11}_{(1)}(C)}\,
\bigl(\tfrac{b(r,t_4)}{b(r,s)}\bigr)^{\N^{\col\,1}_{(4)}(C)-\N^{n1}_{(4)}(C)}\,
\bigl(\tfrac{c(r,t_4)}{c(r,s)}\bigr)^{1-\N^{11}_{(5)}(C)-\N^{n1}_{(5)}(C)}\;\times\qquad}\\
\bigl(\tfrac{a(t_1,t_4)}{a(r,s)}\bigr)^{\N^{11}_{(1)}(C)}\,
\bigl(\tfrac{c(t_1,t_4)}{c(r,s)}\bigr)^{\N^{11}_{(5)}(C)}\,
\bigl(\tfrac{b(t_1,t_2)}{b(r,s)}\bigr)^{\N^{1n}_{(3)}(C)}\,
\bigl(\tfrac{c(t_1,t_2)}{c(r,s)}\bigr)^{\N^{1n}_{(5)}(C)}\;\times\\
\bigl(\tfrac{a(t_3,t_2)}{a(r,s)}\bigr)^{\N^{nn}_{(2)}(C)}\,
\bigl(\tfrac{c(t_3,t_2)}{c(r,s)}\bigr)^{\N^{nn}_{(5)}(C)}\,
\bigl(\tfrac{b(t_3,t_4)}{b(r,s)}\bigr)^{\N^{n1}_{(4)}(C)}\,
\bigl(\tfrac{c(t_3,t_4)}{c(r,s)}\bigr)^{\N^{n1}_{(5)}(C)}.
\end{multline}

It can be seen that
\begin{gather}\label{N1}\N_{(1)}(C)+\N_{(3)}(C)+\N_{(6)}(C)=\tfrac{n(n-1)}{2},\\
\notag\N^{\row\,1}_{(1)}(C)+\N^{\row\,1}_{(3)}(C)=\N^{\col\,n}_{(2)}(C)+\N^{\col\,n}_{(3)}(C)=\hspace*{67mm}\\[-0.5mm]
\label{N2}\hspace*{50mm}\N^{\row\,n}_{(2)}(C)+\N^{\row\,n}_{(4)}(C)=\N^{\col\,1}_{(1)}(C)+\N^{\col\,1}_{(4)}(C)=n-1,\\
\notag\N^{11}_{(1)}(C)+\N^{11}_{(5)}(C)=
\N^{1n}_{(3)}(C)+\N^{1n}_{(5)}(C)=\hspace*{80mm}\\[-0.8mm]
\label{N3}\hspace*{70mm}\N^{nn}_{(2)}(C)+\N^{nn}_{(5)}(C)=\N^{n1}_{(4)}(C)+\N^{n1}_{(5)}(C)=1,\end{gather}
where~\eqref{N1} follows from~\eqref{Nrel} and $\sum_{k=1}^6\N_{(k)}(C)$ $=n^2$,~\eqref{N2}
follows from~\eqref{Nrelbound}--\eqref{Cbound} and $\sum_{k=1}^6\N^{\row\,i}_{(k)}(C)=\sum_{k=1}^6\N^{\col\,j}_{(k)}(C)=n$,
and~\eqref{N3} follows from~\eqref{Ccorn}.

By using~\eqref{N1}--\eqref{N3} to eliminate $\N_{(3)}(C)$,
$\N^{\row\,1}_{(3)}(C)$, $\N^{\col\,n}_{(3)}(C)$, $\N^{\row\,n}_{(4)}(C)$,
$\N^{\col\,1}_{(4)}(C)$, $\N^{11}_{(5)}(C)$, $\N^{1n}_{(5)}(C)$,
$\N^{nn}_{(5)}(C)$ and $\N^{n1}_{(5)}(C)$ from~\eqref{Z6VASM1},
and then using the bijection of Section~\ref{ASMbijsect} between $\ASM(n)$ and~$\SVDWBC(n)$,
the behaviour~\eqref{6VDWBCstat} of the statistics \eqref{numuA}--\eqref{rho}
under this bijection, and the definition~\eqref{Z} of the generalized ASM generating function, it now follows
that~\eqref{Z6VASM1} can be written as
\begin{multline}\label{Z6VASM2}
Z^\SV_n(t_1,r,\ldots,r,t_3;t_4,s,\ldots,s,t_2)=\\
\shoveleft{\;\;b(r,s)^{n(n-1)}\,c(r,s)^n\,\bigl(\tfrac{b(t_1,s)\,b(r,t_2)\,b(t_3,s)\,b(r,t_4)}{b(r,s)^4}\bigr)^{n-1}\,
\tfrac{c(t_1,t_4)\,c(t_1,t_2)\,c(t_3,t_2)\,c(t_3,t_4)}{c(t_1,s)\,c(r,t_2)\,c(t_3,s)\,c(r,t_4)}\;\times}\\
\shoveleft{\;\;Z^\gen_n\Bigl(\!\bigl(\tfrac{a(r,s)}{b(r,s)}\bigr)^2,\bigl(\tfrac{c(r,s)}{b(r,s)}\bigr)^2;\:
\tfrac{a(t_1,s)\,b(r,s)}{a(r,s)\,b(t_1,s)},\,
\tfrac{a(r,t_2)\,b(r,s)}{a(r,s)\,b(r,t_2)},\,
\tfrac{a(t_3,s)\,b(r,s)}{a(r,s)\,b(t_3,s)},\,
\tfrac{a(r,t_4)\,b(r,s)}{a(r,s)\,b(r,t_4)};\:
\tfrac{a(r,s)\,a(t_1,t_4)\,c(t_1,s)\,c(r,t_4)}{a(t_1,s)\,a(r,t_4)\,c(r,s)\,c(t_1,t_4)},}\\
\tfrac{b(r,s)\,b(t_1,t_2)\,c(t_1,s)\,c(r,t_2)}{b(t_1,s)\,b(r,t_2)\,c(r,s)\,c(t_1,t_2)},\,
\tfrac{a(r,s)\,a(t_3,t_2)\,c(t_3,s)\,c(r,t_2)}{a(t_3,s)\,a(r,t_2)\,c(r,s)\,c(t_3,t_2)},\,
\tfrac{b(r,s)\,b(t_3,t_4)\,c(t_3,s)\,c(r,t_4)}{b(t_3,s)\,b(r,t_4)\,c(r,s)\,c(t_3,t_4)}\Bigr).\end{multline}

\subsection{The Izergin--Korepin determinant formula}\label{IK}
In this section, the Izergin--Korepin formula
for the partition function~\eqref{ZSVDWBC}, with certain assignments of the weights~\eqref{Wabc}, is stated.

It was found by Izergin~\cite[Eq.~(5)]{Ize87}, using certain results of Korepin~\cite{Kor82}, that if
the weights~\eqref{Wabc} satisfy the Yang--Baxter equation (see, for example, Baxter~\cite[pp.~187--189]{Bax82}),
then the partition function~\eqref{ZSVDWBC} is given by an explicit formula involving the determinant of a certain $n\times n$ matrix.

Let the functions $a$, $b$ and $c$ in~\eqref{Wabc} be given by
\begin{align}\notag a(u,v)&=u\,q^{1/2}-v\,q^{-1/2},\\
\notag b(u,v)&=v\,q^{1/2}-u\,q^{-1/2}=a(v,u),\\
\label{wuv}c(u,v)&=\bigl(q-q^{-1}\bigr)\,u^{1/2}\,v^{1/2},\end{align}
where $q$ is a further indeterminate, often known as the crossing parameter.
(Note that~$q$ will be present in many subsequent equations, even though it may not appear explicitly.)

It can be seen that these functions satisfy
\begin{equation}\label{int}\frac{a(u,v)^2+b(u,v)^2-c(u,v)^2}{a(u,v)\,b(u,v)}=-(q+q^{-1}).\end{equation}
The fact that the LHS of~\eqref{int} (which is often denoted as $2\Delta$) is independent of $u$ and~$v$ implies that
the six-vertex model weights given by~\eqref{Wabc} and~\eqref{wuv} satisfy the Yang--Baxter equation
(see, for example, Baxter~\cite[Eq.~(9.6.14)]{Bax82}).

The resulting Izergin--Korepin determinant formula is then given by the following result.
\begin{theorem*}[Izergin]
The partition function~\eqref{ZSVDWBC}, with weights given by~\eqref{Wabc} and~\eqref{wuv}, satisfies
\begin{equation}\label{Idet1}
Z^\SV_n(u_1,\ldots,u_n;v_1,\ldots,v_n)=
\frac{\prod_{i,j=1}^na(u_i,v_j)\,b(u_i,v_j)}{\prod_{1\le i<j\le n}(u_i-u_j)(v_j-v_i)}
\det_{1\le i,j\le n}\left(\frac{c(u_i,v_j)}{a(u_i,v_j)\,b(u_i,v_j)}\right).\end{equation}
\end{theorem*}
This theorem can be proved by showing that each side of~\eqref{Idet1} satisfies, and is uniquely determined by, four particular properties.
Specifically, each side is symmetric in $u_1,\ldots,u_n$ and in $v_1,\ldots,v_n$ (which can be obtained for the LHS using the Yang--Baxter
equation, and is immediate for the RHS), is a polynomial of degree $2n-1$ in each of $u_1^{1/2},\ldots,u_n^{1/2},v_1^{1/2},\ldots,v_n^{1/2}$,
satisfies the same recursion relation for cases in which $u_i=q^{\pm1}v_j$ for some $i$ and $j$, and satisfies the
same initial condition at $n=1$.  For further details of this proof, see, for example,
Izergin, Coker and Korepin~\cite[Sec.~5]{IzeCokKor92},
Fonseca and Zinn-Justin~\cite[Sec.~B.1]{FonZin08},
Kuperberg~\cite[Sec.~1]{Kup96}, or
Zinn-Justin~\cite[Secs.~2.5.2--2.5.3]{Zin09}. For an alternative proof, see Bogoliubov, Pronko and Zvonarev~\cite[Sec.~4]{BogProZvo02}.

Note that, although the determinant and the denominator of the prefactor on the RHS of~\eqref{Idet1}
both vanish if $u_i=u_j$ or $v_i=v_j$ for some $i\ne j$, the RHS has a well-defined limit in these cases,
as a polynomial in $u_1^{1/2},\ldots,u_n^{1/2},v_1^{1/2},\ldots,v_n^{1/2}$.  Accordingly, it will be
valid to use~\eqref{Idet1} to derive properties of the partition function for the assignments~\eqref{specpar},
as will be done in Section~\ref{pf}.

\subsection{The Desnanot--Jacobi identity}\label{DesJacsect}
In this section, the Desnanot--Jacobi determinant identity is stated, and discussed briefly.

For a square matrix $M$, let $M_\mathrm{TL}$, $M_\mathrm{TR}$, $M_\mathrm{BR}$, $M_\mathrm{BL}$ and $M_\mathrm{C}$
be the submatrices of $M$ defined in Section~\ref{prev3}.
The Desnanot--Jacobi identity, in the form which will be used here, can be stated as follows.
\begin{theorem*}[Desnanot, Jacobi]
For any square matrix $M$,
\begin{equation}\label{DesJac}\det M\,\det M_\mathrm{C}=\det M_\mathrm{TL}\,\det M_\mathrm{BR}-\det M_\mathrm{TR}\,\det M_\mathrm{BL}.\end{equation}
\end{theorem*}
For proofs of~\eqref{DesJac} using various methods, see, for example,
Bressoud~\cite[Sec.~3.5]{Bre99}, Fulmek~\cite[Sec.~5.1]{Ful12}, Fulmek and Kleber~\cite{FulKle01}, or
Zeilberger~\cite{Zei97}.  
Cases of~\eqref{DesJac} for $n\times n$ matrices $M$ with small values of~$n$ were published by Desnanot in~1819
(see Muir~\cite[Eqs.~(A)--(G), (A$'$)--(G$'$), pp.~139--142]{Mui06}).  The further attribution to Jacobi is based on the fact
that, for arbitrary~$n$,~\eqref{DesJac} corresponds to the case $m=2$ of the identity,
published by Jacobi in~1834 (see Muir~\cite[Eq.~(XX.~4), p.~208]{Mui06}), that for any $n\times n$ matrix~$M$ and any $m\le n$, each
$m\times m$ minor of the matrix of $(n-1)\times(n-1)$ minors of~$M$ equals the complementary minor of~$M$ multiplied by $(\det M)^{m-1}$.
For proofs of the Jacobi identity using
various methods, see, for example, Brualdi and Schneider~\cite[Sec.~4]{BruSch83}, Knuth~\cite[Eq.~(3.16)]{Knu96},
Leclerc~\cite[Sec.~3.2]{Lec93}, Muir~\cite[Sec.~175]{Mui60}, or Turnbull~\cite[pp.~77--79]{Tur60}.
Some closely related identities will be given in~\eqref{PolPartBazin} and~\eqref{PolBazin}.

For a discussion of alternative forms of~\eqref{DesJac}, and of further determinant identities of which~\eqref{DesJac} is a special case,
see, for example, Behrend, Di Francesco and Zinn-Justin~\cite[Sec.~4]{BehDifZin13}.

\subsection{\protect Proof of Theorem~\ref{thm}}\label{pf}
In this section, the main steps in the proof of Theorem~\ref{thm} are given.  These involve using
the relation~\eqref{Z6VASM2} between the generalized ASM generating function and
the partition function of the six-vertex model with DWBC,
the Izergin--Korepin formula~\eqref{Idet1}, and
the Desnanot--Jacobi identity~\eqref{DesJac}.

By applying the Desnanot--Jacobi identity~\eqref{DesJac} to the matrix
$\Bigl(\frac{c(u_i,v_j)}{a(u_i,v_j)\,b(u_i,v_j)}\Bigr)_{1\le i,j\le n}$,
and then applying the Izergin--Korepin formula~\eqref{Idet1} to each of the six determinants which appear,
it follows that the partition function~\eqref{ZSVDWBC}, with weights given by~\eqref{Wabc} and~\eqref{wuv}, satisfies
\begin{multline}\label{DJIK}
(u_1-u_n)\,(v_n-v_1)\,Z^\SV_n(u_1,\ldots,u_n;v_1,\ldots,v_n)\,Z^\SV_{n-2}(u_2,\ldots,u_{n-1};v_2,\ldots,v_{n-1})=\\
a(u_1,v_n)\,b(u_1,v_n)\,a(u_n,v_1)\,b(u_n,v_1)\,
Z^\SV_{n-1}(u_1,\ldots,u_{n-1};v_1,\ldots,v_{n-1})\;\times\\
\shoveright{Z^\SV_{n-1}(u_2,\ldots,u_n;v_2,\ldots,v_n)}\\
-\,a(u_1,v_1)\,b(u_1,v_1)\,a(u_n,v_n)\,b(u_n,v_n)\,
Z^\SV_{n-1}(u_1,\ldots,u_{n-1};v_2,\ldots,v_n)\;\times\\
Z^\SV_{n-1}(u_2,\ldots,u_n;v_1,\ldots,v_{n-1}),\end{multline}
for any $u_1,\ldots,u_n,v_1,\ldots,v_n$.
In fact, it can be seen, from the previous derivation, that~\eqref{DJIK} is satisfied by any function which has the form of~\eqref{Idet1}, for
arbitrary functions $a$, $b$ and $c$.

Certain forms of the Desnanot--Jacobi identity have previously been combined
with certain cases of the Izergin--Korepin formula
by Korepin and Zinn-Justin~\cite[Sec.~3]{KorZin00} (see also Sogo~\cite[Sec.~4]{Sog93}),
and by Behrend, Di Francesco and Zinn-Justin~\cite[Sec.~5.4]{BehDifZin13}.

Now let parameters $x$, $y$, $z_1$, $z_2$, $z_3$, $z_4$, $z_{41}$, $z_{12}$, $z_{23}$, $z_{34}$, $z_{41}'$, $z_{12}'$, $z_{23}'$ and $z_{34}'$ be
given in terms of parameters~$q$,~$r$,~$s$,~$t_1$,~$t_2$,~$t_3$ and $t_4$ by
\begin{alignat}{4}\notag&&
x&=\bigl(\tfrac{a(r,s)}{b(r,s)}\bigr)^2,&\quad y&=\bigl(\tfrac{c(r,s)}{b(r,s)}\bigr)^2,\\
\notag z_1&=\tfrac{a(t_1,s)\,b(r,s)}{a(r,s)\,b(t_1,s)},&\quad
z_2&=\tfrac{a(r,t_2)\,b(r,s)}{a(r,s)\,b(r,t_2)},&\quad
z_3&=\tfrac{a(t_3,s)\,b(r,s)}{a(r,s)\,b(t_3,s)},&\quad
z_4&=\tfrac{a(r,t_4)\,b(r,s)}{a(r,s)\,b(r,t_4)},\\
\notag z_{41}&=\tfrac{a(r,s)\,a(t_1,t_4)}{a(t_1,s)\,a(r,t_4)},&\quad
z_{12}&=\tfrac{b(r,s)\,b(t_1,t_2)}{b(t_1,s)\,b(r,t_2)},&\quad
z_{23}&=\tfrac{a(r,s)\,a(t_3,t_2)}{a(t_3,s)\,a(r,t_2)},&\quad
z_{34}&=\tfrac{b(r,s)\,b(t_3,t_4)}{b(t_3,s)\,b(r,t_4)},\\
z_{41}'&=\tfrac{b(r,s)\,b(t_1,t_4)}{b(t_1,s)\,b(r,t_4)},&\quad
z_{12}'&=\tfrac{a(r,s)\,a(t_1,t_2)}{a(t_1,s)\,a(r,t_2)},&\quad
z_{23}'&=\tfrac{b(r,s)\,b(t_3,t_2)}{b(t_3,s)\,b(r,t_2)},&\quad
\label{zcorn}z_{34}'&=\tfrac{a(r,s)\,a(t_3,t_4)}{a(t_3,s)\,a(r,t_4)},\end{alignat}
where the functions $a$, $b$ and $c$ are given by~\eqref{wuv}.

The parameterizations~\eqref{zcorn} allow
arbitrary $x$, $y$ and $z_1,\ldots,z_4$ to be expressed in terms
of~$q$,~$r$,~$s$ and~$t_1,\ldots,t_4$.  In particular, it follows, using~\eqref{wuv},~\eqref{int} and~\eqref{zcorn},
that $q$ can first be taken to be a solution of
\begin{equation}x^{1/2}\,q^2+(x-y+1)\,q+x^{1/2}=0,\end{equation}
that $r$ and $s$ can then be taken to satisfy
\begin{equation}(x^{1/2}+q)\,r=(x^{1/2}\,q+1)\,s,\end{equation}
and that $t_1,\ldots,t_4$ are then given by
\begin{equation}
t_1=\tfrac{s(x^{1/2}\,z_1\,q+1)}{x^{1/2}\,z_1+q},\quad
t_2=\tfrac{r(x^{1/2}\,z_2+q)}{x^{1/2}\,z_2\,q+1},\quad
t_3=\tfrac{s(x^{1/2}\,z_3\,q+1)}{x^{1/2}\,z_3+q},\quad
t_4=\tfrac{r(x^{1/2}\,z_4+q)}{x^{1/2}\,z_4\,q+1}.\end{equation}

It can be checked easily that the parameterizations~\eqref{zcorn} imply that
$z_{41}$, $z_{12}$, $z_{23}$, $z_{34}$, $z_{41}'$, $z_{12}'$, $z_{23}'$ and $z_{34}'$
can be expressed in terms of $x$, $y$ and $z_1,\ldots,z_4$ as
\begin{equation}\label{Zcorn1}z_{41}=1\!-\!\tfrac{(z_4-1)(z_1-1)}{y\,z_4\,z_1},\ z_{12}=1\!-\!\tfrac{x(z_1-1)(z_2-1)}{y},\
z_{23}=1\!-\!\tfrac{(z_2-1)(z_3-1)}{y\,z_2\,z_3},\ z_{34}=1\!-\!\tfrac{x(z_3-1)(z_4-1)}{y},\end{equation}
and
\begin{equation}
\label{Zcorn2}z_{41}'=1\!-\!\tfrac{x(z_4-1)(z_1-1)}{y},\ z_{12}'=1\!-\!\tfrac{(z_1-1)(z_2-1)}{y\,z_1\,z_2},\
z_{23}'=1\!-\!\tfrac{x(z_2-1)(z_3-1)}{y},\ z_{34}'=1\!-\!\tfrac{(z_3-1)(z_4-1)}{y\,z_3\,z_4}.\end{equation}
Thus, each $z_{ij}$ and $z'_{ij}$ has the form $f(y,z_i,z_j)=1-\tfrac{(z_i-1)(z_j-1)}{y\,z_i\,z_j}$
or $f(\frac{y}{x},\frac{1}{z_i},\frac{1}{z_j})$.

The parameterizations~\eqref{zcorn} also give
\begin{equation}\label{tz}
(t_2-t_4)\,(t_1-t_3)=\tfrac{a(r,s)^2\,b(t_1,s)\,b(r,t_2)\,b(t_3,s)\,b(r,t_4)}{(b(r,s)\,c(r,s))^2}\,(z_4-z_2)\,(z_1-z_3).\end{equation}

Assigning the parameters in~\eqref{DJIK} according to~\eqref{specpar},
applying the relation~\eqref{Z6VASM2} to each of the six
cases of the partition function in~\eqref{DJIK}, and using~\eqref{zcorn} and~\eqref{tz}, it now follows that
\begin{multline}\label{ZZquad}
(z_4\!-\!z_2)\,(z_1\!-\!z_3)\,Z^\gen_n(x,y;z_1,z_2,z_3,z_4;z_{41},z_{12},z_{23},z_{34})\,Z^\gen_{n-2}(x,y;1,1,1,1;1,1,1,1)=\\
y\,z_1\,z_2\,z_3\,z_4\,\bigl(z_{12}\,z_{12}'\,z_{34}\,z_{34}'\,Z^\gen_{n-1}(x,y;z_1,1,1,z_4;z_{41},1,1,1)\;\times\\
\shoveright{Z^\gen_{n-1}(x,y;1,z_2,z_3,1;1,1,z_{23},1)}\\
-\,z_{41}\,z_{41}'\,z_{23}\,z_{23}'\,Z^\gen_{n-1}(x,y;z_1,z_2,1,1;1,z_{12},1,1)\;\times\\
Z^\gen_{n-1}(x,y;1,1,z_3,z_4;1,1,1,z_{34})\bigr),\end{multline}
where $x$, $y$ and $z_1,\ldots,z_4$ are arbitrary, while the remaining parameters are given by~\eqref{Zcorn1}--\eqref{Zcorn2}.

Finally, the required identity~\eqref{quadrel} follows
from~\eqref{ZZquad} by applying~\eqref{Zexp}, including the special cases~\eqref{Zexpadj}, to each of the
six cases of generalized ASM generating functions, and
using~\eqref{Zcorn1}--\eqref{Zcorn2} to eliminate $z_{41}$, $z_{12}$, $z_{23}$, $z_{34}$,~$z_{41}'$,~$z_{12}'$,~$z_{23}'$ and $z_{34}'$.

\subsection{\protect Alternative statement of some results of Sections~\ref{new1} and~\ref{new2}}\label{add0}
In this section, it is shown that some of the results of Sections~\ref{new1} and~\ref{new2} can be stated,
slightly more compactly, in terms of a certain function which will be defined
in~\eqref{Y} (and which has essentially already appeared in~\eqref{ZZquad}).
Also, an additional symmetry property for ASM generating functions is obtained,
and further quadratic relations satisfied by these functions are discussed briefly.

Define a function
\begin{multline}\label{Y}Y_n(x,y;z_1,z_2,z_3,z_4)=\sum_{A\in\ASM(n)}\!x^{\nu(A)}\,y^{\mu(A)+A_{11}+A_{1n}+A_{nn}+A_{n1}}\;\times\\
\qquad\quad z_1^{\rT(A)+A_{11}-1}\,z_2^{\rR(A)+A_{nn}-1}z_3^{\rB(A)+A_{nn}-1}\,z_4^{\rL(A)+A_{11}-1}\;\times\\
\qquad\qquad\qquad\quad(yz_4z_1\!-\!(z_4\!-\!1)(z_1\!-\!1))^{1-A_{11}}\,(y\!-\!x(z_1\!-\!1)(z_2\!-\!1))^{1-A_{1n}}\;\times\\
(yz_2z_3\!-\!(z_2\!-\!1)(z_3\!-\!1))^{1-A_{nn}}\,(y\!-\!x(z_3\!-\!1)(z_4\!-\!1))^{1-A_{n1}}.\end{multline}

It follows that $Y_n(x,y;z_1,z_2,z_3,z_4)$ is a polynomial in $x$, $y$ and $z_1,\ldots,z_4$, with integer coefficients.

Also define
\begin{equation}\label{Yadj}Y_n(x,y;z_1,z_2)=Y_n(x,y;z_1,z_2,1,1).\end{equation}

It can be seen, using the definition~\eqref{Z} of the generalized ASM generating function, that
\begin{align}\notag Y_n(x,y;z_1,z_2,z_3,z_4)&=
y^4Z^\gen_n\bigl(x,y;z_1,z_2,z_3,z_4;\\
\label{YZ}&\quad\,1\!-\!\tfrac{(z_4-1)(z_1-1)}{y\,z_4\,z_1},
1\!-\!\tfrac{x(z_1-1)(z_2-1)}{y},
1\!-\!\tfrac{(z_2-1)(z_3-1)}{y\,z_2\,z_3},
1\!-\!\tfrac{x(z_3-1)(z_4-1)}{y}\bigr),\end{align}
i.e., $Y_n(x,y;z_1,z_2,z_3,z_4)=y^4Z^\gen_n(x,y;z_1,z_2,z_3,z_4;z_{41},z_{12},z_{23},z_{34})$, with
$z_{41}$, $z_{12}$, $z_{23}$ and~$z_{34}$ given by~\eqref{Zcorn1}.

Therefore, by applying~\eqref{Zexp} to the RHS of~\eqref{YZ},
$Y_n(x,y;z_1,z_2,z_3,z_4)$ can be expressed in terms of quadruply-refined, adjacent-boundary doubly-refined
and unrefined ASM generating functions.

It can be seen immediately from~\eqref{Y}--\eqref{Yadj} that
\begin{align}\notag Y_n(x,y;z_1,1,z_2,1)&=y^4\,Z^\opp_n(x,y;z_1,z_2),\\
\notag Y_n(x,y;z,1)&=y^4\,Z_n(x,y;z),\\
\label{YZopp}Y_n(x,y;1,1)&=y^4\,Z_n(x,y),\end{align}
and by acting on $\ASM(n)$ in~\eqref{Y} with transposition or anticlockwise quarter-turn rotation and using~\eqref{TQ}
(or by applying~\eqref{Zrefrot} to~\eqref{YZ}), it follows that
\begin{align}\notag Y_n(x,y;z_1,z_2,z_3,z_4)&=Y_n(x,y;z_4,z_3,z_2,z_1)\\
&=x^{n(n-1)/2+4}\,(z_1z_2z_3z_4)^{n-1}\,
\label{Ysymm}Y_n\bigl(\tfrac{1}{x},\tfrac{y}{x};\tfrac{1}{z_2},\tfrac{1}{z_3},\tfrac{1}{z_4},\tfrac{1}{z_1}\bigr).
\end{align}

It can also be shown, using~\eqref{Y} and the properties of ASMs
in which a 1 on a boundary is in a corner or separated from a corner by a single zero
(as discussed in Section~\ref{stat}), that
\begin{equation}\label{Y0}Y_n(x,y;z_1,z_2,z_3,0)=
\bigl(1\!+\!\tfrac{x(z_1-1)}{y}\bigr)\,\bigl(1\!+\!\tfrac{x(z_3-1)}{y}\bigr)\,Y_{n-1}(x,y;z_1,z_2,z_3,1).\end{equation}

Taking $x$, $y$, $z_1,\ldots,z_4$, $z_{41}$, $z_{12}$, $z_{23}$ and $z_{34}$ to be parameterized by~\eqref{zcorn},
and using~\eqref{Z6VASM2},~\eqref{Zcorn1} and \eqref{YZ}, gives
\begin{multline}\label{YZ6V}b(r,s)^{(n-1)(n-4)+8}\,c(r,s)^{n-8}\,\bigl(b(t_1,s)\,b(r,t_2)\,b(t_3,s)\,b(r,t_4)\bigr)^{n-1}\,
t_1^{1/2}\,t_2^{1/2}\,t_3^{1/2}\,t_4^{1/2}\,(rs)^{-1}\;\times\\
\qquad Y_n\Bigl(\!\bigl(\tfrac{a(r,s)}{b(r,s)}\bigr)^2,\bigl(\tfrac{c(r,s)}{b(r,s)}\bigr)^2;\,
\tfrac{a(t_1,s)\,b(r,s)}{a(r,s)\,b(t_1,s)},
\tfrac{a(r,t_2)\,b(r,s)}{a(r,s)\,b(r,t_2)},
\tfrac{a(t_3,s)\,b(r,s)}{a(r,s)\,b(t_3,s)},
\tfrac{a(r,t_4)\,b(r,s)}{a(r,s)\,b(r,t_4)}\Bigr)=\\
Z^\SV_n(t_1,r,\ldots,r,t_3;t_4,s,\ldots,s,t_2),\end{multline}
where the weights in the partition function are given by~\eqref{Wabc},
and the functions $a$, $b$ and $c$ are given by~\eqref{wuv}.

It now follows from the symmetry of the partition function on the RHS of~\eqref{YZ6V}
in $t_1$ and~$t_3$, and in $t_2$ and $t_4$, as discussed after~\eqref{Idet1}, that
\begin{equation}\label{Yaddsymm}Y_n(x,y;z_1,z_2,z_3,z_4)=Y_n(x,y;z_3,z_2,z_1,z_4)=Y_n(x,y;z_1,z_4,z_3,z_2).\end{equation}
In contrast to~\eqref{Ysymm}, there does not seem to be a simple combinatorial derivation of~\eqref{Yaddsymm}.
(Note that the first equality of~\eqref{Ysymm} together with either equality of~\eqref{Yaddsymm} gives
the other equality of~\eqref{Yaddsymm}.)

It can be seen that~\eqref{Yaddsymm}, together with~\eqref{Zexp}, provides an identity (referred to
after~\eqref{Zadjsymm}) involving
$Z^\qua_n(x,y;z_1,z_2,z_3,z_4)$ and $Z^\qua_n(x,y;z_3,z_2,z_1,z_4)$ (or
$Z^\qua_n(x,y;z_1,z_2,z_3,z_4)$ and $Z^\qua_n(x,y;z_1,z_4,z_3,z_2)$).

It follows from~\eqref{ZZquad} and~\eqref{YZ}, together with \eqref{Zcorn1}--\eqref{Zcorn2},~\eqref{Yadj}
and~\eqref{YZopp}--\eqref{Ysymm}, that
\begin{multline}\label{Yquad}
y^7(z_4\!-\!z_2)\,(z_1\!-\!z_3)\,Y_n(x,y;z_1,z_2,z_3,z_4)\,Z_{n-2}(x,y)=\\
\bigl(x(z_1\!-\!1)(z_2\!-\!1)\!-\!y\bigr)
\bigl((z_1\!-\!1)(z_2\!-\!1)\!-\!yz_1z_2\bigr)
\bigl(x(z_3\!-\!1)(z_4\!-\!1)\!-\!y\bigr)\,\times\qquad\qquad\qquad\qquad\\
\shoveright{\bigl((z_3\!-\!1)(z_4\!-\!1)\!-\!yz_3z_4\bigr)\,Y_{n-1}(x,y;z_1,z_4)\,Y_{n-1}(x,y;z_2,z_3)\;-}\\
\bigl(x(z_4\!-\!1)(z_1\!-\!1)\!-\!y\bigr)
\bigl((z_4\!-\!1)(z_1\!-\!1)\!-\!yz_4z_1\bigr)
\bigl(x(z_2\!-\!1)(z_3\!-\!1)\!-\!y\bigr)\,\times\qquad\qquad\qquad\qquad\\
\bigl((z_2\!-\!1)(z_3\!-\!1)\!-\!yz_2z_3\bigr)\,Y_{n-1}(x,y;z_1,z_2)\,Y_{n-1}(x,y;z_3,z_4).\end{multline}
This can be regarded as a restatement of~\eqref{quadrel}, where~\eqref{Zexp} and~\eqref{YZ}
enable the conversion between~\eqref{quadrel} and~\eqref{Yquad}.

Setting $z_2=1$ in~\eqref{Yquad}, and relabelling $z_3$ as $z_2$ and $z_4$ as~$z_3$, gives
\begin{multline}\label{Ytrip}
y(z_1\!-\!z_2)\,(z_3\!-\!1)\,Y_n(x,y;z_1,1,z_2,z_3)\,Z_{n-2}(x,y)=\\
\shoveleft{\;\;
\bigl(x(z_2\!-\!1)(z_3\!-\!1)\!-\!y\bigr)\bigl((z_2\!-\!1)(z_3\!-\!1)\!-\!yz_2z_3\bigr)z_1\,Y_{n-1}(x,y;z_1,z_3)\,Z_{n-1}(x,y;z_2)\;-}\\
\bigl(x(z_3\!-\!1)(z_1\!-\!1)\!-\!y\bigr)\bigl((z_3\!-\!1)(z_1\!-\!1)\!-\!yz_3z_1\bigr)z_2\,Y_{n-1}(x,y;z_2,z_3)\,Z_{n-1}(x,y;z_1),\end{multline}
which can be regarded as a restatement of~\eqref{triprel}.

Setting $z_3=0$ in~\eqref{Yquad}, using~\eqref{Y0}, replacing $n$ by $n+1$, and relabelling
$z_3$ as $z_2$ and $z_4$ as~$z_3$ gives
\begin{multline}
y^2(z_1\!-\!z_2)\,z_3\,Y_n(x,y;z_1,1,z_2,z_3)\,Z_{n-1}(x,y)=\\
(z_1\!-\!1)\bigl(x(z_2\!-\!1)(z_3\!-\!1)\!-\!y\bigr)\bigl((z_2\!-\!1)(z_3\!-\!1)\!-\!yz_2z_3\bigr)\,Y_n(x,y;z_1,z_3)\,Z_{n-1}(x,y;z_2)\;-\qquad\\
(z_2\!-\!1)\bigl(x(z_1\!-\!1)(z_3\!-\!1)\!-\!y\bigr)\bigl((z_1\!-\!1)(z_3\!-\!1)\!-\!yz_1z_3\bigr)\,Y_n(x,y;z_2,z_3)\,Z_{n-1}(x,y;z_1),
\end{multline}
which can be regarded as a restatement of~\eqref{triprelalt}.

It can be seen that various other identities from Section~\ref{newres} involving quadruply-refined, triply-refined or adjacent-boundary doubly-refined
ASM generating functions could similarly be restated in terms of the functions~\eqref{Y}--\eqref{Yadj}.

Finally, note that some further quadratic identities satisfied by ASM generating functions can be obtained
from~\eqref{Yquad}, or by combining certain variations of the Desnanot--Jacobi identity with the Izergin--Korepin formula,
for certain assignments of the spectral parameters, and then using~\eqref{YZ6V}.

An example of such an identity is
\begin{multline}
(z_1\!-\!z_2)\,(z_3\!-\!z_4)\,Y_n(x,y;z_1,w_1,z_2,w_2)\,Y_n(x,y;z_3,w_1,z_4,w_2)\;-\\
(z_1\!-\!z_3)\,(z_2\!-\!z_4)\,Y_n(x,y;z_1,w_1,z_3,w_2)\,Y_n(x,y;z_2,w_1,z_4,w_2)\;+\quad\\
(z_1\!-\!z_4)\,(z_2\!-\!z_3)\,Y_n(x,y;z_1,w_1,z_4,w_2)\,Y_n(x,y;z_2,w_1,z_3,w_2)=0.
\end{multline}
This reduces to~\eqref{propeq2} for $w_1=w_2=1$, and can be obtained by applying~\eqref{Yquad} to
each of the six cases of $(z_i-z_j)Y_n(x,y;z_i,w_1,z_j,w_2)$ on the LHS,
and then checking that the overall expression on the LHS vanishes.

\subsection{\protect Derivations of \eqref{ff1} and \eqref{ff2}}\label{add3}
In this section, the explicit expressions~\eqref{ff1} and~\eqref{ff2} for the ASM generating functions
at $y=x+1$ are derived using a method different from that used in Section~\ref{prev3},
which instead involves the six-vertex model with DWBC.

Let the crossing parameter $q$, for the weights~\eqref{wuv}, be given by $q=\pm i$,
this corresponding to the so-called free fermion case of the six-vertex model.
It can be shown that, for this assignment of $q$, the partition function~\eqref{ZSVDWBC},
with weights given by~\eqref{Wabc} and~\eqref{wuv}, is explicitly
\begin{equation}\label{Zff}Z^\SV_n(u_1,\ldots,u_n;v_1,\ldots,v_n)\big|_{q=\pm i}=
\textstyle (\pm2)^n\,i^{n^2}\,\prod_{i=1}^nu_i^{1/2}v_i^{1/2}\;\prod_{1\le i<j\le n}(u_i+u_j)(v_i+v_j).\end{equation}
This result can be obtained by combining the Izergin--Korepin formula~\eqref{Idet1}
(in which $q=\pm i$ gives $a(u,v)b(u,v)=-(u^2+v^2)$) with
the Cauchy double alternant evaluation
\begin{equation}\label{Cau}\det_{1\le i,j\le n}\left(\frac{1}{\alpha_i+\beta_j}\right)=
\frac{\prod_{1\le i<j\le n}(\alpha_i-\alpha_j)(\beta_i-\beta_j)}{\prod_{1\le i,j\le n}(\alpha_i+\beta_j)}.\end{equation}

For previous appearances of~\eqref{Zff}, see, for example, Bogoliubov, Pronko and Zvonarev~\cite[Eq.~(58)]{BogProZvo02},
or Okada~\cite[Thm.~2.4(1), third eq.]{Oka06}, and for information regarding~\eqref{Cau}, see, for example,
Muir~\cite[p.~345]{Mui06}, \cite[Sec.~353]{Mui60}.

Assigning the spectral parameters in~\eqref{Zff} according to~\eqref{specpar},
applying the relation~\eqref{Z6VASM2}, and using the first ten parameterizations of~\eqref{zcorn},
which enable arbitrary $x$ and $z_1,\ldots,z_4$ to be obtained, with $y$, $z_{41}$, $z_{12}$, $z_{23}$ and $z_{34}$ then
being given by $y=x+1$ and~\eqref{Zcorn1}, it follows that
\begin{multline}\label{Zff1}
Z^\gen_n\bigl(x,x+1;z_1,z_2,z_3,z_4;
1\!-\!\tfrac{(z_4-1)(z_1-1)}{(x+1)\,z_4\,z_1},
1\!-\!\tfrac{x(z_1-1)(z_2-1)}{x+1},
1\!-\!\tfrac{(z_2-1)(z_3-1)}{(x+1)\,z_2\,z_3},
1\!-\!\tfrac{x(z_3-1)(z_4-1)}{x+1}\bigr)\\
=(xz_1z_3\!+\!1)(xz_2z_4\!+\!1)\bigl((xz_1\!+\!1)(xz_2\!+\!1)(xz_3\!+\!1)(xz_4\!+\!1)\bigr)^{n-2}\;\times\\
(x\!+\!1)^{(n-2)(n-3)/2-2n+3}.\end{multline}
(Using~\eqref{YZ}, the function in~\eqref{Zff1} could be written as $(x+1)^{-4}Y_n(x,x+1;z_1,z_2,z_3,z_4)$.)

The required expressions~\eqref{ff1}--\eqref{ff2} now follow from~\eqref{Zff1} by applying~\eqref{Zexp}--\eqref{Zexpadj},
and, in some cases, setting boundary parameters to~1.

\subsection{\protect Derivation of \eqref{ZmultASM}}\label{add1}
In this section, a derivation is given of the identity~\eqref{ZmultASM} satisfied by the
generating function~\eqref{ZrowASM} associated with several rows (or several columns) of ASMs.

For $0\le m\le n$ and $1\le k_1<\ldots<k_m\le n$, and indeterminates $r$,~$s$ and $t_1,\ldots,t_m$,
consider the partition
function $Z^\SV_n(r,\ldots,r,t_1,r,\ldots\ldots,r,t_m,r,\ldots,r;s,\ldots\ldots,s)$,
as given by~\eqref{ZSVDWBC}--\eqref{Wabc}, with arbitrary functions $a$, $b$ and $c$,
where
$t_i$ appears in position $k_i$ within $r,\ldots,r,t_1,r,\ldots\ldots,r,t_m,r,\ldots,r$.
Using the bijection of Section~\ref{ASMbijsect} between $\ASM(n)$ and $\SVDWBC(n)$,
and the behaviour~\eqref{6VDWBCstat} of the relevant statistics under this bijection, it can be checked
(by applying a process similar to that used in Section~\ref{ASMbijsect} for the derivation of~\eqref{Z6VASM2})
that this partition function is related to the ASM generating function of~\eqref{ZrowASM} by
\begin{multline}\label{6VmultASMmult}
Z^\SV_n(r,\ldots,r,t_1,r,\ldots\ldots,r,t_m,r,\ldots,r;s,\ldots\ldots,s)=\\
\shoveleft{\textstyle\qquad b(r,s)^{(n-1)(n-m)}\,c(r,s)^{n-m}\,\prod_{i=1}^mb(t_i,s)^{n-1}\,c(t_i,s)\;\times}\\
Z^{k_1,\ldots,k_m}_n\Bigl(\!\bigl(\tfrac{a(r,s)}{b(r,s)}\bigr)^2,\bigl(\tfrac{c(r,s)}{b(r,s)}\bigr)^2;\:
\tfrac{a(t_1,s)\,b(r,s)}{a(r,s)\,b(t_1,s)},\ldots,\tfrac{a(t_m,s)\,b(r,s)}{a(r,s)\,b(t_m,s)};\:
\bigl(\tfrac{c(t_1,s)}{b(t_1,s)}\bigr)^2,\ldots,\bigl(\tfrac{c(t_m,s)}{b(t_m,s)}\bigr)^2\Bigr).\end{multline}
(Note that the case of~\eqref{Z6VASM2} obtained by setting $t_2=t_4=s$, and then relabelling $t_3$ as $t_2$,
matches the case $m=2$, $k_1=1$ and $k_2=n$ of~\eqref{6VmultASMmult}.)

It follows from a result of Colomo and Pronko~\cite[Eq.~(6.8)]{ColPro08},~\cite[Eq.~(A.13)]{ColPro10}
(with a special case stated previously by Colomo and Pronko~\cite[Eq.~(5.8)]{ColPro06}) that
if the functions~$a$,~$b$ and~$c$ are given by~\eqref{wuv}, then, for $1\le m\le n$,
\begin{multline}\label{6Vmult}
\textstyle(a(r,s)\,b(r,s))^{m(m-1)/2}\:\prod_{1\le i<j\le m}\bigl(t_i-t_j\bigr)\;\times\\
\shoveleft{\;\;\textstyle
Z^\SV_n(r,\ldots,r,t_1,r,\ldots\ldots,r,t_m,r,\ldots,r;s,\ldots\ldots,s)\:\prod_{i=1}^{m-1}Z^\SV_{n-i}(r,\ldots,r;s,\ldots,s)}\\
=\;\det_{1\le i,j\le m}
\Bigl((a(t_i,s)\,b(t_i,s))^{j-1}\,(t_i-r)^{m-j}
\,Z^\SV_{n-j+1}(t_i,r,\ldots,r;s,\ldots\ldots,s)\Bigr).\end{multline}
Due to the symmetry of the partition function in each set of spectral parameters
(as discussed after~\eqref{Idet1}), the positions of $t_1,\ldots,t_m$
within $r,\ldots,r,t_1,r,\ldots\ldots,r,t_m,r,\ldots,r$ on the LHS,
and within $t_i,r,\ldots,r$ on the RHS, are immaterial in~\eqref{6Vmult}.
A result related to~\eqref{6Vmult}, in the case $q=e^{\pm 2\pi i/3}$
(where $q$ is the crossing parameter which appears in~\eqref{wuv}),
has been obtained by Ayyer and Romik~\cite[Thm.~6 \& App.~A]{AyyRom13}.

An alternative derivation of~\eqref{6Vmult} (in which a general identity for minors of a matrix is
combined with the Izergin--Korepin formula~\eqref{Idet1},
analogously to the derivation of~\eqref{DJIK}) will now be given, before continuing with the main
derivation of~\eqref{ZmultASM}.

It can be established that, for any $1\le m\le n$, any $(n+m-1)\times n$ matrix $M$,
any $1\le k_1<\ldots<k_m<l_1<\ldots<l_{m-1}\le n+m-1$, and any $1\le p_1<\ldots<p_{m-1}\le n$,
\begin{equation}\label{PolPartBazin}\textstyle
\det M^{l_1,\ldots,l_{m-1}}\;\prod_{i=1}^{m-1}\det M^{k_1,\ldots,k_m,l_1,\ldots,l_{i-1}}_{p_1,\ldots,p_i}=\displaystyle
\det_{1\le i,j\le m}\Bigl(\det M^{k_1,\ldots,k_{i-1},k_{i+1},\ldots,k_m,l_1,\ldots,l_{j-1}}_{p_1,\ldots,p_{j-1}}\Bigr),\end{equation}
where $M^{i_1,\ldots,i_r}_{j_1,\ldots,j_c}$ denotes the submatrix of $M$ in which rows $i_1,\ldots,i_r$ and columns
$j_1,\ldots,j_c$ have been deleted.  The identity~\eqref{PolPartBazin} is closely related to various identities for compound determinants,
as outlined, for example, by Leclerc~\cite[Sec.~3]{Lec93}.  In particular, using a certain polarization of Bazin's theorem,
as given by Leclerc~\cite[Prop.~3.4]{Lec93}, it follows that,
for any $1\le m\le n$, any $(n+2m-2)\times n$ matrix $N$, and any
$1\le k_1<\ldots<k_m<l_1<\ldots<l_{m-1}<q_1<\ldots<q_{m-1}\le n+2m-2$,
\begin{multline}\label{PolBazin}\textstyle
\det N^{l_1,\ldots,l_{m-1},q_1,\ldots,q_{m-1}}\;\prod_{i=1}^{m-1}\det N^{k_1,\ldots,k_m,l_1,\ldots,l_{i-1},q_{i+1},\ldots,q_{m-1}}=\\
\det_{1\le i,j\le m}\Bigl(\det N^{k_1,\ldots,k_{i-1},k_{i+1},\ldots,k_m,l_1,\ldots,l_{j-1},q_j,\ldots,q_{m-1}}\Bigr),\end{multline}
where the same notation as previously is used for the deletion of rows from a matrix.
The identity~\eqref{PolPartBazin} can be obtained from~\eqref{PolBazin} by appending $m-1$
rows to the bottom of an $(n+m-1)\times n$ matrix $M$ to form a matrix $N$,
where the $i$th appended row (i.e., row $n+m-1+i$ of $N$) contains a~$1$ in column $p_i$ and $0$'s elsewhere,
and then applying~\eqref{PolBazin} to~$N$, choosing $q_i=n+m-1+i$.

By applying~\eqref{PolPartBazin} to the matrix
$\Bigl(\frac{c(u_i,v_j)}{a(u_i,v_j)\,b(u_i,v_j)}\Bigr)_{1\le i\le n+m-1;\;1\le j\le n}$,
and then applying the Izergin--Korepin formula~\eqref{Idet1} to each of the minors which appear,
it follows that, for any $1\le m\le n$, any $1\le k_1<\ldots<k_m<l_1<\ldots<l_{m-1}\le n+m-1$, any $1\le p_1<\ldots<p_{m-1}\le n$,
and indeterminates $u_1,\ldots,u_{n+m-1}$ and $v_1,\ldots,v_n$,
the partition function~\eqref{ZSVDWBC}, with weights given by~\eqref{Wabc} and~\eqref{wuv}, satisfies
\begin{multline}\label{ZPolBazin}
\textstyle\prod_{1\le i\le j\le m-1}a(u_{l_j},v_{p_i})\,b(u_{l_j},v_{p_i})\;\prod_{1\le i<j\le m}\bigl(u_{k_i}-u_{k_j}\bigr)\;\times\\
\textstyle Z^\SV_n(u^{l_1,\ldots,l_{m-1}};v_1,\ldots,v_n)\:
\prod_{i=1}^{m-1}Z^\SV_{n-i}(u^{k_1,\ldots,k_m,l_1,\ldots,l_{i-1}};v^{p_1,\ldots,p_i})=\\
\det_{1\le i,j\le m}\Bigl(\textstyle\prod_{j'=1}^{j-1}a(u_{k_i},v_{p_{j'}})\,b(u_{k_i},v_{p_{j'}})\;
\prod_{j'=j}^{m-1}\bigl(u_{k_i}-u_{l_{j'}}\bigr)\;\times\\[-2.5mm]
Z^\SV_{n-j+1}(u^{k_1,\ldots,k_{i-1},k_{i+1},\ldots,k_m,l_1,\ldots,l_{j-1}};v^{p_1,\ldots,p_{j-1}})\Bigr),
\end{multline}
where $u^{i_1,\ldots,i_r}$ and $v^{j_1,\ldots,j_c}$ denote the subsequences of $u_1,\ldots,u_{n+m-1}$
and $v_1,\ldots,v_n$ in which $u_{i_1},\ldots,u_{i_r}$ and $v_{j_1},\ldots,v_{j_c}$ have been deleted.
The identity~\eqref{6Vmult} can now be obtained from~\eqref{ZPolBazin} by
considering $1\le k_1<\ldots<k_m\le n$, choosing arbitrary $1\le p_1<\ldots<p_{m-1}\le n$,
and setting~$l_i=n+i$, $u_i=r$ for $i\notin\{k_1,\ldots,k_m\}$, $u_{k_i}=t_i$, and $v_j=s$
(and also using the symmetry of the partition function to reassign positions of $t_i$ in the
partition functions in the determinant on the RHS).

It can be seen from the previous derivation that~\eqref{6Vmult} and~\eqref{ZPolBazin} are, in fact,
satisfied by any function which has the form of~\eqref{Idet1}, for
arbitrary functions $a$, $b$ and $c$.

The main derivation of~\eqref{ZmultASM} will now be completed.  Using~\eqref{6VmultASMmult} to express
each case of a partition function in~\eqref{6Vmult} as a case of the ASM generating function~\eqref{ZrowASM},
noting (as already done after~\eqref{mult1}) that
$Z^{k_1,\ldots,k_m}_n(x,y;z_1,\ldots,z_m;w_1,\ldots,w_m)$ is $Z_n(x,y)$ for $m=0$, or
$Z_n(x,y;z_1)$ for $m=1$ and $k_1=1$,
and then applying the definition~\eqref{Zmult}, gives
\begin{multline}\label{Zmultparam}Z^{k_1,\ldots,k_m}_n\Bigl(\!\bigl(\tfrac{a(r,s)}{b(r,s)}\bigr)^2,\bigl(\tfrac{c(r,s)}{b(r,s)}\bigr)^2;\:
\tfrac{a(t_1,s)\,b(r,s)}{a(r,s)\,b(t_1,s)},\ldots,\tfrac{a(t_m,s)\,b(r,s)}{a(r,s)\,b(t_m,s)};\:
\bigl(\tfrac{c(t_1,s)}{b(t_1,s)}\bigr)^2,\ldots,\bigl(\tfrac{c(t_m,s)}{b(t_m,s)}\bigr)^2\Bigr)=\\
X_n\Bigl(\!\bigl(\tfrac{a(r,s)}{b(r,s)}\bigr)^2,\bigl(\tfrac{c(r,s)}{b(r,s)}\bigr)^2;\:
\tfrac{a(t_1,s)\,b(r,s)}{a(r,s)\,b(t_1,s)},\ldots,\tfrac{a(t_m,s)\,b(r,s)}{a(r,s)\,b(t_m,s)}\Bigr).\end{multline}

The required result~\eqref{ZmultASM} now follows immediately from~\eqref{Zmultparam} by
parameterizing $x$, $y$, $z_1,\ldots,z_m$ and $w_1,\ldots,w_m$ in terms of $q$, $r$, $s$ and $t_1,\ldots,t_m$ as
$x=\bigl(\tfrac{a(r,s)}{b(r,s)}\bigr)^2$, $y=\bigl(\tfrac{c(r,s)}{b(r,s)}\bigr)^2$,
$z_i=\tfrac{a(t_i,s)\,b(r,s)}{a(r,s)\,b(t_i,s)}$ and $w_i=\bigl(\tfrac{c(t_i,s)}{b(t_i,s)}\bigr)^2$, and
observing that this parameterization enables arbitrary $x$, $y$ and $z_1,\ldots,z_m$ to be obtained,
with $w_1,\ldots,w_m$ then being given by $w_i=xz_i^2+(y-x-1)z_i+1$.

\subsection{\protect Derivations of \eqref{unwquad} and \eqref{schurgen}}\label{add2}
In this section, an expression for the partition function of the six-vertex model with DWBC at a certain value
of its crossing parameter is given.  This result, together with certain other previously-stated results, is then used
to obtain derivations of the identity~\eqref{unwquad} satisfied by the alternative quadruply-refined ASM generating function at $x=y=1$,
and of the identity~\eqref{schurgen} satisfied by the function~\eqref{Zmult} at $x=y=1$.

Let the crossing parameter $q$, for the weights~\eqref{wuv}, be given by $q=e^{\pm 2\pi i/3}$,
this corresponding to the so-called combinatorial point of the six-vertex model.
It can be shown that, for this assignment of $q$, the partition function~\eqref{ZSVDWBC},
with weights given by~\eqref{Wabc} and~\eqref{wuv}, can be expressed as
\begin{multline}\label{Zcomb}Z^\SV_n(u_1,\ldots,u_n;u_{n+1},\ldots,u_{2n})\big|_{q=e^{\pm 2\pi i/3}}=\\
(\pm1)^n\,i^{n^2}\,3^{n/2}\,u_1^{1/2}\!\ldots u_{2n}^{1/2}\;s_{(n-1,n-1,\ldots,2,2,1,1)}(u_1,\ldots,u_{2n}),\end{multline}
where $s_{(n-1,n-1,\ldots,2,2,1,1)}(u_1,\ldots,u_{2n})$ is the Schur function indexed by the double-staircase
partition $(n-1,n-1,\ldots,2,2,1,1)$, evaluated at the spectral parameters $u_1,\ldots,u_{2n}$.

An important consequence of~\eqref{Zcomb} is that $Z^\SV_n(u_1,\ldots,u_n;u_{n+1},\ldots,u_{2n})|_{q=e^{\pm 2\pi i/3}}$ is
symmetric in all spectral parameters $u_1,\ldots,u_{2n}$.

The result~\eqref{Zcomb} was first obtained by Okada~\cite[Thm.~2.4(1), second equation]{Oka06}, and Stroganov~\cite[Eq.~(17)]{Str06}.
For further related information, derivations and results, see, for example, Aval~\cite{Ava09},
Fonseca and Zinn-Justin~\cite[Sec.~B.2]{FonZin08},
Lascoux~\cite[p.~4]{Las07b},
Razumov and Stroganov~\cite[Sec.~2]{RazStr04b},
or Zinn-Justin~\cite[Sec.~2.5.6]{Zin09}.

Proceeding to the derivation of~\eqref{unwquad}, using the symmetry of the LHS of~\eqref{Zcomb} in all of its spectral parameters, it follows that
\begin{multline}\label{unwsym}Z^\SV_n(t_1,r,\ldots,r,t_3;t_4,r,\ldots,r,t_2)\big|_{q=e^{\pm 2\pi i/3}}=\\
Z^\SV_n(t_1,t_2,t_3,t_4,r,\ldots,r;r,\ldots\ldots,r)\big|_{q=e^{\pm 2\pi i/3}}.\end{multline}
Applying~\eqref{Z6VASM2} and~\eqref{6VmultASMmult} (with $r=s=1$, $m=4$ and $k_i=i$) to the LHS and RHS, respectively, of~\eqref{unwsym},
and then using \eqref{wuv} (with $q=e^{\pm 2\pi i/3}$) for the functions $a$, $b$ and $c$ (and noting that $a(1,1)=b(1,1)$
and $b(1,1)^2=c(1,1)^2$), gives
\begin{multline}\label{Z1234t}\bigl(\tfrac{b(1,t_2)\,b(1,t_4)}{b(t_2,1)\,b(t_4,1)}\bigr)^{n-1}
Z^\gen_n\Bigl(1,1;\,\tfrac{a(t_1,1)}{b(t_1,1)},\tfrac{a(1,t_2)}{b(1,t_2)},\tfrac{a(t_3,1)}{b(t_3,1)},\tfrac{a(1,t_4)}{b(1,t_4)};\\[-1.5mm]
\shoveright{\tfrac{a(1,1)\,a(t_1,t_4)}{a(t_1,1)\,a(1,t_4)},
\tfrac{b(1,1)\,b(t_1,t_2)}{b(t_1,1)\,b(1,t_2)},
\tfrac{a(1,1)\,a(t_3,t_2)}{a(t_3,1)\,a(1,t_2)},
\tfrac{b(1,1)\,b(t_3,t_4)}{b(t_3,1)\,b(1,t_4)}\Bigr)\Big|_{q=e^{\pm 2\pi i/3}}=}\\[2mm]
Z^{1,2,3,4}_n\Bigl(1,1;\,
\tfrac{a(t_1,1)}{b(t_1,1)},\tfrac{a(t_2,1)}{b(t_2,1)},\tfrac{a(t_3,1)}{b(t_3,1)},\tfrac{a(t_4,1)}{b(t_4,1)};\,
\bigl(\tfrac{c(t_1,1)}{b(t_1,1)}\bigr)^2\!,\bigl(\tfrac{c(t_2,1)}{b(t_2,1)}\bigr)^2\!,
\bigl(\tfrac{c(t_3,1)}{b(t_3,1)}\bigr)^2\!,\bigl(\tfrac{c(t_4,1)}{b(t_4,1)}\bigr)^2\Bigr)\Big|_{q=e^{\pm 2\pi i/3}}.\end{multline}
It follows from~\eqref{Z1234t}, by applying~\eqref{Zmultparam} to the RHS, and then setting
$t_i=\tfrac{qz_i+1}{z_i+q}$, or equivalently $z_i=\tfrac{qt_i-1}{q-t_i}=\tfrac{a(t_i,1)}{b(t_i,1)}$,
for each $1\le i\le4$, that
\begin{multline}\label{Z1234z}(z_2z_4)^{n-1}
Z^\gen_n\Bigl(1,1;z_1,\tfrac{1}{z_2},z_3,\tfrac{1}{z_4};
1\!+\!\tfrac{(z_4-1)(z_1-1)}{z_1},1\!+\!\tfrac{(z_1-1)(z_2-1)}{z_2},1\!+\!\tfrac{(z_2-1)(z_3-1)}{z_3},1\!+\!\tfrac{(z_3-1)(z_4-1)}{z_4}\Bigr)\\
=X_n(1,1;z_1,z_2,z_3,z_4).\end{multline}
The required result~\eqref{unwquad} now follows from~\eqref{Z1234z} by applying~\eqref{Zexp} to the LHS, and using the
first definition of~\eqref{Zaltdef}, the definition~\eqref{Zmult} and the identity~\eqref{Zadj}.

Proceeding to the derivation of~\eqref{schurgen}, setting $r=s=1$, $q=e^{\pm 2\pi i/3}$ and $k_i=i$ in~\eqref{6Vmult}, and then
applying~\eqref{Zcomb} to the partition function
$Z^\SV_n(t_1,\ldots,t_m,1,\ldots,1;1,\ldots,1)$, applying~\eqref{6VmultASMmult} (with $m=0$ or $m=1$) to the remaining cases of partition functions,
and using~\eqref{Zmult} and~\eqref{wuv}, gives
\begin{multline}\label{schurder}s_{(n-1,n-1,\ldots,2,2,1,1)}(t_1,\ldots,t_m,\,\underbrace{1,\ldots,1}_{2n-m}\,)=\\
3^{n(n-1)/2}\,b(1,1)^{-m(n-1)}\,(b(t_1,1)\ldots b(t_m,1))^{n-1}\,
X_n\bigl(1,1;\:\tfrac{a(t_1,1)}{b(t_1,1)},\ldots,\tfrac{a(t_m,1)}{b(t_m,1)}\bigr)\big|_{q=e^{\pm 2\pi i/3}}.\end{multline}
(Alternatively, this can be obtained by combining~\eqref{wuv},~\eqref{6VmultASMmult},~\eqref{Zmultparam} and~\eqref{Zcomb}.)

The required result~\eqref{schurgen} now follows from~\eqref{schurder}
by setting $t_i=\tfrac{qz_i+1}{z_i+q}$, or equivalently $z_i=\tfrac{qt_i-1}{q-t_i}=\tfrac{a(t_i,1)}{b(t_i,1)}$,
for each $1\le i\le m$.

\let\oldurl\url
\makeatletter
\renewcommand*\url{%
        \begingroup
        \let\do\@makeother
        \dospecials
        \catcode`{1
        \catcode`}2
        \catcode`\ 10
        \url@aux
}
\newcommand*\url@aux[1]{%
        \setbox0\hbox{\oldurl{#1}}%
        \ifdim\wd0>\linewidth
                \strut
                \\
                \vbox{%
                        \hsize=\linewidth
                        \kern-\lineskip
                        \raggedright
                        \strut\oldurl{#1}%
                }%
        \else
                \hskip0pt plus\linewidth
                \penalty0
                \box0
        \fi
        \endgroup
}
\makeatother
% to fix MR issues
\gdef\MRshorten#1 #2MRend{#1}%
\gdef\MRfirsttwo#1#2{\if#1M%
MR\else MR#1#2\fi}
\def\MRfix#1{\MRshorten\MRfirsttwo#1 MRend}
\renewcommand\MR[1]{\relax\ifhmode\unskip\spacefactor3000 \space\fi
  \MRhref{\MRfix{#1}}{{\tiny \MRfix{#1}}}}
\renewcommand{\MRhref}[2]{%
 \href{http://www.ams.org/mathscinet-getitem?mr=#1}{#2}}
\bibliography{Bibliography}
\bibliographystyle{amsplainhyper}
\end{document}